\documentclass[a4paper,11pt]{conm-p-l}
\usepackage{amssymb,amsxtra}
\usepackage{hyperref, pinlabel}
\usepackage[cmtip,all]{xy}

\newtheorem{theorem}{Theorem}[section]
\newtheorem{lemma}[theorem]{Lemma}
\newtheorem{proposition}[theorem]{Proposition}
\newtheorem{corollary}[theorem]{Corollary}

\theoremstyle{definition}
\newtheorem{definition}[theorem]{Definition}
\newtheorem{problem}[theorem]{Problem}

\theoremstyle{remark}
\newtheorem{remark}[theorem]{Remark}
\newtheorem{example}[theorem]{Example}
\newtheorem{notation}[theorem]{Notation}

\numberwithin{equation}{section}

\DeclareMathOperator{\trop}{\mathrm{Trop}}
\DeclareMathOperator{\ptrop}{Trop_{>0}}
\DeclareMathOperator{\nntrop}{Trop_{\geq 0}}
\DeclareMathOperator{\Hom}{Hom}
\DeclareMathOperator{\im}{Im}

\DeclareMathOperator{\supp}{Supp}
\DeclareMathOperator{\spec}{Spec}
\DeclareMathOperator{\sat}{Sat}
\DeclareMathOperator{\Star}{Star}
\DeclareMathOperator{\init}{in}
\DeclareMathOperator{\Newton}{Newton}
\newcommand{\torus}{\mathbb{T}}
\newcommand{\N}{\mathbb{N}}
\newcommand{\Z}{\mathbb{Z}}
\newcommand{\Q}{\mathbb{Q}}
\newcommand{\R}{\mathbb{R}}
\newcommand{\C}{\mathbb{C}}
\newcommand{\V}{\mathcal{V}}
\newcommand{\W}{\mathcal{W}}
\newcommand{\ri}{\overline{\R}}
\newcommand{\zcal}{\mathcal{Z}}
\newcommand{\scal}{\mathcal{S}}

\begin{document}
\title{Local tropicalization}

\author{Patrick Popescu-Pampu}
   \address{Universit{\'e} Lille 1, UFR de Maths., B\^atiment M2\\
     Cit\'e Scientifique, 59655, Villeneuve d'Ascq Cedex, France.}
   \email{patrick.popescu@math.univ-lille1.fr}
 \author{Dmitry Stepanov}
   \address{Bauman Moscow State Technical University, Dep. of
     mathematical modeling \\
     ul. 2-ya Baumanskaya 5, Moscow 105005, Russia}
   \email{dstepanov@bmstu.ru}
\thanks{The second author was supported by the Russian Grants for Leading Scientific 
Schools no. 5139.2012.1, and RFBR grant no. 11-01-00336-a.}

\date{19th of December 2012}


\begin{abstract}
In this paper we propose a general functorial definition of the 
operation of \emph{local tropicalization} in commutative algebra. 
Let $R$ be a commutative ring, $\Gamma$ a finitely generated subsemigroup of a 
lattice, $\gamma : \Gamma \rightarrow R/ R^*$ a morphism of semigroups, and
$\V(R)$ the topological space of valuations on $R$ taking values in $\R \cup \infty$.
Then we may \emph{tropicalize} with respect to $\gamma$ 
any subset $\W$ of the space of valuations $\V(R)$. By definition, we get a subset
of a rational polyhedral 
cone canonically associated to $\Gamma$, enriched with strata at infinity. 
In particular, when $R$ is a local ring, $\gamma$ is a \emph{local} morphism of 
semigroups, and $\W$ is the space of valuations which are either positive or 
non-negative on $R$, we call these processes \emph{local tropicalizations}. They 
depend only on the ambient toroidal structure, which in turn allows to define 
tropicalizations of subvarieties of toroidal embeddings. We prove that with suitable 
hypothesis, these 
local tropicalizations are the supports of finite rational polyhedral fans enriched 
with strata at infinity and 
we compare the global and local tropicalizations of a subvariety of a 
toric variety. 
\end{abstract}

\maketitle

\tableofcontents

\section{Introduction}

Let $V$ be a subvariety of a torus $(K^*)^n$ over an algebraically closed field 
$K$ endowed with a non-trivial valuation $v: K \rightarrow \R \cup \{\infty\}$. 
Denote by $I_V$ the ideal of $V$ in the ring $K[X_1^{\pm 1} ,\dotsc, X_n^{\pm 1}]$ 
of Laurent polynomials. Denote by 
$x_1 ,\dotsc, x_n$ the images in the ring $K[V]$ 
of the canonical coordinate functions $X_1 ,\dotsc, X_n$ on $(K^*)^n$. 

If $\R^n$ is the real vector space generated by the lattice $\Z^n$ of 
$1$-parameter subgroups of the torus $(K^*)^n$, we may also think about a vector 
$w \in \R^n$ as a weight of the variables $X_1 ,\dotsc, X_n$. Then, by 
definition, the $w$-initial ideal $\init_w(I_V)$ is generated by all $w$-initial 
forms of elements of $I_V$ (see also Section~\ref{S:powerseriesring}).

One associates canonically to $V \subset (K^*)^n$ a polyhedral set (that is, a set
that may be represented as a finite union of convex polyhedra) in $\R^n$. 
This set is called \emph{the tropicalization} of $X$. It can be defined in at least 
three different but equivalent ways, expressed as conditions (1), (2) and (3) in 
the following theorem (see \cite{Sp 02}, \cite{SS 04}, \cite{EKL 06}, \cite{P 07},
\cite{D 08}):

\begin{theorem}\label{T:3defs}
  The following subsets of $\R^n$ coincide (the horizontal bar meaning the 
  closure with respect to the usual topology of $\R^n$):
  \begin{enumerate}
    \item \label{pointsvp}
          $\overline{\{(v(s_1) ,\dotsc, v(s_n))\: | \: (s_1 ,\dotsc, s_n) \in V \} }$.

    \item \label{invp}
          $\{ w \in \R^n \: | \:  \mbox{ the ideal } \init_w(I_V) \mbox{ is 
           monomial free } \}$.

    \item \label{valvp} 
          $\overline{\{ (W(x_1)  ,\dotsc,  W(x_n) ) \: | \: W \mbox{ is a valuation of }
            K[V] \mbox{ extending } v\} }$. 
  \end{enumerate}
\end{theorem}

Our initial aim was to define \emph{a local analog of tropicalization},
adapted to the study of singularities and of their deformations. More precisely, 
we wanted to tropicalize ideals of formal power series rings of the form 
$K[[X_1 ,\dotsc, X_n]]$, where $K$ is any field, and to compare them with the 
previous (global) tropicalizations. 

As subvarieties of tori are most naturally studied by taking their closures in 
associated toric varieties, we wanted to be able to define, 
more generally, tropicalizations of ideals in formal completions of the coordinate 
rings of affine toric varieties at their closed orbits, that is, in rings of the 
form $K[[\Gamma]]$, where $\Gamma$ is a (not necessarily saturated) finitely 
generated subsemigroup of a lattice. In the sequel, following \cite{CLS 11}, we 
call such semigroups \emph{affine}. 
In order to get more geometric flexibility (see Remark \ref{reasnn}), 
we consider not necessarily normal toric 
varieties, that is, not necessarily saturated semigroups. 

In order to compare local and global tropicalizations, we have to change the ring 
defining the object under study. That is why we need to develop a sufficiently 
general functorial framework for tropicalization. Among the characterizations 1--3 
in the previous definition of (global) tropicalization, it is the third one which 
lends itself most easily to such a functorial treatment. This is not surprising
since the set described by (\ref{valvp}) is an image of Berkovich's analytification
of $V$, see \cite{B 90}. Therefore, we propose the following general framework for 
tropicalization (both local and global): 

\begin{itemize}
  \item  Start from a semigroup morphism 
     $(\Gamma, +) \stackrel{\gamma}{\longrightarrow} (R, \cdot)$ from an affine 
     semigroup to the multiplicative semigroup of an arbitrary commutative ring. 
     
   \item  Consider the space $\V(R)$ of valuations of the ring $R$ with values in 
         $\overline{\R}= \R \cup \{\infty\}$.  
         
    \item Consider the tautological map: 
    $$ \begin{array}{ccc}
             \V(R) & \stackrel{\gamma^*}{\longrightarrow} & 
                   L(\Gamma) : = \Hom(\Gamma, \overline{\R}) \\
             \nu &  \longmapsto   &   \nu \circ \gamma . 
          \end{array}.  $$
          
  \item If $\W$ is any subset of $\V(R)$, its \emph{tropicalization} is defined as 
     the image $\gamma^*(\W)$. 
\end{itemize}

This construction is a functor from the category of pairs $(\gamma, \W)$ and 
commutative diagrams of morphisms between such pairs to that of maps 
$\W \longrightarrow L(\Gamma)$ and commutative diagrams between them. 

We speak about \emph{local tropicalization} when $(R,\mathfrak{m})$ is a local 
ring, $\gamma$ is a local morphism (that is, $\gamma^{-1}(\mathfrak{m})$ is the set 
of non-invertible elements of $\Gamma$), and $\W$ is a subset of the space of 
valuations centered on $R$ (that is, nonnegative on $R$). There are two main 
instances of local tropicalization. The \emph{positive tropicalization} of $R$ with 
respect to a local morphism $\gamma$ is the tropicalization of the space $\W$ of 
valuations which are strictly positive on the maximal ideal $\mathfrak{m}$ of $R$. 
The \emph{nonnegative tropicalization} is defined similarly, with the only 
difference that we tropicalize all nonnegative valuations on $\mathfrak{m}$.

We consider the following particular instances of the previous definition:

\begin{itemize}
   \item $\Gamma= \Z^n$, $R = K[V]$ where $V$ is an algebraic subvariety of the 
       torus $(K^*)^n$, $\gamma$ is the natural morphism which sends each basis 
       vector $e_i$ of $\Z^n$ to the image $x_i$ in $K[V]$ of the corresponding 
       variable $X_i$, and $\W$ is the set of valuations extending the given one 
       on $K$. Therefore, as a special case of our definition, we get the third 
       version of the definition of the tropicalization of a subvariety of a torus
       as in Theorem~\ref{T:3defs}.

   \item  $\Gamma$ is an arbitrary saturated affine semigroup, $R = K[V]$, 
       $V$ being an algebraic subvariety of the affine toric variety 
       $\mbox{Spec } K[\Gamma]$ defined by $\Gamma$ over $K$ and $\W$ is the whole 
       space $\V(R)$. We get then the notion of tropicalization of a subvariety of 
       a normal affine toric variety introduced by Payne \cite{P 08}. 

\end{itemize}

Our definition of local tropicalization can be applied in the following new 
setting:

\begin{itemize}
   \item We let $I$ be an ideal of a power-series ring $K[[\Gamma]]$, 
       $R := K[[\Gamma]] / I$, $\gamma$ be the natural semigroup morphism 
       associating to each element of $\Gamma$ the image in $R$ of the corresponding 
       monomial, and $\W$ be the subspace of $\V(R)$ of valuations centered at $R$ 
       which extend the trivial valuation of $K$. 
\end{itemize}

Our main structural results about tropicalization state the piecewise-linear 
structure of the local positive tropicalization (see 
Theorem~\ref{T:main} and Proposition~\ref{globloc} for the
general statements). To give the reader an idea of these results, we state here a 
particular case. Let us take $\Gamma=\Z_{\geq 0}^{n}$. Then $K[[\Gamma]]$ is 
isomorphic to the ring of formal power series in $n$ variables.

\begin{theorem}
Let $I$ be an ideal of the ring $K[[X_1 ,\dotsc, X_n]]$ of formal power series in $n$
variables over an arbitrary field $K$ endowed with the trivial valuation. Then:
  \begin{enumerate}
       \item The finite part of the local positive tropicalization $\ptrop(I)$ of 
          the ideal $I$ (that is, of the natural morphism from $\Gamma$ 
          to the quotient local ring 
          $K[[X_1,\dotsc ,X_n]]/I$) is the support of a finite rational polyhedral 
          fan in $(\R_+)^n$. 
          
       \item  If $I$ is prime and $K[[X_1,\dotsc,X_n]]/I$ has Krull dimension $d$, 
          then $\ptrop(I)$ has pure dimension $d$. 
            
      \item If $I$ is the formal completion of the localization at $0$ of an ideal 
          $J$ of the polynomial ring $K[X_1,\dotsc,X_n]$, then the local positive 
          tropicalization $\ptrop(I)$ coincides with the global tropicalization 
          $\trop(J)$ of the subvariety of the torus defined by $J$ inside the open 
          cone $(\R_{>0})^n$. 
   \end{enumerate}
\end{theorem}

The last point of the theorem shows that it is possible to reconstruct the (global) 
tropicalization of a subvariety of a torus from local tropicalizations of its 
closure at the closed points of various toric varieties associated to that torus. 
In this sense, global tropicalization depends only on the boundary structure of 
the subvariety of the torus. In fact, the local tropicalization of an ideal $I$ of 
$K[[\Gamma]]$ depends only on the toroidal structure of the ambient space 
$\mbox{Spec } K[[\Gamma]]$. In order to show this, we prove that, more generally, 
we can tropicalize semigroup morphisms of the form: 
$$(\Gamma, +) \stackrel{\gamma}{\longrightarrow} (R, \cdot)/ (R^*, \cdot),$$ 
where $(R^*, \cdot)$ denotes the subgroup of invertible elements of $(R, \cdot)$. 
This allows, e.~g., to tropicalize objects which are not necessarily endowed 
with a toroidal structure:

\begin{itemize}
    \item If $(X, 0)$ is a germ of normal (algebraic or analytic) variety and $D$ 
    is a reduced Weil divisor on it, consider a finitely generated semigroup 
    $\Gamma$ of effective Cartier divisors supported on $D$. Then, taking 
    $R = \mathcal{O}_{X,0}$, we have a natural semigroup morphism 
    $\Gamma \rightarrow  R/ R^*$, obtained by associating to each Cartier divisor a 
    defining function in $R$, which is well-defined modulo units. 
    
    \item We keep the same setting as in the previous example and let $\Gamma$ be
    the full semigroup of effective Cartier divisors supported on $D$. Then we
    obtain a canonical tropicalization for each ideal $I$ of $\mathcal{O}_{X,0}$ 
    associated to the pair $(X,D)$, by taking $R := \mathcal{O}_{X,0} / I$ and the 
    natural semigroup morphism $\Gamma \rightarrow  R/ R^*$ given by composing the 
    map of the previous example with the map $\mathcal{O}_{X,0}/ \mathcal{O}_{X,0}^* 
    \to R/R^*$ induced by the quotient morphism $\mathcal{O}_{X,0} \to 
    \mathcal{O}_{X,0}/I$. 
\end{itemize}

These examples should be useful for the local study of Weil divisors on algebraic 
or analytic varieties, in such simple cases as those of germs of plane curves. In 
particular, they should allow to understand tropically a good amount of 
combinatorial invariants of singularities, for instance those extracted from 
weighted dual graphs of resolutions or embedded resolutions. 

\medskip

Each section in this paper begins with a brief description of its content. The 
comparison with the existing literature on the subject is concentrated in the last 
section, which also contains a brief description of possible interactions with 
developing fields of mathematics and two open problems.

\medskip
{\bf Acknowledgments:} The first author has benefited a lot from conversations 
with Ang{\'e}lica Cueto, Charles Favre and Anders Jensen and the second one 
from discussions with Mark Spivakovsky. We are grateful to Pedro D. Gonz\'alez 
P\'erez, Hussein Mourtada and Bernard Teissier for their remarks on previous 
versions of this paper. Last but not least, we thank heartedly the two anonymous 
referees for their very careful reading and for their many suggestions for 
clarification.

\medskip 
\section{Geometry of semigroups}\label{S:geomofsmgrps}

In this section we introduce the vocabulary and basic facts about semigroups that
we shall use in this paper. 
\medskip

\begin{definition} \label{defsg}
  A {\bf semigroup} is a set $\Gamma$ endowed with an associative binary operation 
  $+ \colon \Gamma \times \Gamma {\longrightarrow} \Gamma$. 
\end{definition}  
  
\emph{In the sequel, we shall consider only \emph{commutative} semigroups}. The 
simplest examples are abelian groups, but semigroups are interesting precisely 
because of the existence of elements which are not invertible.

If a semigroup has a neutral element $0$, then we call it a semigroup \emph{with 
origin}. If it has an $\infty$ element (also called \emph{absorbing}), that is, an 
element which is unchanged by the addition of any other element, then it is called 
a semigroup \emph{with infinity}. We see immediately that, if they exist, then the 
origin and the infinity are unique.
  
\begin{remark}
  If the semigroup law is thought multiplicatively then, by analogy with 
  $(\R, \cdot)$, the origin is denoted $1$ and the infinity is denoted $0$ (see for 
  example \cite{H 95}), and we speak sometimes about semigroups \emph{with identity}
  and \emph{with zero}. Nevertheless, in the sequel we are consequent with the 
  previous terminology and we say that, when $(R, +,  \cdot)$ is a ring, then $0$ 
  \emph{is the infinity of} the semigroup $(R, \cdot)$. 
\end{remark}
  
  If a semigroup $\Gamma$ has no origin, then we may canonically add such element 
  to it, obtaining the semigroup with origin $\Gamma_0$. If it has no infinity, we 
  can analogously add to it a new element $\infty$, getting $\overline{\Gamma}$.

\begin{definition} \label{gensg}
  A semigroup is called {\bf cancellative} if,  
  whenever $a,b,c \in \Gamma$ satisfy $a+ b = a+c$, we have $b=c$. 
  It is called {\bf of finite type} if it can be generated by a finite number of 
  elements.  It is called {\bf torsion-free} if whenever $a,  b \in \Gamma$ and 
  $ma=mb$ for some $m \in \N^*$,  we have $a=b$. 
\end{definition}
  
Note that a 
semigroup with infinity is not cancellative, excepted in the degenerate 
case when it has only one element, which is necessarily both the origin and the 
infinity. The following type of semigroups will play an essential role in our 
paper: 

\begin{definition} \label{toricsg}
  A semigroup with origin $(\Gamma, +)$ is called {\bf affine} if it is 
  commutative, of finite type, cancellative and torsion-free. 
\end{definition}

The simplest affine semigroups are the various $(\N^n, +)$. The terminology 
``\emph{affine}'' is motivated by the fact that those are precisely the semigroups 
associated to \emph{affine} toric varieties (see the next section and 
\cite{CLS 11}). 
  
Consider a semigroup $\Gamma$ with origin. If $a \in \Gamma$, an \emph{inverse} 
of $a$ is an element $b \in \Gamma$ such that $a + b =0$. If it exists, the 
inverse of $a$ is unique and we denote it simply by $-a$. The set of invertible 
elements is a subgroup of $\Gamma$, which we denote $\Gamma^*$. We let $\Gamma^+$ 
be its complement in $\Gamma$. It is a prime  ideal of $\Gamma$ (see Lemma \ref{primid}), 
in agreement with the next definition:

\begin{definition} \label{sgideal}
  If $\Gamma$ is a semigroup, an \textbf{ideal} of $\Gamma$ is a subset 
  $I \subset \Gamma$ satisfying $I + \Gamma \subset I$. An ideal is called 
  {\bf proper} if $I \neq \Gamma$. The ideal $I$ is {\bf prime} if it is proper  
  and,  whenever $a, b \in \Gamma$ satisfy $a + b \in I$, then at least one of  
  $a,b$ is in $I$. 
\end{definition}

This vocabulary is motivated by the following fundamental example of semigroups:

\begin{example} \label{multring}
   Let $(R, +, \cdot)$ be a commutative ring. Forgetting the addition, $(R, 
   \cdot)$ is a semigroup with origin $1 \in R$. $(R\setminus\{0\},\cdot)$ is 
   cancellative if and only if $R$ is a domain. Any ideal $I$ of the ring 
   $(R, +, \cdot)$ is an ideal of the semigroup  $(R, \cdot)$. The converse is not 
   true, as we do not ask for stability of the operation $+$ in the 
   semigroup-theoretical definition of an ideal. For example, if $R= \Z$, the 
   semigroup-ideal generated by $2$ and $3$ is the set of integers divisible 
   either by $2$ or by $3$, which is not a ring-ideal. 
\end{example}

\begin{lemma}  \label{primid}
    The subsemigroup $\Gamma^+$ of non-invertible elements of $\Gamma$ is a prime 
    ideal of $\Gamma$.
\end{lemma}

\begin{proof}
  Let us first verify that $\Gamma^+$ is an ideal. Suppose that $a \in \Gamma^+$ 
  and $b \in \Gamma$ satisfy $a+b \in \Gamma^*$. This means that there exists 
  $c \in \Gamma$ such that $(a+b) +c =0$. But this can be rewritten by associativity
  as $a + (b+c) =0$, which shows that $a \in \Gamma^*$, a contradiction. Therefore, 
  $\Gamma^+ + \Gamma \subset \Gamma^+$, which is the definition of the fact 
  that $\Gamma^+$ is an ideal. The fact that this ideal is prime is immediate, 
  consequence of the fact that $\Gamma^*$ is stable under addition. 
\end{proof}

It is a formal exercise to see that the preimage of an ideal by a morphism of 
semigroups is again an ideal and that, moreover, in this way, prime ideals are 
transformed into prime ideals. Notice also that each semigroup $\Gamma$ with 
origin is {\em local}, in the sense that it contains a unique maximal ideal 
$\Gamma^+$.

In ring theory, ideals are precisely the kernels of the ring-morphisms. This is 
not true for semigroups. In order to speak about this phenomenon, we introduce 
basic notation about morphisms of semigroups. 
If $\Gamma_1$ and $\Gamma_2$ are semigroups, we denote by: 
$$\Hom_{Sg}(\Gamma_1, \Gamma_2)$$ the set of morphisms of semigroups from 
$\Gamma_1$ to $\Gamma_2$. Analogously, if $R_1$ and $R_2$ are two rings, we denote 
by: 
   $$\Hom_{Rg}(R_1, R_2)$$
the set of ring-morphisms. 

\emph{If both semigroups $\Gamma_1$ and $\Gamma_2$ have origins, 
we assume that a morphism of semigroups sends one origin into the other.} 
$\Hom_{Sg}(\Gamma_1, \Gamma_2)$ has also naturally a structure of semigroup, 
by pointwise addition of the values. 

Let $\phi : \Gamma_1 \rightarrow \Gamma_2$ be a morphism of semigroups with origins. 
Its set-theoretic image $\im (\phi)$ is a subsemigroup of $\Gamma_2$ and its kernel 
$\ker (\phi):= \phi^{-1}(0)$ is a subsemigroup of $\Gamma_1$ (in general it is not 
an ideal; that is, from this viewpoint, semigroups behave more like groups than 
like rings). Nevertheless, unlike for abelian groups, the knowledge of this kernel 
is not enough to determine the associated surjective map 
$\overline{\phi} :\Gamma_1 \rightarrow \im (\phi)$ up to isomorphism.
Indeed, the fact that not all elements are invertible does not allow to conclude 
from $\phi(a) = \phi(b)$ that $a$ is obtained from $b$ by adding an element of the 
kernel. Briefly said, in general the kernel does not describe \emph{all} the 
fibers of the map $\phi$. 

In order to be able to reconstruct the whole map $\overline{\phi}$, 
we have to encode the whole collection of its fibers. This may be done by looking 
at them as the equivalence classes of an equivalence relation $\sim$. This 
equivalence relation on $\Gamma_1$ is \emph{compatible with the addition}, so it is a 
{\em congruence}:

\begin{definition} \label{congr}
   Let $(\Gamma, +)$ be a semigroup with origin. A {\bf congruence} on $\Gamma$ is 
   an equivalence relation compatible with the addition. 
\end{definition}

If $\sim$ is a congruence on $\Gamma$, we see immediately that the addition on 
$\Gamma$ descends naturally to a semigroup law on the quotient $\Gamma/ \sim$, 
the quotient map becoming a morphism of semigroups with origin.

    For instance, the relation defined by 
       $$ a \sim b \: \Leftrightarrow \: \exists \: c \in \Gamma^* 
         \mbox{ such that } a = b + c $$
    is a congruence. It allows to define the quotient semigroup $\Gamma / \Gamma^*$.

To any commutative semigroup $\Gamma$ with origin, one 
functorially associates \emph{the group} $M(\Gamma)$ generated formally by the 
differences of its elements:
  $$a_1 - b_1 = a_2 - b_2 \ \Leftrightarrow \ \exists\  
    c \in \Gamma, \ a_1 + b_2 + c = a_2 + b_1 +c.$$  
  
  The canonical morphism of semigroups 
$\gamma : \Gamma \to M(\Gamma)$ is an embedding if and only if 
$\Gamma$ is cancellative. Indeed, it is an embedding if and only if it is injective, 
which is equivalent to the fact that for any $a_1, a_2, c \in \Gamma$, 
the equality $a_1 + c = a_2 +c$ implies that $a_1 = a_2$. But this is precisely the 
condition of cancellation! 
For example, when $\Gamma = \N$ this gives
the canonical inclusion $\N \hookrightarrow \Z$. 

Assuming $\Gamma$ to be cancellative, it is moreover torsion-free if and only if 
$M(\Gamma)$ is a torsion-free abelian group. Indeed, if there exists $n \in \N^*$ 
and $a \in \Gamma$ such that $n \gamma(a) =0$, then there exists $c \in \Gamma$ 
such that $na + c = c$. As $\Gamma$ is cancellative, we deduce that $na=0$. 
As $\Gamma$ is torsion-free, we conclude that $a =0$. 

On the other hand, it is not true that $\Gamma$ is of finite type if and only if
$M(\Gamma)$ is of finite type. For instance, $M((\N^*)^2) = \Z^2$ is of finite 
type but $(\N^*)^2$ is not of finite type. Only the following implication holds: if 
$\Gamma$ is of finite type, then so is $M(\Gamma)$. 

We define a \emph{lattice} as an abelian torsion-free group of finite type. The 
previous explanations have as a direct consequence the following characterization 
of affine semigroups: 

\begin{lemma}
   A semigroup is affine if and only if it is a finitely generated subsemigroup  
   of a lattice and it has an origin.
\end{lemma}

Let $\Gamma$ be an affine semigroup and $M(\Gamma)$ be its associated lattice. 
We denote by $N(\Gamma) := \Hom_{Gp}(M(\Gamma), \Z)$ the dual lattice. 

\begin{definition} \label{satur}
   The {\bf saturation}  $\sat(\Gamma)
\hookrightarrow M(\Gamma)$ of $\Gamma$ (inside 
$M(\Gamma)$) is the subset of $M(\Gamma)$ formed by the elements $v$
satisfying $n v\in\Gamma$ for some $n \in \N^*$. A semigroup is 
called {\bf saturated} if it is equal to its saturation.
\end{definition}

\begin{example} \label{nonsatex}
  Let us consider the affine subsemigroup $\Gamma$ of $\N \times \Z$ generated 
  by $v_1 = (2,1), v_2 = (5,2), v_3 = (0, 3), v_4 = (0, -3)$ (see 
  Figure \ref{fig:Affsg}). 
  The associated lattice $M(\Gamma)$ is equal to $\Z^2$, and $\sat(\Gamma)$
  is $\N \times \Z$. 
  As is visible in the drawing, $\Gamma^*$ is the subgroup of $\Z \times \Z$ 
  generated by $v_3$. In the drawing is also represented the quotient map 
  $p : \Gamma \to \Gamma/ \Gamma^*$. This last semigroup $\Gamma/ \Gamma^*$ 
  is isomorphic to the image of $\Gamma$ by the canonical projection of 
  $\N \times \Z$ to the first factor $\N$. Therefore it is affine.
\end{example}

\bigskip
\begin{figure}[h!]
\labellist
\small\hair 2pt
\pinlabel  {$0$} at 36 294
\pinlabel  {$v_1$} at 96 332
\pinlabel  {$v_2$} at 150 314
\pinlabel  {$v_3$} at 37 349
\pinlabel  {$v_4$} at 37 240
\pinlabel  {$\Gamma^*$} at -10 130
\pinlabel  {$p$} at 234 78
\pinlabel  {$0$} at 34 7
\endlabellist
\centering
\includegraphics[scale=0.50]{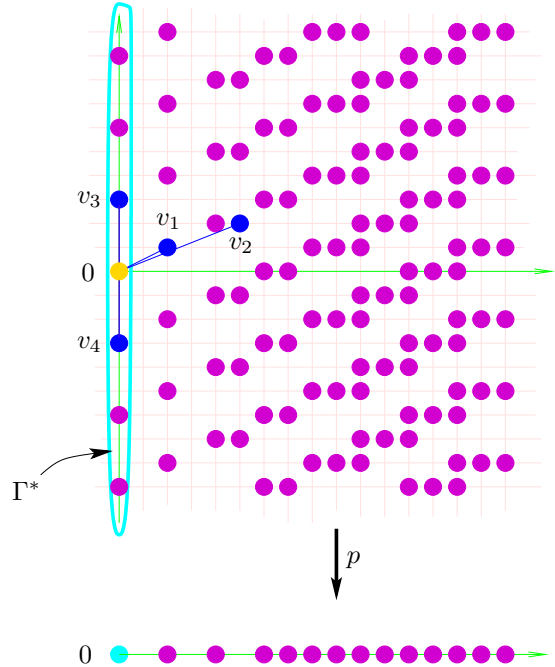}
\caption{An affine semigroup and its quotient by its subgroup of invertible 
elements}
\label{fig:Affsg}
\end{figure}
\medskip

\begin{example} \label{nonsatexbis}
  In the previous example, the quotient $\Gamma / \Gamma^*$ was again affine. This 
  is not true for all affine semigroups $\Gamma$. Consider for instance the affine 
  subsemigroup of $\Z^2$ generated by $v_1 = (1,0), v_2 = (1,1), v_3 = (0,2), 
  v_4 = (0, -2)$ (see Figure \ref{fig:Affsgtor}).  
  Then $\Gamma^*$ is the lattice 
  of rank one generated by $v_3$. 
  The quotient $\Gamma / \Gamma^*$ has torsion, as the images $\gamma(v_1)$ 
  and $\gamma(v_2)$ are different (there does not exist any $c \in \Gamma^*$ such 
  that $v_2 = v_1 + c$) but their doubles are equal (as $2v_2 = 2v_1 + v_3$). 
  In fact, the restriction to $\Gamma$ of the second projection $\Z^2 \to \Z$ 
  factors through the quotient map $p$, inducing a map $\overline{p} :  \Gamma / \Gamma^*
  \to \N$. The fibers $\overline{p}^{-1}(n)$ of this map have two points for $n >0$, 
  only the origin being covered by one point. That is why we represented the 
  $\Gamma / \Gamma^*$ as the set $\N$ in which every positive number is split into 
  two points. 
\end{example}

\bigskip
\begin{figure}[h!]
\labellist
\small\hair 2pt
\pinlabel  {$0$} at 36 294
\pinlabel  {$v_2$} at 106 350
\pinlabel  {$v_1$} at 103 286
\pinlabel  {$v_3$} at 28 372
\pinlabel  {$v_4$} at 37 227
\pinlabel  {$\Gamma^*$} at -10 180
\pinlabel  {$p$} at 234 78
\pinlabel  {$0$} at 34 7
\endlabellist
\centering
\includegraphics[scale=0.50]{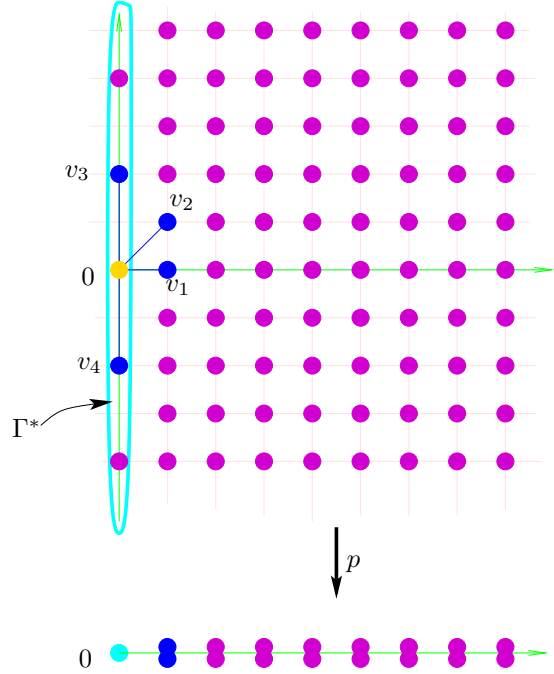}
\caption{A quotient with torsion of an affine semigroup by its subgroup of invertible 
elements}
\label{fig:Affsgtor}
\end{figure}
\medskip

The situation of the previous example cannot happen for saturated affine semigroups:

\begin{proposition}  \label{quotaff}
  If the affine semigroup $\Gamma$ is saturated, then 
     $\Gamma / \Gamma^*$ is also a saturated affine semigroup. 
\end{proposition}

\begin{proof}
  If $a \in \Gamma$, denote by $\overline{a}$ its image in $\Gamma' := \Gamma / 
  \Gamma^*$. As a quotient of a commutative semigroup of finite type, $\Gamma'$ is 
  also commutative and of finite type. 
        
  Let us show that $\Gamma'$ is {\em cancellative}. Suppose that $a, b, c \in \Gamma$ 
  satisfy the equality $\overline{a} + \overline{b} = \overline{a} + \overline{c}$. 
  This implies that there exists $d \in \Gamma^*$ such that 
  $a+ b = (a + c) +d  = a + (c + d)$. As $\Gamma$ is assumed cancellative, 
  we deduce that $b = c + d$, which implies that $\overline{b} = \overline{c}$. 
  That is, $\Gamma'$ is also cancellative.
       
  We show now that $\Gamma'$ is {\em torsion-free}. Assume that $a, b \in \Gamma$ 
  satisfy an equality of the type $n\overline{a} = n \overline{b}$, with $n \in 
  \N^*$. Therefore, there exists $c \in \Gamma^*$ such that $na = nb + c$. Inside 
  the lattice $M(\Gamma)$ (into which $\Gamma$ embeds canonically), we 
  may write the previous equality as $n(a-b) = c$. Our hypothesis that $\Gamma$ is 
  saturated implies that there exists $d \in \Gamma$ such that $a-b =d$. 
  The same argument repeated with the equality $n(b-a) = -c$ would give us 
  a $d' \in \Gamma$ with $b-a = d'$. Then $d + d' =0$, which shows that 
  $d \in \Gamma^*$. We conclude that $\overline{a} = \overline{b}$. That is, 
  $\Gamma'$ is torsion-free. 
       
  Finally, let us show that $\Gamma'$ is also {\em saturated}. 
  The previous argument shows that 
  $\Gamma^*$ is a primitive sublattice of $M(\Gamma)$, that is, that the quotient 
  $M(\Gamma) / \Gamma^*$ is torsion-free. As the images of the generators of 
  $\Gamma$ also generate this quotient, 
  we deduce that $M(\Gamma) / \Gamma^*$ is canonically isomorphic 
  to $M(\Gamma')$. We will work with this representative of the associated lattice. 
  Consider therefore $v \in M(\Gamma)$ such that there exists 
  $n \in \N$ and $a \in \Gamma$ with $n \overline{v} = \overline{a}$ in 
  $M(\Gamma) / \Gamma^*$. Therefore, there exists $c \in \Gamma^*$ such that 
  $nv = a + c$. But $a + c \in \Gamma$ and $\Gamma$ is saturated, which implies that 
  $v \in \Gamma$. Therefore  $\overline{v} \in \Gamma'$, which shows that 
  $\Gamma'$ is saturated.       
\end{proof}

\medskip
Until now we have only discussed algebraic aspects of semigroups. We now describe
their topology. Suppose that the semigroup $\Gamma$ \emph{has no infinity} and 
moreover \emph{is totally ordered}, that is, it is endowed with a total order 
compatible with the addition. Then, we can equip $\Gamma$ with a natural topology 
generated by the ``open'' intervals. We extend this topology to $\overline{\Gamma}$ 
by taking as basis of neighborhoods of $\infty$ the
subsets of the form:  
$$(a, \infty]:= \{x \in \Gamma \ | \ 
   x > a \}  \cup \{ \infty \}.$$

Note that with this convention, $\infty = + \infty$, not $- \infty$. We will mainly 
use the previous construction of topology when $\Gamma$ is $\R$, $[0, \infty)$, 
$(0, \infty)$, $\Z$, $\N$, the semigroup operation beeing addition.

\section{Toric varieties} \label{toric}

This section is intended only to set the notations we use for toric geometry. 
For details on \emph{normal} toric varieties, we refer to Fulton's book \cite{F 93}.  
For not necessarily normal toric varieties, the reader can consult the recent 
monograph \cite{CLS 11} of Cox, Little, and Schenck, or Gonzalez P\'erez and 
Teissier's paper \cite{GT 09}. 
\medskip

In the sequel, if $G$ is an abelian group and $K$ is a field, we will denote by 
$G_{K}$ the $K$-vector space $G \otimes_{\Z} K$. 

Let $\Gamma$ be an affine smigroup. We denote by: 
   $$\check{\sigma}(\Gamma) \subset M(\Gamma)_{\R}$$ 
the finite rational polyhedral cone generated by $\Gamma$. By definition, it 
consists of all the combinations with nonnegative real coefficients of elements
of $\Gamma$. It is a sub-cone of the real vector space $M(\Gamma)_{\R}$ with 
non-empty interior. The saturation of $\Gamma$ may be described geometrically using 
the cone $\check{\sigma}(\Gamma)$:
  \begin{equation} \label{geoint}
      \sat(\Gamma)= \check{\sigma}(\Gamma) \cap M(\Gamma).  
  \end{equation}

We denote by
$\sigma(\Gamma)\subset N(\Gamma)_{\R}$ the dual cone, defined by:
  $$\sigma(\Gamma):= \{ w \in N(\Gamma)_{\R} \: | 
  \: w(\check{\sigma}(\Gamma)) \subset
  \R_{\geq 0} \}.$$

More generally, if $\sigma$ is a polyhedral cone in a finite dimensional 
real vector space $V$ and $\check{\sigma}$ is its dual cone in the dual space $V^*$,
then $\check{\check{\sigma}} = \sigma$ and the map:
  $$ \tau \longrightarrow \check{\sigma} \cap  \tau^{\perp}$$
establishes an inclusion-reversing bijection between the closed faces of $\sigma$ 
and those of $\check{\sigma}$. 
  
The cone $\check{\sigma}$ has non-empty interior if and only if $\sigma$ is 
\emph{strictly convex}, that is, if it does not contain any vector subspace of 
positive dimension. We will say also in this case that $\sigma$ is a 
\emph{pointed cone}. If $\Gamma^* =0$, we say that $\Gamma$ is a \emph{pointed
semigroup}. It is immediate to check that the affine semigroup $\Gamma$ is
pointed if and only if the cone $\check{\sigma}(\Gamma)$ is pointed.

The vocabulary introduced in the following definition is taken from \cite{GT 09}:

\begin{definition}
  If $\Gamma$ is an affine semigroup, then a {\bf face} of $\Gamma$ is a 
  subsemigroup $\Lambda \subset \Gamma$ such that whenever $x, y \in \Gamma$ 
  satisfy $x + y \in \Lambda$, then both $x$ and $y$ are in $\Lambda$.  
\end{definition} 

The following proposition characterizes the faces of $\Gamma$:

\begin{proposition} \label{faces}
   The faces of $\Gamma$ are precisely the complements of the prime ideals of 
   $\Gamma$. The map $\tau \longrightarrow \Gamma \ \cap \ \tau^\perp$ establishes 
   an inclusion-reversing  bijection between the faces of $\sigma(\Gamma)$ and 
   those of $\Gamma$. 
\end{proposition}

In the sequel, if $\Gamma$ is affine and $\tau$ is a face of $\sigma(\Gamma)$, we 
denote: 
  $$\Gamma_{\tau} : =  \Gamma \cap \tau^\perp, \ \
    M(\tau, \Gamma) := M( \Gamma_{\tau}),  \ \   
  M(\tau) := M  \cap \tau^\perp. $$
If $N_{\tau}$ denotes the sublattice of $N$ spanned by $\tau \cap N$, the 
quotient $N / N_{\tau}$ is canonically dual to $M(\tau)$, i.e., $N/N_{\tau}\simeq
\Hom(M(\tau),\Z)$. 

The subgroup $\Gamma^*$ of invertible elements is the {\em minimal} face, in the sense 
that it is contained in all the other ones. By the previous bijection, 
it corresponds to $\tau = \sigma$.

\medskip
Let $\Gamma$ be an affine semigroup and $K$ be a field. We denote by:  
  $$\zcal_K(\Gamma):= \spec K[\Gamma]$$
the associated \emph{toric variety} defined over $K$. Its $K$-\emph{valued points} 
are naturally identified with the semigroup: 
   \begin{equation} \label{idpoints}
       \Hom_{Rg}(K[\Gamma], K) \simeq \Hom_{Sg}((\Gamma, +), (K, \cdot)).
    \end{equation}
Notice that the multiplicative semigroup $(K, \cdot)$ has $0$ as infinity. 

When $\Gamma$ is of the form $\check{\sigma}\cap M$, and $\sigma$ is a strictly 
convex rational polyhedral cone in the real vector space $N_{\R}$, we define:
 $$\zcal_K(\sigma, N):= \zcal_K(\check{\sigma} \cap M).$$
These are precisely the \emph{normal} affine toric varieties. 
 
In the same way as abstract varieties over a field are obtained by gluing affine 
cones, we can glue affine toric varieties by respecting the ambient structure, 
that is, the action of the torus $T_K(N) : = \spec K[M]$. This is easiest to 
describe in the case of normal toric varieties: the combinatorial object encoding 
the gluing is a \emph{fan}.

\begin{definition}
A {\bf fan} in $N_{\R}$ is a finite set $\Delta$ of convex  
polyhedral cones inside $N_{\R}$, such that:

a) for each cone $\sigma$ in $\Delta$, all its faces are in $\Delta$;

b) if $\sigma_1$ and $\sigma_2$ are cones of $\Delta$, then $\sigma_1
\cap \sigma_2$ is a common face. 

\noindent If all the cones are rational, that is, they are defined as the 
intersections of halfspaces $\{v\in N_{\R} \ |  \ \langle v,m\rangle \geq 0\}$, where 
$m\in M$ and $\langle\cdot,\cdot\rangle$ is the pairing between $N$ and $M$, then 
the fan is called {\bf rational}.
\end{definition}

These conditions imply that all the cones in $\Delta$ have a maximal common
linear suspace. We say that $\Delta$ is a \emph{pointed fan}, if this linear 
subspace is the origin, that is, if all the cones of $\Delta$ are \emph{strictly
convex}. 

If $\Delta$ is a pointed finite rational polyhedral fan inside $N_{\R}$,  
we denote by $\zcal_K(\Delta, N)$ the normal toric variety over the field $K$ 
associated to the lattice $N$ and the fan $\Delta$. It is obtained by 
the usual gluing of  affine toric varieties: 
  $\zcal_K(\sigma_1, N)$ and  $\zcal_K(\sigma_2, N)$ are glued along  
  $\zcal_K(\sigma_1 \cap \sigma_2, N)$, which is an open affine toric subvariety 
  of both of them  (see \cite{F 93} or \cite{CLS 11}).

There is an incidence-reversing bijection between the cones of $\Delta$ and the 
orbits of the torus action on $\zcal_K(\Delta, N)$. We denote by $O_{\sigma}$ the 
orbit associated to the cone $\sigma \in \Delta$. It is canonically identified 
with the torus $T_{N/N_{\sigma}, K}$. 

We may also encode combinatorially the gluing of not necessarily normal affine 
toric varieties, as explained by Gonz\'alez P\'erez and Teissier 
in \cite{GT 09}. For this, we propose the following new notion: 

\begin{definition} \label{sgfan}
A {\bf fan of semigroups} $\mathcal{S}$ in $N$ is a rational fan $\Delta$ in 
$N_{\R}$, enriched with an affine subsemigroup $\Gamma_\alpha$ of $\Gamma$ for each 
cone $\alpha \in \Delta$, with the property that:

i) for each cone $\alpha$ in $\Delta$, $\sigma(\Gamma_{\alpha}) = \alpha$ and 
   $M(\Gamma_{\alpha}) = M$;

ii) if $\beta$ is a face of $\alpha \in \Delta$, then $\Gamma_{\beta} = 
    \Gamma_{\alpha} + M(\beta, \Gamma_{\alpha})$. 
\end{definition}

Note that condition i) implies that the fan $\Delta$ is pointed. 
These conditions allow to glue the affine toric varieties $\zcal_K(\Gamma)$ 
corresponding to the various semigroups of a given fan of semigroups 
$\mathcal{S}$. We denote by $\zcal_K(\mathcal{S})$ the associated toric variety.

\medskip
\section{Linear varieties associated to semigroups}  \label{troptor} 

In this section we develop a theory of embeddings of topological semigroups into 
bigger stratified topological spaces endowed with an action of the initial 
semigroup. This generalizes a construction introduced by Ash, Mumford, Rapoport, 
Tai \cite[I.1]{AMRT 75} and developed recently by Payne \cite{P 08} and 
Kajiwara \cite{K 08}. Their setting corresponds to the case when the semigroup is 
a strictly convex cone in a finite dimensional real vector space. 
\medskip

Recall that we assume all the semigroups to be commutative and with origin. Let $G$ 
and $H$ be semigroups, and denote by:
   $$\zcal_H(G) : = \Hom_{Sg}(G, H)$$
the semigroup of semigroup morphisms from $G$ to $H$. We think about it as the set 
of $H$-\emph{valued points} of the semigroup $G$. Moreover, when $H$ is a 
topological semigroup, we endow $\zcal_H(G)$ with the \emph{topology of pointwise 
convergence}, that is, the induced topology coming from the natural embedding 
$\zcal_H(G) \hookrightarrow H^G$, the target space being endowed with the 
product topology. 

\begin{example} \label{sgpoints}
    When $G$ is an affine semigroup and $H$ is the multiplicative group 
    $(K^*, \cdot)$ of a field $K$, then $\zcal_H(G)$ equals the torus $T_K(N(G))$, 
    whose lattice of characters is the lattice $M(G)$ associated to $G$. The torus
    $T_K(N(G))$ is naturally an algebraic variety and bears the Zariski topology.
    If there is some natural topology on $K$, the topology of pointwise convergence
    on $T_K(N(G))$ is different from the Zariski topology.
\end{example}

\begin{example}
    When $G$ is either an affine semigroup or a polyhedral cone and 
    $H$ is the additive group $(\R, +)$, then $\zcal_H(G)$ is the real vector 
    space $N(G)_{\R}$. This notation was explained before in the case of affine 
    semigroups. When $G$ is a cone $\check{\sigma}$, $N(G)_{\R}$ denotes the dual 
    space to the vector space $M(\check{\sigma})$ generated by $\check{\sigma}$.
\end{example}

Notice that the functor $(G, H) \rightarrow \zcal_H(G)$ is contravariant in the 
variable $G$ and covariant in the variable $H$ (this is, of course, valid in any 
category). When $H$ has no infinity, we get in particular a natural embedding 
of semigroups: 
 $$\zcal_H(G)  \hookrightarrow \zcal_{\overline{H}}(G). $$

\begin{definition} \label{linvarsg}
  If $G$ is a semigroup and $H$ a semigroup without infinity, we say that 
  $\zcal_{\overline{H}}(G)$ is the $H$-{\bf valued (affine) linear variety of} $G$. 
\end{definition}

\begin{remark} \label{reasonlin}
 We chose this name in analogy with that of \emph{toric varieties}. Indeed, when 
 $G$ is an affine semigroup and $H = (K^*, \cdot)$, as in Example \ref{sgpoints}, 
 $\zcal_{\overline{K^*}}(G) = \zcal_{(K, \cdot)}(G) = \Hom_{Sg}(G, (K, \cdot))$, 
 is the set of $K$-valued points of the affine toric variety $\zcal_K(G)$ 
 (see formula \eqref{idpoints}). The attribute ``\emph{toric}'' makes reference to
 a natural action of an algebraic (split) torus, whose law is thought of
 multiplicatively. In our context, the analog of the torus is the semigroup 
 $\zcal_H(G)$, thought of additively. It acts naturally on the linear variety
 $\zcal_{\overline{H}}(G)$. The most important case for us is $H=\R$, when
 $\zcal_H(G)$ is a vector spce. This explains the attribute ``\emph{linear}''
 in our terminology. In what concerns the attribute ``affine'', it makes reference
 to the fact that we define an analog of the notion of \emph{affine} 
 toric variety. 
\end{remark}
 
Assume now that $H$ is a group. In the same way as toric varieties are canonically
stratified into the orbits of the associated torus, the linear variety 
$\zcal_{\overline{H}}(G)$ is stratified into the orbits of the natural action of 
$\zcal_H(G)$ on $\zcal_{\overline{H}}(G)$ induced by the addition 
$H \times \overline{H} \rightarrow \overline{H}$ on the values. For an affine 
$G$ and divisible $H$, these orbits may be described in a different way, 
using the notion of prime ideal of a semigroup (see Definition \ref{sgideal}): 
 
\begin{proposition}
   Let $H$ be a divisible group and $G$ an affine semigroup. The orbits of the 
   natural action of $\zcal_H(G)$ on $\zcal_{\overline{H}}(G)$ are in a bijection
   with the prime ideals of $G$. The bijection is given by:
   $$\text{the orbit of } \gamma\in \zcal_{\overline{H}}(G)\quad \leftrightarrow
   \text{ the prime ideal } \gamma^{-1}(\infty \in \overline{H}).$$
\end{proposition}
 
Therefore, those orbits are in natural bijection with the faces of $G$ (see 
Proposition~\ref{faces}). If $\scal$ is a fan of affine semigroups and $H$ is a 
divisible group, then we have the canonical identification 
$\zcal_H(\Gamma) = \zcal_H(M(\Gamma))$, where $\Gamma \in \scal$ and $M(\Gamma)$ 
is the same lattice for all the semigroups $\Gamma \in \scal$, by condition i) of 
Definition~\ref{sgfan}.

\medskip

\begin{remark} \label{remlin}
The construction of \cite[I.1]{AMRT 75} alluded to in the introductory 
paragraph of this section,
and developed further by Payne \cite{P 08} and Kajiwara \cite{K 08}, corresponds to 
the case when $G$ is a saturated affine semigroup $\check{\sigma} \cap M$ and 
$H = \R$ (see again Example \ref{sgpoints}). We chose to develop 
this more general categorical viewpoint for the following reasons:

\begin{enumerate}  
  \item To get extra structures on $G_H$ 
     from the functorial properties of our construction. For instance, when $G$ is 
     affine, then, the integral points of $\zcal_{\overline{\R}}(G)$ are the 
     points in $\zcal_{\overline{\Z}}(G)$.
  
  \item  To study also valuations taking values in totally ordered groups which do 
     not embed into $\R$, that is, which have rank at least $2$. For 
     instance, this could be useful when developing the theory initiated by 
     F. Aroca in \cite{A 10}. 
\end{enumerate}
\end{remark}

We focus now on the topological aspects of the constructions. Let  
$G$ be a saturated affine semigroup $\check{\sigma} \cap M$, $H = \R$, 
$\sigma\subset N_{\R}$ is a strictly convex rational polyhedral cone. 
Fix:
  \begin{equation} \label{aftt}
     L(\sigma, N):= \Hom_{Sg}(\check{\sigma}\cap M, \ri ) =
     \zcal_{\ri} (\check{\sigma}\cap M),   
  \end{equation}

  \begin{equation} \label{tc}
     \overline{\sigma}= \overline{(\sigma, N)}:=
     \Hom_{Sg}(\check{\sigma}\cap M, 
     \ri_{\geq 0} )  = \zcal_{\ri_{\geq 0}} (\check{\sigma}\cap M). 
  \end{equation}
Whenever $N$ is clear from the context, we omit it and denote 
$\overline{(\sigma, N)}$ simpy by $\overline{\sigma}$. We denote by
$\overline{\sigma}^\circ$ the subspace of $\overline{\sigma}$
consisting of those semigroup morphisms $\check{\sigma}\cap M\to \ri_{\geq 0}$
which take only positive values (possibly $+\infty$) on the maximal ideal
of the semigroup $\check{\sigma}\cap M$. We say that $\overline{\sigma}^\circ$
is the \emph{interior} of $\overline{\sigma}$.

We view $L(\sigma, N)$ as a space of functions from the set
$\check{\sigma} \cap M$ to the topological space $\ri$, and endow it with the 
topology of pointwise convergence. This is the weakest topology
for which all the sets $\{\gamma\in L(\sigma, N)\,|\,\gamma(m)\in U\}$ are
open, where $m\in\check{\sigma}\cap M$ and $U$ is an open subset of $\ri$. 
Since the topological space $\ri$ is separated, this topology is also separated. 
Since $\R$ is a dense open subspace of $\ri$, we see that $N_{\R}$ is also a dense
open subspace of $L(\sigma, N)$. 

Respecting the conventions of Definition \ref{linvarsg},  we introduce the 
following terminology:

\begin{definition} \label{linvardef}
   The topological space $L(\sigma, N)$, endowed with the natural
   continuous action $N_{\R} \times L(\sigma, N) \rightarrow L(\sigma,
   N)$ extending the action of $N_{\R}$ on itself by translations
   is called {\bf the affine linear variety} associated to the
   pair $(N, \sigma)$. We say that the closure of the cone $\sigma$ in $L(\sigma, N)$ 
   is the {\bf extended cone} $\overline{\sigma}$. 
\end{definition}

The affine linear variety $L(\sigma, N)$ is obtained by adding to the vector 
space $N_{\R}$ some strata \emph{at infinity}, each stratum being by definition 
an orbit of the previous action. These strata have a canonical structure of vector 
spaces, in the same way as the orbits of the canonical action of a torus on an 
associated affine toric variety are canonically lower-dimensional tori. More 
precisely, they are canonically identified with the vector spaces 
$(N/ N_{\tau})_{\R}$, where $\tau$ is a face of the cone $\sigma$ (including $0$ 
and $\sigma$ itself). Here, $N_{\tau}$ denotes the intersection of the vector 
space spanned by $\tau$ with the lattice $N$.  

We now introduce a topology on the disjoint union 
$\bigsqcup_{\tau} (N/ N_{\tau})_{\R}$.  
  If $U$ is an open subset of $N_{\R}$ and $\delta$ is a face of $\sigma$, 
  we consider the set:
  \begin{equation}\label{E:opensetoflv}
      \overline{U}_{\delta}:= \bigsqcup_{\tau \leq \delta}
      \pi_{\tau}(U + \delta) \subseteq 
      \bigsqcup_{\tau \leq \sigma} (N/ N_{\tau})_{\R} 
  \end{equation}
  where $N_{\R} \overset{\pi_{\tau}}{\longrightarrow} (N/N_{\tau})_{\R}$ is the 
  canonical projection, the first union is taken over all faces $\tau$ of the cone 
  $\delta$, and the second over those of $\sigma$. 
  
  The disjoint union $\bigsqcup_{\tau} (N/ N_{\tau})_{\R}$ 
enowed with the previous topology is a partial 
compactification of $N_\R$. Intuitively, this topology may be 
explained as follows: {\em the sequence $(v_n)_{n\in \N}\subset N_\R$ tends to 
$v^\tau \in (N/ N_{\tau})_{\R}$ if and only if $(v_n)_{n\in \N}$ 
tends to infinity in the direction
of the cone $\tau$ and the sequence of projections $(p_\tau(v_n))_{n\in \N}$ 
converges
to $v^\tau \in (N/ N_{\tau})_{\R}$ inside the space $(N/ N_{\tau})_{\R}$}. 
Let us be more precise about the meaning of the first part of this 
characterization. Choose an arbitrary 
linear projection $\psi_{\tau}$ of $N_{\R}$ onto 
the linear span $(N_{\tau})_{\R}$ of the cone $\tau$. Then {\em $(v_n)_{n\in \N}$ 
tends to infinity in the direction of the cone $\tau$ if and only if 
the sequence $(\psi_{\tau}(v_n))_{n\in \N}$ gets eventually out of any compact of 
$(N_{\tau})_{\R}$ and also enters eventually any fixed neighborhood of 
$\tau$, also inside $(N_{\tau})_{\R}$}.

\begin{example}
  Let $N = \Z^2$ and let $\sigma$ be the convex polyhedral cone generated by the 
  vectors $(1,0)$ and $(1,2)$ in $\R^2$. Denote the $1$-dimensional faces of 
  $\sigma$ by $\tau_1=\langle(1,0)\rangle$ and $\tau_2=\langle(1,2)\rangle$. 
  The corresponding
  stratification of the disjoint union $L=\bigsqcup_{\tau\leq\sigma} 
  (N/ N_{\tau})_{\R}$ consists of the four pieces: $L_0=\R^2$, 
  $L_1=\R^2/\R\cdot\tau_1$, $L_2=\R^2/\R\cdot\tau_2$, and the point $L_{12}=
  \R^2/\R\cdot\sigma$. The cone $\sigma$ and the stratification are 
  schematically
  shown in Figure~\ref{fig:Compactif}. Let the open set $U$ be an 
  open circle and
  $\delta=\tau_2$. Then the corresponding open subset of $L$ is
  $$U_{\tau_2} = (U+\tau_2)\ \sqcup \  \pi_2(U+\tau_2),$$
  where $\pi_2 := \pi_{\tau_2}\colon \R^2\to L_2$ 
 is the canonical projection (see again Figure~\ref{fig:Compactif}).
\end{example}

\bigskip
\begin{figure}[h!]
\labellist
\small\hair 2pt
\pinlabel  {$\tau_1$} at 144 50
\pinlabel  {$\tau_2$} at 46 211
\pinlabel  {$\sigma$} at 66 99
\pinlabel  {$0$} at -10 60
\pinlabel  {$U$} at 144 111
\pinlabel  {$U + \tau_2$} at 195 181
\pinlabel  {$L_1$} at 295 150
\pinlabel  {$L_2$} at 104 346
\pinlabel  {$\pi_2(U + \tau_2)$} at 220 306
\pinlabel  {$L_{12}$} at 290 260
\endlabellist
\centering
\includegraphics[scale=0.50]{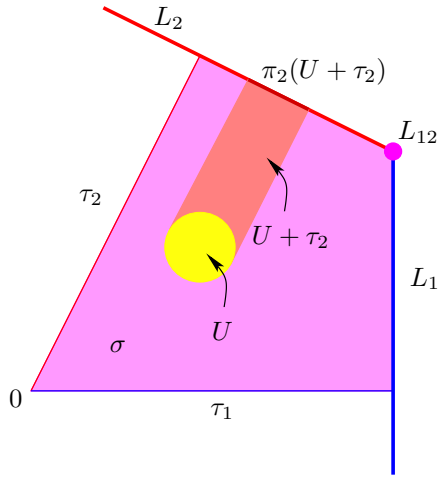}
\caption{An affine linear variety of dimension $2$}
\label{fig:Compactif}
\end{figure}
\medskip

\begin{example}
  In order to indicate better the adjacencies of strata which appear by the 
  construction of the affine linear variety associated to a pair $(\sigma, N)$, 
  let us also represent a $3$-dimensional situation. We consider a lattice $N$ of 
  rank $3$ and inside $N_{\R}$ a strictly convex cone $\sigma$ of dimension $3$ 
  having $4$ edges, denoted $\tau_1,\dotsc, \tau_4$. Denote also by $\tau_I$ the 
  face of $\sigma$ spanned by the subset $I$ of $\{1 ,\dotsc,4\}$, whenever we get 
  indeed a face, and by $L_I := (N/ N_{\tau_I})_{\R}$. In particular, 
  $\tau_{1234}= \sigma$, therefore $L_{1234}$ is a point. In 
  Figure \ref{fig:Spaceconebis} we represented $L(\sigma, N)$, as well as the 
  canonical projections $\pi_I(\sigma)$ of $\sigma$ to the strata at infinity 
  $L_I$ (where, as in the previous example, we denote $\pi_I := \pi_{\tau_I}$).  
\end{example}

\bigskip
\begin{figure}[h!]
\labellist
\small\hair 2pt
\pinlabel  {$0$} at 75 -10
\pinlabel  {$\tau_1$} at 124 88
\pinlabel  {$\tau_2$} at 124 41
\pinlabel  {$\tau_3$} at 109 140
\pinlabel  {$\tau_4$} at 72 126
\pinlabel  {$L_1$} at 237 143
\pinlabel  {$L_2$} at 267 179
\pinlabel  {$L_4$} at 46 227
\pinlabel  {$L_{1234}$} at 189 300
\pinlabel  {$\pi_1(\sigma)$} at 165 325
\pinlabel  {$\pi_2(\sigma)$} at 276 285
\pinlabel  {$\pi_4(\sigma)$} at 88 312
\pinlabel  {$\sigma$} at 185 37

\endlabellist
\centering
\includegraphics[scale=0.8]{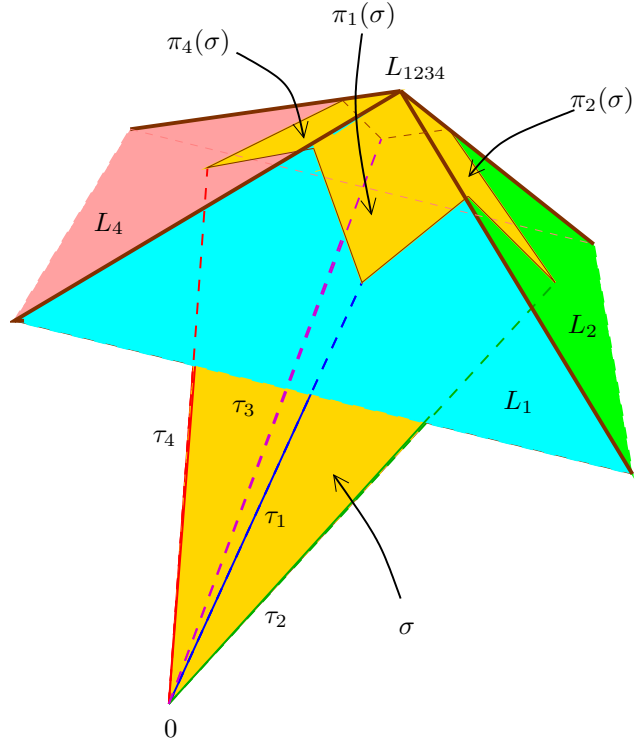}
\caption{A 3-dimensional affine linear variety}
\label{fig:Spaceconebis}
\end{figure}
\medskip

\begin{proposition}\label{P:opensetoflv}
    The sets of the form $\overline{U}_{\delta}$, where $U$ is an open subset
    of $N_{\R}$ and $\delta$ is a face of the cone $\sigma$, form a 
    basis of open sets for a topology on 
    $\bigsqcup_{\tau} (N/ N_{\tau})_{\R}$, where $\tau$ varies over the 
    faces of the cone $\sigma$. 
\end{proposition}
\begin{proof}
The proof is easy and left to the reader.
\end{proof}

Note that any element $\gamma$ of $\bigsqcup_{\tau} (N/ N_{\tau})_{\R}$ defines
a semigroup homomorphism from $\check{\sigma}\cap M$ to $\ri$. Indeed,
if $\gamma\in (N/ N_{\tau})_{\R}$, then $\gamma$ defines a canonical homomorphism 
from $\tau^\perp \cap \check{\sigma}\cap M$ to $\R$, where $\tau^\perp$ is the 
subspace of $M_{\R}$ orthogonal to $\tau$. Extend this homomorphism 
to $\check{\sigma}\cap M$ by setting $\gamma(m)=+\infty$ for all
$m\in\check{\sigma}\cap M$, $m\notin \tau^\perp$. In this way we get a canonical 
map $\bigsqcup_{\tau\leq \sigma} (N/ N_{\tau})_{\R} \rightarrow L(\sigma,N)$. 
Now the stratification of $L(\sigma,N)$ may be described set-theoretically
as follows:

\begin{lemma} \label{setdescraff}
   The canonical map $\bigsqcup_{\tau\leq \sigma} (N/ N_{\tau})_{\R} 
   \rightarrow L(\sigma,N)$ is a homeomorphism. 
\end{lemma}
\begin{proof}
   Let us denote this canonical map by $F$. The injectivity of $F$ follows
   directly from its construction. To show surjectivity, we consider a homomorphism
   $\gamma\colon\check{\sigma}\cap M\to \ri$. The set of 
   $m\in\check{\sigma}\cap M$ satisfying $\gamma(m)=+\infty$ is a
   semigroup ideal of $\check{\sigma}\cap M$. Moreover, this ideal is prime,
   that is, if $\gamma(m_1+m_2)=+\infty$, then $\gamma(m_1)=+\infty$ or 
   $\gamma(m_2)=+\infty$. As we already said in Section~\ref{S:geomofsmgrps}, such
   an ideal can be only the complement in $\check{\sigma}\cap M$ of an intersection
   $\tau^\perp \cap \check{\sigma}$ for some face $\tau$ of the cone $\sigma$. This
   clearly implies that $\gamma$ is in the image of $F$.

   Next, let us prove that the map $F$ is continuous in both directions. Let
   $U$ be an interval $(a,b)$, $a\in\R$, $b\in\ri$ or half-interval
   $(a,+\infty]$ of the extended real line $\ri$ and $m$ an element of
   the semigroup $\check{\sigma}\cap M$. Suppose that $\delta^\perp \cap
   \check{\sigma}$ is the minimal face of $\check{\sigma}$ containing $m$,
   where $\delta$ is a face of $\sigma$. Consider the open subset 
   $W=\{\gamma\in L(\sigma, N)\,|\,\gamma(m)\in U\}$ of $L(\sigma, N)$. 
   If $U$ is contained in $\R$, then $F^{-1}(W)$ consists of an open part 
   $H\subseteq N_{\R}$ and some strata at infinity. Namely, $F^{-1}(W)=
   \bigsqcup_{\tau\leq\delta}\pi_\tau(H+\delta)=\overline{H}_{\delta}$, 
   and this set is open in $\bigsqcup_{\tau\leq \sigma} (N/ N_{\tau})_{\R}$.
   Assume that $U=(a,+\infty]$. Let $W_0$ be the open subset of 
   $L(\sigma, N)$ corresponding to the interval $U_0=(a,+\infty)$. Then 
   $F^{-1}(W_0)$ contains the open halfspace
   $H=\{\gamma\in N_{\R}\,|\, \gamma(m)>a\}$ of $N_{\R}$ and we can write: 
   $$F^{-1}(W)=\overline{H}_{\sigma}\subseteq
   \bigsqcup_{\tau\leq \sigma} (N/ N_{\tau})_{\R}.$$
   This set is open by the definition of the topology on 
   $\bigsqcup_{\tau\leq \sigma} (N/ N_{\tau})_{\R}$.

   Let $\overline{U}_{\delta}$ be an open subset of 
   $\bigsqcup_{\tau\leq \sigma} (N/ N_{\tau})_{\R}$ of the form 
   \eqref{E:opensetoflv}. We show that the set $F(\overline{U}_{\delta})$ 
   is also open. First note that it suffices to assume that the open set $U$ 
   in $N_{\R}$ is an intersection of finite number of open half-spaces 
   $U^+(m,a)=\{\gamma\,|\,\gamma(m)>a\}$ or $U^-(m,a)=\{\gamma\,|\,\gamma(m)<a\}$ 
   for \emph{some elements $m$ of $M$ contained in the interior of 
   $\check{\sigma}$} (recall that the cone $\check{\sigma}$ has nonempty interior, 
   thus it contains a basis for $M$) and for some $a\in\R$. Using this, we can 
   further reduce to the case when $U$ is actually one half-space of the form 
   $U^+(m,a)$ or $U^-(m,a)$. In the first case:
   $$F(\overline{U}_{\delta})=F(\overline{U^{+}_{\delta}(m,a)})=$$
   $$\{\gamma\in L(\sigma,N)\,|\,\gamma(m)\in (a,+\infty]\}\setminus
   \bigcup_{m'\in \delta^\perp\cap\check{\sigma}\cap M} 
   \{\gamma\in L(\sigma,N)\,|\,\gamma(m')=+\infty\}.$$
   The set $\{\gamma\,|\,\gamma(m')=+\infty\}$ is closed in $L(\sigma,N)$,
   and the union that we subtract in the formula above is in fact finite.
   Therefore the set $F(\overline{U}_{\delta})$ is open. In the second case:
   $$F(\overline{U}_{0})=F(\overline{U^{-}_{0}(m,a)})=
   \{\gamma\in L(\sigma,N)\,|\,\gamma(m)<a\}$$
   if $\delta=\{0\}$ and:
   $$F(\overline{U}_{\delta})=F(\overline{U}^{-}_{\delta}(m,a))=L(\sigma,N)
   \setminus\bigcup_{m'\in \delta^\perp\cap\check{\sigma}\cap M} 
   \{\gamma\in L(\sigma,N)\,|\,\gamma(m')=+\infty\}$$
   if $\delta\ne \{0\}$. Again we conclude that $F(\overline{U}_{\delta})$ is
   open.
\end{proof}

Consider now a rational fan $\Delta$. By analogy with $L(\sigma, N)$, define:
\begin{equation}\label{linvar}
   L(\Delta,N)= \zcal_{\ri}(\Delta, N) = \bigsqcup_{\tau \in \Delta} 
   (N/ N_{\tau})_{\R}, \ \overline{\Delta}=\bigsqcup_{\tau\in\Delta} \overline{\tau}.
\end{equation}

The following result generalizes Proposition~\ref{P:opensetoflv}. Its 
proof is also left to the reader. 

\begin{proposition}
   The sets of the form $\overline{U}_{\delta}$ (as in \eqref{E:opensetoflv}),
   when $U$ is any open subset of $N_{\R}$ and $\delta$ is a cone of the fan 
   $\Delta$, form a basis of open sets for a topology on $L(\Delta,N)$.
\end{proposition}

The following definition is to Definition \ref{linvardef} what the Definition of 
the toric variety associated to a fan is to that of an affine toric variety: 

\begin{definition} \label{linvarglob}
  Let $N$ be a lattice and $\Delta$ a rational polyhedral fan in $N_{\R}$. 
  The topological space $L(\Delta, N)$, endowed with the natural
   continuous action $N_{\R} \times L(\Delta, N) \rightarrow L(\Delta,
   N)$ extending the action of $N_{\R}$ on itself by translations
   is called {\bf the linear variety} associated to the
   pair $(N, \Delta)$. We say that the closure of $\Delta$ in $L(\Delta, N)$ 
   is the {\bf extended fan} $\overline{\Delta}$. 
\end{definition}

The fan $\Delta$ determines a fan $\Star(\tau)$  in every 
stratum $(N/N_{\tau})_{\R}$ of the linear variety $L(\Delta,N)$. The cones of 
$\Star(\tau)$ are the projections of those cones $\delta$ of $\Delta$ which 
contain the cone $\tau$ as a face. In the case of an affine linear variety 
$L(\sigma,N)$, the system of fans $\{\Star{\tau}\}_{\tau\leq\sigma}$ has the 
following interpretation:

\begin{proposition}
The system of fans $\{\Star{\tau}\}_{\tau\leq\sigma}$ gives a stratification of 
the subspace $\overline{\sigma}=\overline{(\sigma,N)}$ of the affine linear
variety $L(\sigma,N)$:
$$\overline{\sigma}=\bigsqcup_{\tau\leq\sigma} \Star(\tau).$$
\end{proposition}

Notice that the subspace $\overline{\sigma}^\circ$ consists of the interior
$\mathring{\sigma}$ of $\sigma$ and all projections of $\mathring{\sigma}$ to the 
strata $(N/N_{\tau})_\R$ and it is indeed the interior of $\overline{\sigma}$ in 
the usual topological sense. Moreover, if $\tau$ is a face of $\sigma$, we define 
the \emph{relative interior $\overline{\tau}^\circ$ of $\tau$ inside 
$\overline{\sigma}$} to be the union of the usual relative interior 
$\mathring{\tau}$ of $\tau$ and all projections of $\mathring{\tau}$ to the 
strata $(N/N_{\rho})_\R$, where $\rho$ is a face of $\tau$.

We could have easily avoided using the lattice $N$ and the rationality of the 
fan $\Delta$ for the construction of the linear variety $L(\Delta)$ (for instance, 
we could have defined $L(\sigma)$ as the set of semigroup morphisms from 
$\check{\sigma}$ to $\overline{R}$ which are equivariant under the natural action 
of $\R^*$). In fact, these discrete data determine an additional {\em integral 
structure} on the linear variety, as a particular case of Remark \ref{remlin}, (2):

\begin{proposition}
   If $\Delta$ is a rational polyhedral fan, then every stratum 
   $(N/N_{\tau})_{\R}$ of the linear variety $L(\Delta,N)$ carries a
   lattice $N/N_{\tau}$ in such a way that the natural action of $N$ on
   itself extends canonically to an action by addition of $N$ on all the lattices
   $N/N_{\tau}$, $\tau\in\Delta$.
\end{proposition}

If we consider as starting data of the construction a fan of semigroups $\mathcal{S}$ 
(see Definition \ref{sgfan}), instead of 
simply the underlying fan of cones $\Delta$, the supplementary structure induced on 
$L(\Delta,N)$ is the knowledge, for each stratum at infinity $(N/N_{\tau})_{\R}$, 
of the semigroup $\Gamma_{\tau}$, seen as a special additive semigroup of linear 
functions on $N$.

\medskip
\section{Valuation spaces}

In this section we present the material from valuation theory we need in the sequel. 
We will work only with valuations taking values in the extended real line $\ri$. 
\medskip

\begin{definition} \label{ringval}
   Let $R$ be a ring. A \textbf{real (ring) valuation} on $R$ is a map 
   $R \overset{w}{\longrightarrow} \ri$ such that: 
 \begin{enumerate}
    \item $w$ is a morphism of semigroups from $(R,\cdot)$ to $(\ri,+)$.

    \item $w(0)= \infty$ and $w(1) = 0$.

    \item $w(f + g) \geq \min \{ w(f), w(g) \}$, for all $f,g \in R$.  
 \end{enumerate}
\end{definition}
The \emph{trivial valuation} is the valuation on an integral domain which
vanishes identically on $R\setminus\{0\}$.

In general, valuations take values in arbitrary totally ordered abelian groups 
extended by $\infty$ (see Zariski and Samuel's book \cite{ZS 60}, as well as 
Vaqui{\'e}'s introductory text \cite{V 00}). As we will not use that level of
generality, in the sequel by \emph{valuation} we mean a real ring valuation. 

The following is an immediate consequence of the definition:

\begin{lemma} \label{imcomp}
  If $f, g \in R$, $w$ is a valuation of $R$, and $w(f) \neq w(g)$, then $w(f + g)= 
  \min \{ w(f),w(g) \}$. 
\end{lemma}

\begin{definition}
   We denote by $\V(R)$ the set of valuations on $R$, endowed with the
   topology of pointwise convergence of maps from $R$ to $\ri$. We call it
   {\bf the valuation space of} $R$.  
\end{definition}

Recall that the topology of pointwise convergence is generated by the subsets of 
the form:
$$U_f=\{w_S\in\V(S)\,|
   \,w_S(f)\in U\}, \mbox{ for some } f\in S \mbox{ and some open } U\subset \ri,$$
in the sense that its open sets are arbitrary unions of finite intersections of such 
sets (one says that these sets form a {\em subbasis} of the topology). 

Any morphism of rings $\phi\colon S\to R$ induces, by composition, a function between 
the associated valuation spaces:
\begin{equation}\label{indmap}
\V(R)\overset{\V(\phi)}{\longrightarrow} \V(S),\quad 
\V(\phi)(v)=v\circ \phi.
\end{equation}

\begin{proposition} \label{valfunc}
   The function $\V(\phi)$ is continuous. Therefore, $\V$ defines a contravariant
   functor from the category of rings to the category of topological spaces. 
\end{proposition}

\begin{proof}
   Consider an arbitrary subbasic open subset $U_f$ of $\V(S)$, where $f \in S$ and 
   $U$ is open in $\ri$.  Then its preimage: 
   $$(\V(\phi))^{-1}(U_f)=\{w_R\in\V(R)\,|\,w_R(\phi(f))\in U\} = U_{\phi(f)}$$
   is, by definition, an open subset of $\V(R)$. This shows that our map $\V(\phi)$ 
   is continuous. 
\end{proof}

\begin{definition}
We say that a valuation $w\in \V(R)$  
{\bf is centered in} $R$ if $w(f)\geq 0$ for every $f\in R$. 
In this case the {\bf center} of the valuation $w$ is the prime ideal 
$\{f\in R\,|\,w(f)>0\}$. The {\bf home} of the valuation $w$ is the prime
ideal $\{f\in R\,|\,w(f)=\infty\}$.
\end{definition}

The home of a valuation is characterized by the following lemma:
 
\begin{lemma} \label{homechar}
  A  valuation $w\in\V(R)$ is the preimage of a valuation $\overline{w} \in 
  \V(R/I)$ if and only if the ideal $I$ is contained in the home of $w$.
\end{lemma}

\begin{remark}
 Classically (see Zariski and Samuel's book \cite{ZS 60}), the definition of
 valuations requires that non-zero elements take values in $\R$. Therefore, a 
 valuation in the extended sense which we use in this paper is simply obtained by
 pulling back a classical valuation from a quotient ring. We need such extended 
 valuations, as they may appear as limits of classical ones. Since we work with 
 valuations centered in a ring $R$, we do not need to add more points to 
 the valuation space, as the next proposition shows. 
\end{remark}

\begin{proposition} \label{compactness}
   The space $\V_{\geq 0}(R)$ of valuations centered in $R$ is compact. 
\end{proposition}

\begin{proof}
  By Tychonoff's theorem (see for instance \cite[Section 1-10]{HY 61}), the space 
  $[0, \infty]^R$ is compact when endowed with the product topology. Therefore it 
  is enough to prove that $\V_{\geq 0}(R)$ is closed inside $[0, \infty]^R$. But 
  any function in $[0, \infty]^R$ which is a limit of valuations is also a 
  valuation. Indeed, the axioms of Definition \ref{ringval} depend on at most two 
  elements of $R$, and those equalities or inequalities are preserved by the limit 
  process.
\end{proof}

\begin{remark}
  The previous proof is similar to Zariski's proof of the quasi-compactness of 
  the Rieman-Zariski space $S$ of a field extension $K/k$ (see \cite[Ch. VI, 
  Sect. 17, Theorem 40]{ZS 60}). By definition, the points of this space are the 
  Krull-valuation rings of $K$ containing $k$. The topology of $S$ is obtained by 
  taking as basis of open sets the subsets of valuation rings which contain 
  a given finitely generated subring of $K$ containing $k$. In order to get the 
  announced quasi-compactness, Zariski embeds $S$ in the space of maps from $K$ 
  to $\{-, 0, +\}$ (associating to each element of $K$ the sign of its value). 
  Then he uses Tychonoff's theorem for the space $\{-, 0, +\}^K$. A subtle point 
  here is that in order to apply Tychonoff's theorem we must use the discrete 
  topology on $\{-, 0, +\}$ (making it Hausdorff compact), but in order to get the
  correct topology on $S$ we have to consider a weaker topology (having the full 
  set and $\{0, + \}$ as basis of open sets), which is non-Hausdorff. 
\end{remark}

In the sequel, we will need to work with special subspaces of the
valuation space of a ring. In addition to the spaces $\V_{\geq 0}(R)$ introduced in
Proposition~\ref{compactness}, the main types of subspaces we need are
described in the next two definitions. 

\begin{definition} \label{valrel}
  Let $S \overset{\phi}{\longrightarrow} R$ be a morphism of rings and
  let $w_S \in \V(S)$ be a fixed valuation. Denote by 
  $\V_{(\phi,  w_S)}(R) \subset \V(R)$ the set of valuations $w_R$ on $R$ 
  such that
  $\V(\phi)(w_R)= w_S$. We call it {\bf the valuation space of $R$
    relative to} $(\phi, w_S)$. 
  When $S$ is a subring of $R$ and $S \overset{\phi}{\hookrightarrow}
  R$ is the inclusion morphism, we also write: $\V_{(S, w_S)}(R) :=  
  \V_{(\phi,  w_S)}(R)$, and we call it {\bf the valuation space of $R$ 
  relative to} $(S, w_S)$. 
\end{definition}

\begin{remark} \label{explberk}
  When $S$ is a subfield of $R$ and $w_S$ is a valuation such that $S$ is 
  complete with respect to the associated norm $e^{-w_S}$, the relative 
  valuation space $\V_{(S, w_S)}(R)$ is precisely the underlying topological space 
  of the {\em Berkovich analytic space} associated to $\mbox{Spec} R$ 
  (see \cite{B 90}). 
  One may consult Gubler \cite{Gub 11} for relations between Berkovich 
  analytification and tropicalization. 
\end{remark}

\begin{definition} \label{relideal}
  Let $\mathfrak{p}$ be a prime ideal of $R$. Denote by $\V(R, \mathfrak{p})$
  the subspace of $\V(R)$ consisting of all valuations centered in $R$ and whose 
  center is precisely  $\mathfrak{p}$. Call it {\bf the valuation space of $R$ 
  relative to $\mathfrak{p}$}.  
\end{definition}

Any valuation $w\in \V(R,\mathfrak{p})$ extends to the localization 
$R_\mathfrak{p}$ by setting $w(f/g)=w(f)-w(g)$ for $f/g\in R_\mathfrak{p}$ 
and $f \in R, \ g \in R \setminus \mathfrak{p}$. 
Moreover, in this way we get a valuation from 
$\V(R_\mathfrak{p},\mathfrak{p}R_\mathfrak{p})$. Thus the spaces 
$\V(R,\mathfrak{p})$ and $\V(R_\mathfrak{p},\mathfrak{p}R_\mathfrak{p})$ are 
naturally homeomorphic. 

As a particular case of the previous definition: 

\begin{definition} \label{Locval}
If $(R,\mathfrak{m})$ is a local ring, we call $\V(R,\mathfrak{m})$ the 
{\bf space of local valuations} of $(R,\mathfrak{m})$. 
\end{definition}

We recall that 
\emph{a local morphism} $(S, \mathfrak{n}) \overset{\phi}{\rightarrow} 
     (R, \mathfrak{m})$ between local rings is a morphism of rings such that 
     $\phi^{-1}(\mathfrak{m})= \mathfrak{n}$.
As a local analog of Proposition \ref{valfunc}, we have:

\begin{proposition} \label{valfuncloc}
  Let $(S, \mathfrak{n}) \overset{\phi}{\rightarrow} (R, \mathfrak{m})$ be a
  local morphism of local rings. Then the canonical map $\V(R, \mathfrak{m})
  \overset{\V(\phi)}{\longrightarrow} \V(S, \mathfrak{n})$ is continuous.  
  Therefore, taking valuation spaces \emph{defines a contravariant
  functor from the category of local rings and local morphisms to the category of
  topological spaces}.  
\end{proposition}

Given a valuation $w$ on a ring $R$, we define the associated
value of an ideal $I$:
  \begin{equation} \label{valid}
      w(I) : = \inf \{ w(f) \: | \: f \in I\}.
  \end{equation} 
When $R$ is Noetherian and $w$ is nonnegative on $R$, any ideal is finitely 
generated, and the infimum is achieved, due to the following lemma:

\begin{lemma} \label{valmin}
  Suppose that the ideal $I$ of the ring $R$ is generated by
  $f_1,...,f_r$ and let $w \in \V(R)$ be a valuation center in $R$. Then:
    $$  w(I) = \min\{ w(f_1) ,\dotsc, w(f_r)\}.$$
\end{lemma}
\begin{proof}
This lemma follows directly from the definition of valuations and from
nonnegativity of $w$ on $R$.
\end{proof}

Now, let $(R,\mathfrak{m})$ be a Noetherian local ring and let 
$(\hat{R}, \hat{\mathfrak{m}})$ be its completion with respect to $\mathfrak{m}$.

\begin{lemma} \label{extcompl}
    Let $w \in \V(R, \mathfrak{m})$. 
    Then $w$ is continuous with respect to the $\mathfrak{m}$-adic topology. It
    may therefore be extended by continuity to a valuation in $\V(\hat{R}, 
    \hat{\mathfrak{m}})$. 
\end{lemma}

\begin{proof}
  Since the ring $R$ is Noetherian, Krull's theorem implies that
  it is separated in its $\mathfrak{m}$-adic topology, that is,  
   $\cap_{n \in \N} \ \mathfrak{m}^n = 0$ 
   (see \cite[Corollary 10.18]{AM 69}).  Using Lemma
  \ref{valmin}, we see that:  
    $$m_0 := w(\mathfrak{m}) \in (0, \infty].$$ 
  The same lemma implies that $w(\mathfrak{m}^n)= n\cdot m_0$, for
  all $n \in \N$. 

  Consider any  $f \in R$. By the definition of the $\mathfrak{m}$-adic
  topology, the sets $(f + \mathfrak{m}^n)_{n \in \N}$ form a basis of
  neighborhoods of $f$. We consider now two cases, according to the value
  of $w(f)$. 

  $\bullet$ First, \emph{suppose that $w(f) \neq \infty$}. Then, there exists
  $n_0 \in \N$ such that $n\cdot m_0 > w(f)$ for any $n \geq
  n_0$. For such a value of $n$, consider any  $g \in f +
  \mathfrak{m}^n$, and we write $g= f + \mu$, with $\mu \in
  \mathfrak{m}^n$. Therefore $w(f) < w(\mu)$, which by Lemma \ref{imcomp}
  implies that  $w(g)= w(f)$. Thus, $w$ is constant in the neighborhood 
  $f + \mathfrak{m}^n$ of $f$, and so it is continuous at $f$. 
  
  $\bullet$ Secondly, \emph{suppose that $w(f) = \infty$}. We split
  this situation into two subcases:
     
      \begin{enumerate}
         \item \emph{If $m_0= \infty$}, then we see that $w(g)=\infty$ for 
                   any $g \in f + \mathfrak{m}$, which implies again that $w$
                   is continuous  
                   in a neighborhood of $f$.  
         
         \item \emph{If $m_0 \in (0, \infty)$}, then we see that $w(g)
           \geq n m_0$  
                   for all $g \in f + \mathfrak{m}^n$, which shows again
                   that $w$ is  
                   continuous at $f$. 
      \end{enumerate}
\end{proof}

As a consequence, we can canonically identifify the valuation space of a local 
Noetherian ring with the one of its completion.

\begin{corollary}\label{C:completion}
   The inclusion $(R, \mathfrak{m}) \overset{i}{\hookrightarrow} (\hat{R},
   \hat{\mathfrak{m}})$  induces an isomorphism of local valuation spaces:
   $\V(\hat{R}, \hat{\mathfrak{m}}) \overset{\V(i)}{\simeq} \V(R,\mathfrak{m})$. 
\end{corollary}

\begin{proof} By the previous lemma applied to $(\hat{R}, \hat{\mathfrak{m}})$, 
   any $\hat{w} \in \V(\hat{R}, \hat{m})$ is continuous for the 
   $\mathfrak{m}$-adic topology. Therefore it is determined by its restriction to 
   $R$, which proves the injectivity of $\V(i)$. The surjectivity follows 
   from Lemma~\ref{extcompl}.
\end{proof}

The next lemma shows that we can reduce the study of the valuation space of an 
affine scheme to those of the irreducible components of the associated reduced 
scheme. 

\begin{lemma} \label{redval}
  Let $R$ be a ring and $R \overset{\rho}{\longrightarrow} R_{red}$ be its 
  reduction morphism (that is, the morphism of quotient by its nilradical). 
  Then, the map $\V(\rho)$ induced by $\rho$ is a homeomorphism of $\V(R)$ 
  and $\V(R_{red})$. If $R$ is reduced and the $(\mathfrak{p}_i)_{i \in I}$ are 
  the prime ideals of the primary decomposition of the zero ideal, then 
  $\V(R)= \bigcup_{i \in I} \V(R/ \mathfrak{p}_i)$, that is, the valuation space 
  of $\spec R$ is the union of the valuation spaces  of its irreducible 
  components. The same holds for the space $\V(R,\mathfrak{p})$ of valuations
  relative to a prime ideal $\mathfrak{p}$ of $R$.
\end{lemma}

\begin{proof}
  The result follows from Lemma \ref{homechar} and the fact that if $w$ is a 
  valuation of $R$, then the home of $w$ contains necessarily at least 
  one of the ideals $\mathfrak{p}_i$ of the primary decomposition of $\{0\}$. 
\end{proof}

\medskip
\section{An affine theory of tropicalization}\label{S:defoftrop}

In this section we describe our proposed framework for a theory of tropicalization 
which both generalizes the existing one of tropicalization of subvarieties of tori 
and allows in particular to tropicalize (even formal) germs of subvarieties of 
toric varieties. We stress also the functorial properties of our notion of 
tropicalization. The qualificative ``\emph{affine}'' in the title of this section
is explained in Remark \ref{qualaffine}. In Section~\ref{S:extdef} we describe a 
more general framework for functorial tropicalization.
\medskip

In the sequel, $(\Gamma, +)$ denotes an arbitrary \emph{affine semigroup} and  
$(R, +, \cdot)$ a commutative ring. 
Consider a \emph{morphism of semigroups}:
$$(\Gamma,+) \overset{\gamma}{\longrightarrow} (R, \cdot).$$ 
This is the same as giving a morphism of rings $\Z[\Gamma]
\overset{\gamma}{\longrightarrow} R$, and, thus, a morphism of schemes 
$\spec(R) \overset{u}{\longrightarrow} \spec(\Z[\Gamma])$.
   
If $w\in \V(R)$, we have $w \circ \gamma \in \Hom_{Sg}(\Gamma, \ri)$.
By formula (\ref{aftt}), we see that $w \circ \gamma \in L(\sigma(\Gamma),
N(\Gamma))$. We can define: 
\begin{equation} \label{deffi}
   \mathcal{V}(R) \overset{\Phi_{\gamma}}{\longrightarrow}
       L(\sigma(\Gamma), N(\Gamma)). 
\end{equation}

It is a routine exercise to check:

\begin{lemma}\label{L:contoftrop}
  The map $\Phi_\gamma$ is continuous with respect to the topologies of
  pointwise convergence on $\V(R)$ and $L(\sigma(\Gamma), N(\Gamma))$.
\end{lemma}

In the next definition, we allow $\mathcal{W}$ to be \emph{any}
subset of $\V(R)$. In the sequel we will be particularly interested in
these subsets of valuation spaces relative to valuations defined on subrings
(see Definition \ref{valrel}) or to ideals (see Definition
\ref{relideal}).

\begin{definition} \label{deftrop}
  Let $\W$ be a subset of the valuation space $\V(R)$. The closure in 
  $L(\sigma(\Gamma), N(\Gamma))$ of the image $\Phi_{\gamma}(\mathcal{W} )$ 
  is called {\bf the (global) tropicalization of } $\mathcal{W}$ {\bf with 
  respect to the semigroup morphism} $\gamma$, and it is denoted by
  $\trop(\mathcal{W}, \gamma)$ or $\trop(\mathcal{W}, u)$.
\end{definition}

\begin{remark} \label{caseEKL}
   One of the definitions of tropicalization of subvarieties of tori proposed by 
   \cite{EKL 06} corresponds to the case where $\Gamma$ is the lattice of 
   exponents of monomials of the torus $\mbox{Spec} (K[\Gamma])$,
   $R = K[\Gamma] / I$ for an ideal $I$ of $K[\Gamma]$, $\gamma$ is the 
   composition $\Gamma \hookrightarrow K[\Gamma] \rightarrow K[\Gamma] / I$, 
   and $\mathcal{W} = \mathcal{V}_{K, v}(R)$ is the set of valuations on $R$
   extending a valuation $v$ of the field $K$.
\end{remark}

\begin{remark}
   Our definition is indeed more general than the one explained in the previous 
   remark. More precisely, if $\Gamma$ is an affine semigroup, and 
   $\gamma\colon \Gamma\to (R,\cdot)$ is a morphism of semigroups, then $\gamma$ 
   does not extend in general to a morphism from the associated lattice 
   $M(\Gamma)$ of $\Gamma$. In fact, such an extension exists if and only if the 
   image of $\gamma$ is contained in the group of units of $(R, \cdot)$. 
\end{remark}

\begin{remark} \label{specnot}
  When $K$ is a field and $I$ is an ideal of the ring $K[\Gamma]$, we set:
  $$\trop(I) := \trop(\V(R), \gamma),$$
  where $R := K[\Gamma] / I$ and $\gamma : \Gamma \to R$ is the morphism of 
  semigroups induced by the  quotient map $K[\Gamma] \to K[\Gamma] / I$. 
  When $\Gamma$ is saturated, this agrees with the notion of tropicalization of a 
  subvariety of a normal affine toric variety introduced by Payne \cite{P 08}. In 
  fact, these tropicalizations may be glued to produce a tropicalization of an 
  arbitrary subscheme of a general (not necessarily normal) toric variety. In 
  this case, if the toric variety is clear from the context and $X$ is a subscheme 
  of it, we denote this tropicalization simply by $\trop(X)$. 
\end{remark}

Next, we define the notion of \emph{local tropicalization}.
Denote by $\sigma$ the cone 
$\sigma(\Gamma)  \subset N(\Gamma)_{\R}$. Let $(R,\mathfrak{m})$ be a local ring 
and $\gamma\colon\Gamma\to R$ be a \emph{local} morphism of semigroups, i.e., 
$\gamma^{-1}(\mathfrak{m})=\Gamma^+$. Recall from Definition \ref{Locval} 
that by a \emph{local valuation} of $R$ we mean a valuations nonnegative on $R$ 
and positive on $\mathfrak{m}$, that is, an element  of the space 
$\V(R,\mathfrak{m})$. Note that, by \eqref{tc}, the map $\Phi_\gamma$ considered 
above sends the space $\V(R,\mathfrak{m})$ into the extended cone $\overline{\sigma}$ 
(see Definition \ref{linvardef}).

\begin{definition}\label{D:loctrop}
  The {\bf local positive tropicalization} of $\gamma$, 
  denoted \linebreak
  $\ptrop(\V(R,\mathfrak{m}),\gamma)$ or simply 
  $\ptrop(\gamma)$, is the closure of the image of the map $\nu(R, \mathfrak{m}) 
  \stackrel{\Phi_\gamma}{\rightarrow} L(\sigma, N) $ in the relative interior 
  $\overline{\sigma}^\circ$ of the space $\overline{\sigma}$.
\end{definition}

Notice that in this definition we only consider those valuations of $R$ which have 
as center the closed point of $\spec R$. Instead, if we only require that the 
valuations have a center on $\spec R$, possibly smaller than $\mathfrak{m}$, 
we get another version of local tropicalization, which will also be important 
in the sequel: 

\begin{definition}\label{D:nonnegtrop}
  The {\bf local nonnegative tropicalization} of $\gamma$, 
  denoted $\nntrop(\gamma)$, is the image in the extended cone $\overline{\sigma}$
  of the map $\Phi_\gamma$ applied to all valuations of $R$ having
  a center on $\spec R$, that is, all nonnegative valuations of $R$.
\end{definition}

The following proposition states direct consequences of the definitions of the 
two kinds of local tropicalizations:

\begin{proposition}\label{P:closedtrop} 
  Let $\Gamma$ be an arbitrary affine semigroup and $\sigma = \sigma(\Gamma)$. 
  \begin{itemize}
   \item[(i)] The local nonnegative tropicalization is a closed subset of
    $\overline{\sigma}$.
   \item[(ii)] If the set $\gamma(\Gamma^+)\subseteq\mathfrak{m}$, where
    $\Gamma^+$ is the maximal ideal of $\Gamma$, generates $\mathfrak{m}$ (as
    an ideal of the ring $R$), then the image of the map $\Phi_\gamma\colon
    \nu(R, \mathfrak{m})\to L(\sigma,N)$ coincides with $\nntrop(\gamma)\cap
    \overline{\sigma}^\circ$. In particular, this image is closed in
    $\overline{\sigma}^\circ$ and $\ptrop(\gamma)=\nntrop(\gamma)\cap
    \overline{\sigma}^\circ$.
  \end{itemize}
\end{proposition}

\begin{proof}
  Statement (i) follows from Proposition~\ref{compactness} and 
  Lemma~\ref{L:contoftrop}. Statement (ii) follows from (i) and the definition of
  positive tropicalization. In general, the question of closedness of the
  image of $\Phi_\gamma$ is subtler and connected to the problem of extension
  of valuations, see Section~\ref{S:extval}.
\end{proof}

\begin{remark} \label{specnotloc} (Local analog of Remark~\ref{specnot}). 
When $K$ is a field, $\Gamma$ is an affine pointed semigroup 
and $I$ is an ideal of the ring $K[[\Gamma]]$ of formal power series 
with exponents in $\Gamma$ (discussed more carefully in Section 
\ref{S:powerseriesring}), we denote $\ptrop(I) := \ptrop(\gamma)$ and 
$\nntrop(I) : = \nntrop(\gamma)$.
\end{remark}

\begin{definition} \label{defloctropter}
  Let $(S, \mathfrak{n})$ be a local subring of $(R,\mathfrak{m})$, endowed with a
  valuation $w_S \in \V(S, \mathfrak{n})$.  Denote by $\V_{(S, w_S)}(R, 
  \mathfrak{m})$ the set of valuations in $\V(R, \mathfrak{m})$ which extend $w_S$,
  called {\bf the valuation space of} $(R, \mathfrak{m})$ {\bf relative to}
  $((S, \mathfrak{n}), w_S)$.
  Then $\ptrop(\V_{(S, w_S)}(R, \mathfrak{m}), \gamma)$ is called  
  {\bf the local positive tropicalization of the semigroup morphism} $\gamma$ 
  {\bf relative to} $((S, \mathfrak{n}), w_S)$. We denote it by 
  $\ptrop(R, (S, w_S), \gamma)$. 
\end{definition}

\medskip
Let us now discuss the functorial properties of our definition of tropicalization 
(both local and global). 

\begin{definition}
  Consider two semigroup morphisms $\Gamma_i
  \overset{\gamma_i}{\longrightarrow} (R_i, \cdot)$, for $i =1,2$. A
  {\bf morphism from $\gamma_1$ to $\gamma_2$} is a pair of maps:
  $$(\phi_H \in \Hom_{Rg}(R_1, R_2); \lambda_H \in \Hom_{Sg}(\Gamma_1,
  \Gamma_2))$$ making the following diagram commutative:
  $$\xymatrix{
      R_1  \ar[r]^{\phi_{H}}  & 
      R_2   \\
      \Gamma_1   \ar[u]^{\gamma_1}  \ar[r]_{\lambda_{H}} & 
      \Gamma_2   \ar[u]_{\gamma_2} }$$
  We denote by \emph{\bf SgRg} the category defined in this way, 
  and by \emph{\bf SgRgVal} the category whose objects are
  pairs $(\Gamma \overset{\gamma}{\longrightarrow} (R, \cdot), \W
  \subset \V(R))$ and whose morphisms are the morphisms of the
  category SgRg which respect the chosen subsets of the valuation
  spaces (that is, which send one into the other). 
\end{definition}

\begin{proposition}
  Let $(\Gamma_i
  \overset{\gamma_i}{\longrightarrow} (R_i, \cdot), \W_i)$, for $i =1,2$ be
  two objects of the category SgRgVal and $H$ a morphism from
  $(\gamma_1, \W_1)$ to $(\gamma_2, \W_2)$. Then $H$ induces a
  functorial linear map:
   $$\trop(\W_2, \gamma_2) \overset{\trop(H)}{\longrightarrow} 
         \trop(\W_1, \gamma_1),$$
  Moreover, positive tropicalizations are preserved by $\trop(H)$.
\end{proposition}

\begin{remark} \label{qualaffine}
  We call the theory developed in this section ``affine'', because we think 
  about the category SgRgVal as the analog of affine schemes. A next step, which 
  we do not develop in this paper (for some folow up on this matter, see 
  Sections~\ref{Toroidal} and \ref{S:extdef}), would be to glue objects of 
  this ``affine'' category into non-affine objects which may again be tropicalized.  
\end{remark}

\medskip
\section{Extensions of valuations}\label{S:extval}

In this section we address the problem of extending of a valuation from a ring to 
a bigger ring, in a generality suitable for our purposes. As an application, we 
show that under convenient hypothesis, the real part of the local tropicalization 
is necessarily non-empty (see Lemma~\ref{L:finitepart}), and that tropicalization 
is unchanged by passage to the normalization (see Corollary~\ref{C:normal}). 

\medskip
The following extension principle plays an important role in \cite{BG 84}: 
if $K\subseteq L$ is a field extension and $v$ is a real valuation of $K$, 
then $v$ can always be extended to a \emph{real} valuation $w$ of $L$, that is 
there is a real valuation $w$ of $L$ such that $w$ restricted to $K$ coincides 
with $v$. We now give a local version of this extension principle. Let 
$(R,\mathfrak{m})$ and $(S,\mathfrak{n})$ be two local rings such that 
$R\subseteq S$, $\mathfrak{n}\cap R=\mathfrak{m}$, and let $v$ be a local real 
ring valuation of the ring $R$, i.e., $v$ is nonnegative on $R$ and positive on 
the maximal ideal $\mathfrak{m}$. We address the following question: \emph{Does 
there exist a local real valuation $w$ of the ring $S$ such that $w$ restricted to 
$R$ coincides with $v$?}

As a first approach, we may assume that the given valuation $v$ is only 
nonnegative on $\mathfrak{m}$, and ask whether there exists an extension $w$ 
nonnegative on $\mathfrak{n}$. Geometrically, we consider only valuations 
centered at the maximal ideals of our local rings, or, if $v$ and $w$ are only 
nonnegative on maximal ideals, such that their centers (thought geometrically as 
irreducible subschemes) contain the 
special points of $\spec R$ and $\spec S$ respectively. The answer to this 
question is not always positive, as shown by the following simple example.

\begin{example}
  Let $R=K[[x,y]]$, $S=K[[s,t]]$ be two copies of the ring of formal power
  series in two variables over a field $K$, and assume that the inclusion of $R$ 
  into $S$ is given by the map $x\mapsto s$, $y\mapsto st$ (this corresponds to the
  blow up of a point in a plane). Let $v$ be a monomial valuation on $R$,
  trivial on $K$, and determined by $v(x)=1$, $v(y)=1$. Then $v$ cannot
  be lifted to a local valuation $w$ of $S$ because $w$ must take value $0$ on 
  $t$. If we set $v(x)=2$, $v(y)=1$, then it is impossible to find
  a nonnegative extension $w$, because $w(t)$ must be equal to $-1$.
\end{example}

Next, we derive some sufficient conditions for the extension principle to 
hold.

\begin{theorem}\label{T:extprinciple}
  Let $(R,\mathfrak{m})$ and $(S,\mathfrak{n})$ be two local rings, $R\subseteq S$, 
  $\mathfrak{n}\cap R=\mathfrak{m}$. Let $v$ be a real nonnegative ring valuation 
  on $R$, and assume that one of the following conditions holds:
  \begin{itemize}
    \item[a)] $S$ is an integral extension of $R$ (e.g. $S$ is a finite 
    $R$-module);
    \item[b)] $R$ and $S$ are Noetherian, complete with respect to the 
    $\mathfrak{m}$-adic and $\mathfrak{n}$-adic topologies, and \emph{(i)} $S$ is 
    flat over $R$, \emph{(ii)} the residue fields $R/\mathfrak{m}$ and 
    $S/\mathfrak{n}$ are naturally isomorphic, and \emph{(iii)} the fiber of the 
    scheme $\spec(S)$ over the maximal ideal $\mathfrak{m}$ of $\spec(R)$ is 
    reduced and irreducible, i.~e., the ideal $\mathfrak{m}S$ is prime in $S$.
  \end{itemize}
  Then, there is a real nonnegative valuation $w$ of the ring $S$ such that 
  $w$ restricted to $R$ coincides with $v$. If, moreover, $v$ is local 
  (that is, positive on $\mathfrak{m}$), then $w$ can also be chosen to be local.
  In fact, in case a) every valuation $w$ extending $v$ is
  nonnegative, and local if $v$ is local.
\end{theorem}

\begin{proof}
  First we prove the sufficiency of condition a). Let $\mathfrak{p}$ be 
  the home of  the valuation $v$. By basic properties of integral extensions
  (see, e.~g., \cite[Chapter 5]{AM 69}) there exists a prime ideal 
  $\mathfrak{q}$ of $S$ such that $\mathfrak{q}\cap R=\mathfrak{p}$. Then, 
  $S/\mathfrak{q}$ is an integral extension of $R/\mathfrak{p}$. By 
  Lemma~\ref{homechar}, $v$ defines a valuation $v'$ of the ring $R/\mathfrak{p}$, 
  and it suffices to extend the valuation $v'$ to $S/\mathfrak{q}$. This shows that 
  from the beginning we can assume that $S$ and $R$ are local domains and the home 
  of $v$ is $\{0\}$. Let $K(R)$ and $K(S)$ denote the fields of fractions of $R$ 
  and $S$ respectively, so $K(R)\subseteq K(S)$. The valuation $v$ can be defined 
  on $K(R)$ by the rule $v(a/b)=v(a)-v(b)$. As we have already mentioned at 
  the beginning of this section, valuations from fields can always be 
  extended, so let $w$ be any real valuation of the field $K(S)$ extending 
  $v$ from $K(R)$. Since $v$ is nonnegative on $R$, the valuation ring
  $S_w$ of $w$ contains $R$. On the other hand, the integral closure of $R$ in the 
  field $K(S)$ is the intersection of all valuation rings of $K(S)$ containing $R$ 
  (\cite[Corollary~5.22]{AM 69}), thus $S$ is contained in $S_w$ and $w$ is 
  nonnegative on $S$.

  Now assume that the valuation $v$ is local, and let $w$ be any 
  extension of it. We have just seen that $w$ is nonnegative on $S$.
  Consider the set $I$ of elements $x$ of $S$ which satisfy an integral dependence 
  relation:
  $$f(x)=x^n+r_1x^{n-1}+\dotsc+r_n=0$$
  with $r_1,\dots,r_n\in\mathfrak{m}$ and $n \in \N$. Fix such an $x$, and let
  $s\in S$. The element $s$ also satisfies an integral dependence relation:
  $$g(s)=s^m+a_1s^{m-1}+\dotsc+a_m=0,$$
  where $a_1,\dots,a_m\in R$. Let $s_1=s$, $s_2, \dotsc, s_m$ be all the roots
  of $g$ in the field $\overline{K(S)}$. Consider the polynomial:
  $$F(X)=(s_1\cdots s_m)^n \prod_{i=1}^{m} f\left(\frac{X}{s_i}\right).$$
  This is a monic polynomial in the variable $X$ and, moreover, its coefficients 
  are symmetric polynomials in $s_1, \dotsc, s_m$ with coefficients in 
  $\mathfrak{m}$. It follows that $F$ has coefficients in $\mathfrak{m}$, and,
  since $F(sx)=0$, $sx\in I$. Consider one more element $y\in I$. Let:
  $$h(y)=y^d+t_1y^{d-1}+\dotsc+t_d=0,$$
  where $t_1 ,\dotsc, t_d\in\mathfrak{m}$, be the corresponding integral dependence
  relation. Let $y_1=y ,\dotsc, y_d$ be the roots of $h$ in $\overline{K(S)}$. 
  Applying the same argument to the polynomial:
  $$H(X)=\prod_{i=1}^{d} h(X-y_i),$$
  we show that $x+y\in I$. Thus $I$ is an ideal of the ring $S$. Clearly 
  $I\cap R=\mathfrak{m}$. It follows from \cite[Proposition~4.2 and
  Corollary~5.8]{AM 69} that $I$ is $\mathfrak{n}$-primary and 
  if $s$ is any element of the maximal ideal $\mathfrak{n}$ of $S$, then
  $s^k\in I$ for some $k$ (in fact, we thus have $I=\mathfrak{n}$). But any 
  $x\in I$ should satisfy $w(x)>0$, because otherwise we would have: 
  $$w(x^n+ r_1x^{n-1}+\cdots+r_n)=
  \min\{w(x^n), w(r_1 x^{n-1}) ,\dotsc, w(r_n)\}=0.$$
  It follows that $w(s)>0$. As $s \in \mathfrak{n}$ is arbitrary, we see that $w$ 
  is also local.
  
  \medskip

  Now we prove the \emph{sufficiency} of condition b). By Theorem~\ref{T:flatext},
  we can find analytically independent elemnts $x_1,\dotsc, x_k\in S$ over $R$
  such that $S$ is a finite module over $R[[x_1,\dotsc,x_k]]$. First we have 
  to extend the valuation $v$ to the intermediate ring 
  $R'=R[[x_1,\dotsc,x_k]]$. For this we choose any positive real
  values $w'(x_1) ,\dotsc, w'(x_k)$ and for any $f=\sum_m a_m x^m
  \in R'$, $x^m=x_{1}^{m_1}\cdots x_{k}^{m_k}$, $a_m\in R$, we define:
  $$w'(f)=\min_{m}\{v(a_m)+m_1 w'(x_1)+\cdots +m_k w'(x_k)\}.$$
  We can easily check that this defines a nonnegative (local if
  $v$ is local) valuation $w'$ on the ring $R'$. Then by a) we can extend 
  $w'$ from $R'$ to $S$. This concludes the proof.
\end{proof}

The proof of the following result was communicated to us by Mark Spivakovsky.

\begin{theorem}\label{T:flatext}
  Let $R$ and $S$ be local rings satisfying the assumptions of condition b) 
  of Theorem~\ref{T:extprinciple}. Then, there exists a finite number of elements 
  $x_1 ,\dotsc, x_k$ of $S$ which are analytically independent over $R$, 
  such that the extension $R\subseteq S$ factorizes as: 
  $$R\subseteq R[[x_1,\dotsc,x_k]]\subseteq S,$$
  and $S$ is a finite module over $R[[x_1,\dotsc,x_k]]$.
\end{theorem}

\begin{proof}
  The rings $R$ and $S$ are local and Noetherian, hence they both have
  finite Krull dimension. Since $S$ is flat over $R$, $\dim S -\dim R=
  \dim S/\mathfrak{m}S$ (see \cite[Theorem~19 (2), p. 79]{M 70}).  We 
  denote this number by $k$. Let $x_1 ,\dotsc, x_k$ be elements of
  $\mathfrak{n}\setminus\mathfrak{m}$ whose images in $\mathfrak{n}(S/\mathfrak{m}S)$ 
  form a system of parameters. The fact that $x_1 ,\dotsc, x_k$ are 
  analytically independent over $R$ follows from the local criterion of flatness
  (\cite[Theorem~49 (4), p. 147]{M 70}), which says that $S$ is $R$-flat if and
  only if $S/\mathfrak{m}S$ is $R/\mathfrak{m}$-flat and the canonical maps:
  $$\gamma_n\colon (\mathfrak{m}^n/\mathfrak{m}^{n+1})  
  \bigotimes_{R/\mathfrak{m}R}(S/\mathfrak{m}S) \to 
  \mathfrak{m}^n S/\mathfrak{m}^{n+1}S$$
  are isomorphisms. Indeed, suppose that there is an analytic dependence relation: 
  \begin{equation}\label{E:andeprel}
    \sum_m a_m x^m =0,
  \end{equation}
  where $m=(m_1,\dotsc,m_k)$, $x^m=x_{1}^{m_1}x_{2}^{m_2}\cdots x_{k}^{m_k}$,
  $a_m\in R$. Denote by $n$ the smallest nonnegative integer such that
  $a_m\notin \mathfrak{m}^{n+1}$ for some $m$. Then relation~\eqref{E:andeprel}
  gives rise to a relation of the form: 
  $$\sum_{i=1}^{t} b_i f_i=0$$
  with $b_i\in \mathfrak{m}^n/\mathfrak{m}^{n+1}$, 
  $f_i\in R[[x_1,\dots,x_k]]/\mathfrak{m}R[[x_1,\dots,x_k]]$,
  which holds in $\mathfrak{m}^n S/\mathfrak{m}^{n+1} S$. Thus, the element:
  $$\sum_{i=1}^{t} b_i\otimes f_i \in (\mathfrak{m}^n/\mathfrak{m}^{n+1})
  \bigotimes_{R/ \mathfrak{m}R} (S/\mathfrak{m}S),$$
  where $b_i$ are as above and $f_i$ are now considered as elements of
  $S/\mathfrak{m}S$, is a nonzero element of the kernel of the canonical map 
  $\gamma_n$, but this contradicts the quoted criterion of flatness.

  Now we prove that the ring $S$ is finite over $R'=R[[x_1,\dotsc,x_k]]$.
  Note that $R'$ is also a Noetherian complete local ring with 
  maximal ideal $\mathfrak{m}'$ generated by $\mathfrak{m}$ and $x_1$, $\dots$,
  $x_k$. Consider the extension $R/\mathfrak{m}\subseteq S/\mathfrak{m}S$ of 
  complete local rings. Note that $S/\mathfrak{m}S$ is a domain and its residue 
  field is isomorphic to $R/\mathfrak{m}$, this follows from assumptions b) (ii) 
  and (iii) of Theorem~\ref{T:extprinciple}. Then we can apply 
  \cite[Corollary~31.6, p. 109]{N 62}, which states that if $x_1 ,\dotsc, 
  x_k$ is a system of local parameters for $S/\mathfrak{m}S$, then $S$ is
  finite over $(R/\mathfrak{m})[[x_1,\dotsc,x_k]]$. But then $S/\mathfrak{m}S$ is 
  also finite over $R'$. By \cite[Theorem~30.6, p. 105]{N 62} we conclude that
  $S$ is a finite module over $R'$. This finishes the proof of 
  Theorem~\ref{T:flatext}.
\end{proof}

\medskip

Theorem~\ref{T:extprinciple} may be expressed geometrically in the following 
way:  \emph{Any (flat) deformation of an algebroid germ over another such germ 
may be obtained as a finite (ramified) covering of the product of the base germ 
with a smooth algebroid variety.}

Corollary~\ref{C:completion} implies that when working with local 
tropicalization we can always pass to complete rings. Note also that the
positive tropicalization is never empty, since any local ring possesses the
trivial local valuation $v$, where $v(r)=0$ if $r\notin\mathfrak{m}$ and
$v(r)=\infty$ if $r\in \mathfrak{m}$. Under rather general assumptions on $R$ and 
some natural restrictions on $\gamma$ the real part of the positive 
tropicalization is also nonempty.

\begin{lemma}\label{L:finitepart}
  Assume that $(R,\mathfrak{m})$ is a complete local Noetherian domain, $\Gamma$
  is an affine semigroup, and $\gamma\colon\Gamma\to R$ is any local semigroup 
  morphism such that no element of $\Gamma$ goes to $0$. Then 
  $\ptrop(\gamma)\cap \sigma\ne \emptyset$.
\end{lemma}

\begin{proof}
  By the Cohen structure theorem for complete local rings
  (\cite[Corollary~31.6, p. 109]{N 62}) we know that $R$ is a finite module over 
  a subring of the form $J[[x_1,\dots,x_k]]$, where $J$ is either a field or a discrete 
  valuation ring. In the first case, we choose $v$ to be the trivial valuation on 
  $J$. In the second case, let $v$ be the unique discrete valuation of $I$  
  such that $v(p)=1$ for the generator $p$ of the maximal ideal of $J$. Then,
  we can extend $v$ to $J[[x_1,\dots,x_k]]$ by assigning any positive values to
  $x_1$, $\dots$, $x_k$. By Theorem~\ref{T:extprinciple} a), this valuation can be
  extended to a valuation $w$ of $R$ with home $\{0\}$. This $w$ is a point 
  of $\ptrop{\gamma}$ contained in $\sigma$.
\end{proof}

In the next application of Theorem~\ref{T:extprinciple} we show, essentially, 
that the tropicalization does not change if we pass to the normalization.

\begin{lemma}\label{L:normal}
  Let $R$ be an integral domain, $\gamma\colon\Gamma\to R\setminus\{0\}$ a
  morphism from an affine semigroup $\Gamma$, and $S$ the integral closure of
  $R$ in its field of fractions $Q(R)$. Then there exists a unique extension
  $\bar{\gamma}\colon\sat(\Gamma)\to S$ of $\gamma$, and:
  $$\trop(\gamma,\V(R))=\trop(\bar{\gamma},\V(S)),$$
  where $\V(R)$ and $\V(S)$ denote the spaces of all valuations of $R$ and of $S$
  respectively. If $R$ and $S$ are both local, then also
  $\nntrop(\gamma)=\nntrop(\bar{\gamma})$ and $\ptrop(\gamma)=
  \ptrop(\bar{\gamma})$.
\end{lemma}

\begin{proof}
  Since none of the elements of $\Gamma$ goes to $0$, the morphism $\gamma$
  extends uniquely to a homomorphism from $M(\Gamma)$ to $Q(R)$. But the
  images of elements of $\sat(\Gamma)$ are obviously integral over $R$, thus
  they belong to $S$. Any valuation of $R$ extends to a valuation of $S$, and
  its values on $\sat(\Gamma)$ are uniquely determined by its values on
  $\Gamma$. Moreover, by Theorem~\ref{T:extprinciple}, any nonnegative valuation 
  of $R$ extends to a nonnegative valuation of $S$. This implies all the 
  equalities of tropicalizations. 
\end{proof}

\begin{corollary}\label{C:normal}
  Under the conditions of Lemma~\ref{L:normal}, assume that $I$ is an
  ideal of $R$ and $I$ is disjoint from the semigroup $\Gamma$.  Let
  $p\colon R\to R/I$ and $q\colon S\to S/SI$ be the canonical projections.
  Then:
  $$\trop(p\circ\gamma,\V(R/I))=\trop(q\circ\bar{\gamma},V(S/SI)).$$
  If $R$ and $S$ are both local, then $\nntrop(p\circ\gamma)=
  \nntrop(q\circ\bar{\gamma})$, and similarly for the positive
  tropicalization.
\end{corollary}
  
\begin{proof}
  It suffices to show that if $v$ is a valuation of $R$ such that the home of $v$
  contains $I$, then the home of any extension $\bar{v}$ of $v$ to $S$ contains 
  $SI$. But indeed, if $f$, $g\in R$ and $v(f)=v(g)=\infty$, then for any
  $a$, $b\in S$ we have $\bar{v}(af+bg)=\infty$.
\end{proof}

\begin{remark} \label{reasnn}
  In view of the previous results, the reader could wonder why we made the effort 
  to develop a general framework of tropicalization for non-necessarily saturated 
  affine semigroups. We see two reasons for this:
    \begin{itemize} 
      \item Even if we take a morphism 
  between two normal affine toric varieties (corresponding therefore to saturated 
  affine semigroups), the closure of its image is again toric, but it may be 
  non-normal. An example is given by the map from $\C$ to $\C^2$ defined by 
  $t \to (t^2, t^3)$, which is a parametrization of the cuspidal plane cubic. 
  As another example, consider the parametrization $\phi: \C^2 \to \C^3$ defined 
  by $(s,t) \to (x,y,z)=(st, s, t^2)$ of the Whitney umbrella, defined by the 
  equation $x^2 - y^2 z =0$ in $\C^3$.  
  
     \item  A morphism between two affine toric varieties does not necessarily 
  lift to a morphism between their normalizations. For instance, consider the 
  embedding of the singular locus of the Whitney umbrella $W$ (defined in the 
  previous example) into $W$. This singular locus $S$ is the $z$-axis, therefore 
  the embedding $S \hookrightarrow W$ may be described as a restriction of the 
  toric map $\C \to \C^3$ given by $u \to (x,y,z) = (0, 0, u)$. The morphism 
  $\phi$ of the previous example is a normalization map of $W$. The restriction 
  of $\phi$ to $\phi^{-1}(S)$ is a double covering $\C \to \C$, therefore the map 
  $S \to W$ does not lift to a map from $S$ (equal to its own normalization) to 
  the normalization $\C^2$ of $W$. 
   \end{itemize}
  \end{remark}

\medskip
\section{The formal toric rings $K[[\Gamma]]$}\label{S:powerseriesring}

In this section we explain basic properties of rings of formal power series 
over $K$ with exponents in pointed affine semigroups 
$\Gamma$. We call them ``{\em formal toric rings}'', as they are the completions 
of the rings of the affine toric varieties $\mbox{Spec} \ K[\Gamma]$ at the 
unique closed orbit. In fact,  till 
Corollary \ref{compldom} we deal with arbitrary affine semigroups (satisfying 
perhaps a technical condition, as in Lemma \ref{L:fc}). Then we restrict  
to the pointed ones.

\medskip

Let $K$ be a field and $\Gamma$ an affine  semigroup (see 
Definition~\ref{toricsg}). Whenever we want to use 
multiplicative notation for the elements of the semigroup $\Gamma$ (which happens 
when we look at them as {\em monomials}), we write $\chi^m$ instead of $m$. We will 
say that $m$ is the {\em exponent} of the {\em monomial} $\chi^m$.

Recall that $\Gamma^*$ denotes the  subgroup of invertible elements of $\Gamma$. 
They are related by the short semigroup exact sequence:
\begin{equation}\label{E:gammaseq}
0 \longrightarrow \Gamma^*\longrightarrow \Gamma\stackrel{p}{\longrightarrow} 
\Gamma'\longrightarrow 0,
\end{equation}
where $p$ is the quotient map of the semigroup $\Gamma$ by the subgroup
$\Gamma^*$. 

By Proposition \ref{quotaff}, if $\Gamma$ is saturated, then 
$\Gamma'$ is an affine semigroup. Since its
subgroup of units is trivial, $\Gamma'$ is pointed.

\begin{lemma} \label{splitting}
 Suppose that the affine semigroup $\Gamma$ is saturated. Then the morphism $p$ 
 admits a section and any section induces a splitting of ~\eqref{E:gammaseq}. 
\end{lemma}

\begin{proof}
  As explained in the proof of Proposition \ref{quotaff}, we have the following 
  commutative diagram in which the horizontal lines are short exact sequences:
    $$     \xymatrix
  {  0  \ar[r]  &  \Gamma^*  \ar@{^{(}->}[d] \ar[r] & \Gamma \ar@{^{(}->}[d] \ar[r]^p & 
         \Gamma' \ar@{^{(}->}[d]   \ar[r]  & 0   \\
     0 \ar[r]  & \Gamma^*   \ar[r] & M(\Gamma)   \ar[r]^{M(p)} & 
        M(\Gamma')   \ar[r]   & 0
  }.$$   
  As $M(\Gamma')$ is free, the surjective morphism of groups $M(p)$ admits a section 
  $\alpha$. This shows that the second exact sequence splits. Let us restrict 
  $\alpha$ to $\Gamma'$. We see immediately that $\alpha(\Gamma') \subset \Gamma$, 
  which shows that $\alpha$ is also a section of $p$. Define then the semigroup 
  morphism $\Gamma \stackrel{\Phi}{\to} \Gamma^* \times \Gamma'$ by the formula 
  $\Phi(a) := (a-\alpha(p(a)), p(a))$. It is a routine exercise to check that it 
  is an isomorphism of semigroups, and thus \eqref{E:gammaseq} splits indeed.
\end{proof}
  
The previous proof shows that  \eqref{E:gammaseq} splits once we have a section of 
$p$. This may happen also for non-saturated affine semigroups, as we see by 
starting from a product $\Gamma^* \times \Gamma'$ between a lattice $\Gamma^*$ and 
an arbitrary pointed affine semigroup $\Gamma'$. But such sections do not 
necessarily exist, as illustrated by Examples \ref{nonsatex} and 
\ref{nonsatexbis}.

\begin{definition}
The set of formal infinite sums:
\begin{equation}\label{E:gammaseries}
\sum_{m'\in\Gamma'} a_{m'} \chi^{m'},\quad a_{m'}\in K(\Gamma^*) 
\text{ for all } m', 
\end{equation}
with naturally defined addition and multiplication, is called 
{\bf the ring of formal power series over $\Gamma'$ with coefficients in} 
$K(\Gamma^*)$. We denote it by $K(\Gamma^*)[[\Gamma']]$. 
\end{definition}

\begin{remark}
 If $L$ is a field and $\Gamma$ is an affine semigroup, then the set $L[[\Gamma]]$ 
 of formal power series with exponents in $\Gamma$ is naturally a group by addition
 of coefficients, monomial-wise. But it becomes a ring by adding the intuitive 
 multiplication law if and only if each element of $\Gamma$ can be represented
 only \emph{in a finite number of ways} as a sum of two elements of $\Gamma$, which
 is equivalent to the fact that $\Gamma$ is pointed. This explains why we needed 
 to work only with exponents in $\Gamma'$ in the previous definition.
\end{remark}

In the particular case when $\Gamma$ is pointed, the sum~\eqref{E:gammaseries} 
takes the simpler form:
$$\sum_{m\in\Gamma} a_m \chi^m$$
with $a_m\in K$. In this case, we write $K[[\Gamma]]$ instead of 
$K(\Gamma^*)[[\Gamma']]$ and we call this ring \emph{the power series ring 
over} $\Gamma$.

\begin{example}
If $\Gamma=\Z_{\geq 0}^{n}$, then the ring $K[[\Gamma]]$ is isomorphic to the
ring $K[[x_1,\dots,x_n]]$ of formal power series in $n$ variables with coefficients
in $K$.
\end{example}

The semigroup $\Gamma$ embeds naturally into the multiplicative semigroups of 
the rings $K[\Gamma]$ and $K[[\Gamma]]$. Moreover, a section $\alpha\colon
\Gamma' \to \Gamma$ of $p$ induces an embedding $\tilde{\alpha}\colon \Gamma\to
K(\Gamma^*)[[\Gamma']]$: 
$$\Gamma\ni m\mapsto \chi^{(m - \alpha(p(m)))}\cdot \chi^{\alpha(p(m))}\in 
K(\Gamma^*)[[\Gamma']],$$
(the monomial counterpart of the isomorphism $\Phi$ from the end of the proof of 
Lemma \ref{splitting}). 

Notice that, if $\beta\colon\Gamma'\to \Gamma$ is another section of $p$, then
$\tilde{\alpha}$ and $\tilde{\beta}$ differ by a unit, i.e., for any $m\in\Gamma
\subset K(\Gamma^*)[[\Gamma']]$, there exists an element $u(m)\in\Gamma^*$ such 
that $\tilde{\beta}(m)=\chi^{u(m)} \tilde{\alpha}(m)$. In what follows, we consider 
also the localization $R=K[\Gamma]_{(\Gamma^+)}$ of the semigroup ring $K[\Gamma]$ 
at its ideal $(\Gamma^+)=(\{\chi^m\,|\,m\in\Gamma^+\})$. The semigroup $(\Gamma, +) 
\simeq (\chi^{\Gamma}, \cdot)$ is 
naturally also a subsemigroup of $(R, \cdot)$.

\begin{lemma}\label{L:fc}
  Assume that the pointed affine semigroup $\Gamma$ is such that $p$ 
  admits a section $\alpha\colon \Gamma'\to 
  \Gamma$. Then $\alpha$ induces a unique isomorphism 
  \linebreak $\overline{\alpha}\colon 
  K(\Gamma^*)[[\Gamma']]\to \widehat{R_\mathfrak{m}}$, where 
  $\widehat{R_\mathfrak{m}}$ is the formal completion of the ring $R$ at its 
  maximal ideal $\mathfrak{m}=(\Gamma^+)$, such that the following diagram 
  commutes:
  \begin{equation}\label{E:gammatriangle}
  \xymatrix
  {
    &  &\Gamma\ar[ld]_{\tilde{\alpha}}\ar[rd]  & \\
    &K(\Gamma^*)[[\Gamma']]\ar[rr]^{\overline{\alpha}}  &  &\widehat{R_\mathfrak{m}}
  }
  \end{equation}
\end{lemma}

\begin{proof}
  First, notice that a monomial $\chi^m\in \chi^{\Gamma}$ is contained in the 
  ideal $\mathfrak{m}^n$ if and only if $p(m)\in\Gamma'$ is contained in 
  $n{\Gamma'}^+$. Next, diagram~\eqref{E:gammatriangle} shows that 
  $\tilde{\alpha}$ is defined on the monomials. Notice also that $R=
  K[\Gamma]_{(\Gamma^+)}\simeq K(\Gamma^*)[\Gamma']_{(\Gamma'^+)}$. 
  Recall that the ring $\widehat{R_\mathfrak{m}}$ is defined as the set of 
  sequences $\{f_n\}_{n=1}^{\infty}$, $f_n\in R/\mathfrak{m}^n$, compatible with
  respect to the natural maps $R/\mathfrak{m}^n\to R/\mathfrak{m}^{n-1}$. If
  $f=\sum a_{m'} \chi^{m'} \in K(\Gamma^*)[[\Gamma']]$, then the sequence of 
  its appropriate truncations defines a morphism of rings 
  $K(\Gamma^*)[[\Gamma']]\to \widehat{R_\mathfrak{m}}$ which is obviously 
  injective. To show surjectivity, write a representative for each $f_n$ in the 
  form $b_n/c_n$, where $b_n\in K[\Gamma]$, $c_n=d_n+q_n$, $d_n\in K[\Gamma^*]$, 
  $d_n\ne 0$, $q_n\in (\Gamma^+)$. Since any $m\in\Gamma$ can be written as 
  $u + \alpha(m')$ for some $u\in\Gamma^*$, $m'\in\Gamma'$, we may consider $b_n$ 
  and $c_n$ as polynomials in $K(\Gamma^*)[\Gamma']$. Using the standard identity:
  $$\frac{1}{1-q}\equiv 1+q+\dots+q^{n-1} \mod \mathfrak{m}^n$$
  for $q\in\mathfrak{m}$, we can rewrite:
  $$\frac{b_n}{c_n}\equiv h_0+h_1+\dots+h_{n-1} \mod \mathfrak{m}^n,$$
  where: 
  $$h_k=\sum_{m'\in  k {\Gamma'}^+\ \setminus \ (k+1) {\Gamma'}^+} 
  a_{m'} \chi^{m'} \in\mathfrak{m}^k\setminus \mathfrak{m}^{k+1}, \quad
  a_{m'}\in K(\Gamma^*) \text{ for all } m'.$$ 
  Compatibility of the sequence $\{f_n\}$ implies
  that $\sum_{n=0}^{\infty} h_n$ is a well defined series from 
  $K(\Gamma^*)[[\Gamma']]$. The last assertion of the lemma is obvious.
\end{proof}

From Lemma~\ref{L:fc}, we deduce: 

\begin{corollary} \label{compldom}
  The ring $K(\Gamma^*)[[\Gamma']]$ is a local Noetherian domain, complete with 
  respect to the $\hat{\mathfrak{m}}'$-adic topology, where $\hat{\mathfrak{m}}'$ 
  is its maximal ideal.
\end{corollary}

\emph{In the remaining of this section we suppose that $\Gamma$ is a pointed
semigroup.} Set $\sigma := \sigma(\Gamma)$. Consider a vector  $w\in\sigma$.  
If:
$$f = \sum_{m \in \Gamma} a_m \chi^m \in K[[\Gamma]]$$ 
is a power series over $\Gamma$, $f\ne 0$, the \emph{$w$-order} of $f$ is: 
\begin{equation}\label{E:ord}
w(f)=\min_{m:\;a_m\ne 0} \langle w,m\rangle,
\end{equation}
and the \emph{$w$-initial form} of $f$ is: 
$$\init_w(f)=\sum_{m:\;\langle w,m\rangle=w(f)} a_m \chi^m \in K [\Gamma].$$
Note that if $w\in\mathring{\sigma}$ (the interior of $\sigma$), then $\init_w(f)$ 
is a polynomial. If $I$ is an ideal of $K[[\Gamma]]$, the \emph{$w$-initial ideal} 
$\init_w(I)$ of $I$ is the ideal generated by $w$-initial forms of all the elements
of $I$. The same definitions can be given for the elements and the ideals of 
$K[\Gamma]$. 

\begin{definition}
The {\bf extended Newton diagram} of a series $f\in K[[\Gamma]]$ 
is the set
$$\Newton^+(f)=\text{Convex hull } (\bigcup_{m:\;a_m\ne 0} (m+\check\sigma)) \ 
\subseteq \ M(\Gamma)\otimes \R.$$
The {\bf Newton diagram} $\Newton(f)$ of $f$ is the union of all compact faces of 
$\Newton^+(f)$. 
\end{definition}

The extended Newton diagram of any series $f \in K[[\Gamma]]$ is a finite rational 
convex polyhedron, that is, it can be determined by finite number of linear 
inequalities of the form $\langle n,x\rangle \geq  a$, $n\in N(\Gamma)$, $a\in\Z$. 
Moreover, $\Newton^+(f)$ is contained in the cone $\check\sigma$ (the dual cone
of $\sigma$) and this last cone is equal to the {\em recession cone} of 
$\Newton^+(f)$ (which is defined 
as the maximal cone whose translation by any element of $\Newton^+(f)$ is 
contained in $\Newton^+(f)$). 

More generally, if $\tau$ is a face of $\sigma$ and 
if one takes $w \in \pi_\tau(\sigma)\subseteq\overline{\sigma}$
(see Section~\ref{troptor} and the formula \eqref{E:opensetoflv}), 
then $w$ defines a {\em preorder} (see the next section) on the monomials of the 
semigroup $\Gamma_\tau=\Gamma\cap\tau^\perp$. Thus, we can speak about
$w$-initial forms and $w$-initial ideals for arbitrary weights $w$ from 
$\overline{\sigma}$, but they should be applied to the $\tau$-truncations of 
elements of $K[[\Gamma]]$ and understood as elements or ideals of the 
corresponding ring $K[[\Gamma_\tau]]$:

\begin{definition} \label{Trunc}
  Let $\Gamma$ be a pointed affine semigroup. 
  Let $\tau$ be any face of $\sigma(\Gamma)$ and $\Gamma_{\tau} = \Gamma \cap 
  \tau^{\perp}$. If $f=
  \sum_{m\in\Gamma} a_m \chi^m \in K[[\Gamma]]$,  
  the $\tau$-{\bf truncation} $f_{\tau}$ of $f$ is defined by:
  $$f_\tau=\sum_{m\in\Gamma\cap\tau^\perp} a_m \chi^m \ \in \ K[[\Gamma_\tau]]. $$
  The $\tau$-{\bf truncation} $I_{\tau}$ of an ideal $I \subset K[[\Gamma]]$ 
  is defined as the ideal generated by the $\tau$-truncations of its elements. 
  If $w\in(N/N_{\tau})_\R$,  
  the $w$-{\bf initial form} $\init_w(f)$ of $f$ is defined as $\init_w(f_{\tau})$. 
  The $w$-{\bf initial ideal} $\init_w(I)$ of $I$ is the ideal of the 
  formal toric ring $K[[\Gamma_\tau]]$ which is generated by the $w$-\bf initial 
  forms of its elements. 
\end{definition}  
  
It is easy to check that $K[[\Gamma_\tau]]$ is the quotient of $K[[\Gamma]]$ by the 
prime ideal $\Gamma\setminus \Gamma_{\tau}$, and the assignement $f \mapsto f_\tau$ 
gives the natural quotient homomorphism of rings $K[[\Gamma]] 
\to K[[\Gamma_{\tau}]]$.

\section{Standard bases} \label{Standbas}

In this section we explain the notion of \emph{standard basis} of an ideal $I$ 
in a formal power series ring $K[[\Gamma]]$ with respect to a \emph{local monomial
ordering}, which is a local analog of the notion of \emph{Gr\"obner basis} of an 
ideal in a  polynomial ring with respect to a monomial ordering. We prove the 
existence of a \emph{universal standard basis}, that is, of a finite set of 
elements of $I$ which are a standard basis with respect to \emph{all} local 
monomial orderings. 
\medskip

As stated in Theorem \ref{T:3defs} point \ref{invp}), the global tropicalization of
a subvariety of a torus can be obtained also by looking at the initial ideals of 
the defining ideal of the subvariety with respect to all weight vectors. Such 
weight vectors define preorders on the lattice of monomials of the torus, 
compatible with the addition. Usually they are studied by also bringing into the 
game total orderings compatible with the addition. Those total orderings allow to 
define the notion of \emph{Gr\"obner basis} (see \cite{CLOS 97}, \cite{CLOS 05}, 
\cite{Eis 04}). We refer to \cite{BJSST 07} and \cite{FJT 07} for their 
application to the study of tropicalization of subvarieties of tori. 

Here, we develop an analogous theory of \emph{standard basis} in formal power 
series rings $K[[\Gamma]]$, where $\Gamma$ is an affine pointed semigroup.
In the next two sections we use it to study the local tropicalizations of ideals in 
$K[[\Gamma]]$.

A \emph{preorder} on a set is a binary relation which is both reflexive and 
transitive. A \emph{partial order} is a preorder which is antisymmetric. A 
\emph{total preorder} is a preorder such that any two elements of the set are 
comparable and a \emph{total order} is a total preorder which is also a partial 
order. A \emph{well ordered set} is a set endowed with a total order such that any 
nonempty subset has a minimum.

\begin{definition}\label{D:locord}
A {\bf local monomial ordering} on an affine pointed semigroup $\Gamma$  
is an order relation $\preceq$ on the set $\Gamma$ such that \\
(i) $0$ is the least element; \\
(ii) $\preceq$ is a total ordering; \\
(iii) $\preceq$ is compatible with addition on $\Gamma$, i.~e., if $m\preceq m'$, 
then $m + n \preceq m' + n$ for any $n\in\Gamma$.
\end{definition}

\begin{remark}
  If $m, n \in \Gamma$ and $m \preceq n$, we will also write 
  $\chi^m \preceq \chi^n$. This explains the name \emph{monomial ordering}: it is 
  an order on the monomials of $K[[\Gamma]]$. 
\end{remark}

In the sequel, by a \emph{monomial ordering} we shall always mean a local
monomial ordering.

The following proposition is standard for $\Gamma \simeq \N^n$ 
(see \cite[Chapter 2.4, Cor. 6]{CLOS 97}, where it is proved using the so-called 
{\em Dickson lemma} on finite generation by monomials of monomial ideals). 
We give here a proof  which does not pass through an analog of Dickson's lemma. 

\begin{lemma} \label{wellord}
  Under the axioms (ii) and (iii), condition (i) is equivalent to the fact that 
  $\preceq$ is a well-ordering of $\Gamma$. 
\end{lemma}

\begin{proof} 
  Assume that $\preceq$ is an ordering on $\Gamma$ which satisfies the axioms 
  (ii) and (iii). Suppose first that $\preceq$ is a well-ordering of $\Gamma$. 
  Arguing by contradiction, if (i) is not true, then there exists $m\in \Gamma$ 
  such that $m \prec 0$. Using axiom (ii), we get the following infinite chain of 
  inequalities: $0 \succ m \succ 2m \succ 3m \succ \cdots$. This implies that the 
  set $\{0, m, 2m, 3m ,\dotsc \}$ has no minimal element, which contradicts the 
  hypothesis that we have a well-ordering. 
   
  Suppose then that axiom (i) is satisfied, in addition to (ii) and (iii). Choose 
  a finite generating set $\{ \gamma_i \  | \ i \in I \}$ of non-zero elements 
  of $\Gamma$, which exists by the hypothesis that $\Gamma$ is an affine semigroup.
  Assume by contradiction that $\Gamma$ is not well-ordered. Then we get an 
  infinite decreasing sequence $m_1 \succ m_2 \succ m_3 \succ \cdots$ of elements 
  of $\Gamma$. Choose also an expression $m_j = \sum_{i\in I} a_{ij} \gamma_i$ 
  for each element of the sequence in terms of the chosen generating set. That is,
  $a_{ij} \in \N$ for all $i \in I, j \in \N^*$. Such expressions are in general 
  not unique, but this does not matter here. By axiom (i), as the $\gamma_i$ are 
  non-vanishing, we see that $\gamma_i \succ 0, \ \forall \ i \in I$. 
   
  Consider now an arbitrary $j \geq 2$. As $m_1 \succ m_j$, by axiom (ii) 
  there exists an index $i(j) \in I$ such that $a_{i(j),j} < a_{i(j),1}$. As 
  the sequence $(m_j)_{j \geq 2}$ is infinite, we may extract an infinite 
  subsequence in which $i(j)$ and $a_{i(j),j}$ are constant. Repeating 
  this argument a finite number of times, we arrive at an infinite strictly 
  decreasing sequence in which all the coefficients $a_{ij}$ are constant when $j$
  varies, which is a contradiction. 
\end{proof}

\begin{definition} \label{inipart}
For a given monomial ordering $\preceq$ on $\Gamma$, we define the {\bf initial
monomial} $\init_{\preceq}(f)$ of any element $f$ of $K[[\Gamma]]$ or $K[\Gamma]$ 
as the least monomial with non-zero coefficient in the 
expansion~\eqref{E:gammaseries} of $f$ and the {\bf initial ideal} 
$\init_{\preceq}(I)$ of an ideal $I$ of $K[[\Gamma]]$ or $K[\Gamma]$ as the ideal 
generated by the initial monomials of all the elements of $I$. 
\end{definition}

Consider then any vector $w \in \sigma(\Gamma)$. We define a 
preorder relation $\preceq_w$ on the elements of $\Gamma$ (in fact of the whole 
$M(\Gamma)$) depending on $w$:
$$m \preceq_w m' \text{ if and only if } \langle w,m\rangle \leq 
\langle w,m'\rangle.$$
Note that $0 \preceq_w m$ for any $m \in \Gamma$, according to this definition. 
We say that a monomial ordering $\preceq$ \emph{refines} a preorder $\preceq_w$, 
$w\in\sigma$, if $m\preceq m'$ implies $m\preceq_w m'$.

Note that a monomial $\chi^m \in K[[\Gamma]]$ is {\em divisible} by $\chi^n$ if and only
if $m = n + p$, where $p$ is again an element of $\Gamma$. This implies that 
$n \preceq n+p$ for any monomial ordering on $\Gamma$. 

The presence of a monomial ordering allows to extend the theory of divisibility 
from monomials to arbitrary series: 

\begin{proposition}\label{L:division} \emph{(Division algorithm)}  
  Let $\preceq$ be a fixed monomial ordering on $\Gamma$. 
  If $f \in K[[\Gamma]]$ and $(f_1,\dotsc,f_p)\in K[[\Gamma]]^p$ is an ordered 
  collection of series,  then there are series $g_1, \dotsc, g_p, r \in K[[\Gamma]]$ 
  such that: 
  $$f=g_1 f_1+\cdots+g_p f_p + r,$$
  where $\init_{\preceq}(f)\leq\init_{\preceq}(g_i f_i)$ for all $i$ such that 
  $g_i\ne 0$, $1\leq i\leq p$, and none of the monomials of $r$ is divisible by 
  any of the monomials $\init_{\preceq}(f_1) ,\dotsc, \init_{\preceq}(f_p)$.
\end{proposition}

\begin{proof}
  We simply apply the analog for series of the division algorithm for Gr\"obner 
  basis (see  \cite{CLOS 97}, \cite{CLOS 05}, \cite{Eis 04}). Here this algorithm
  involves an infinite number of steps, which compute the coefficients of the 
  unknown series $g_1, \dotsc, g_p, r$. 
  
  First, we find the smallest term $c\ \chi^m$, $c\in K$, $m \in\Gamma$, 
  of $f$ which is divisible by some 
  $\init_{\preceq}(f_i)$, $1\leq i\leq p$. If $i_0$ is the first such $i$, we
  reduce $f$ by defining:  
  $$R_1(f) : = f-\frac{c\ \chi^m}{a\init_{\preceq}(f_{i_0})} f_{i_0},$$
  where $f_{i_0}=a\init_{\preceq}(f_{i_0})+\cdots$.  We repeat the same process 
  with $R_1(f)$ instead of $f$, defining $R_2(f)$, and continue in the same way. 

  In the limit, we get a reduction $R_{\infty}(f)$ which has the property that no
  monomial of it is divisible by any $\init_{\preceq}(f_i)$, $1\leq i\leq p$. This
  is the remainder. Looking at the way we compute the sequence of reductions of 
  $f$, we see that $f - R_{\infty}(f)$ is indeed of the form 
  $g_1 f_1+\cdots+g_p f_p$, with $\init_{\preceq}(f)\leq\init_{\preceq}(g_i f_i)$ 
  for all $i$ such that $g_i\ne 0$, $1 \leq i \leq p$. 
\end{proof}

\begin{remark} \label{divrem} 
  This division result is usually presented for the ring of formal 
  power series $K[[x_1 ,\dotsc, x_m]]$. An analogous (but more complicated) 
  result for the ring of \emph{convergent} power series $\C\{\{x_1 ,\dotsc, x_m \}\}$ 
  was proved by Grauert (see \cite{JP 00}), but we will not need it here.  
\end{remark}

Notice from the previous proof that the quotients $g_1, \dotsc, g_p$, 
as well as the remainder $r$, are uniquely determined by the process 
\emph{if we carefully respect the order of the collection} $f_1, \dotsc, f_p$. 
But, in general, even the remainder changes if we change this order, as we show 
in the next example. 

\begin{example}
  Take the ring $K[[x,y]]$ with the lexicographic ordering in which $x \prec y$, 
  and the series $f = x, f_1 = x - y, f_2 = x - y^2$. Then, 
  $\init_{\preceq}(f_1) = \init_{\preceq}(f_2) = x$, which shows that 
  $g_1 =1, g_2 =0, r = y$. If we permute $f_1$ and $f_2$, we get $r = y^2$. 
\end{example}

This non-uniqueness of the remainder is eliminated if we take a \emph{standard 
basis} instead of an arbitrary sequence (see Proposition \ref{P:basisquot}).    

\begin{definition}
Let $\preceq$ be a monomial ordering on $\Gamma$ and $I$ an ideal of $K[[\Gamma]]$.
A finite sequence $\mathcal{B}\in I^p$ for some $p \in \N$ or, by abuse of language, 
the underlying set  is called a {\bf standard basis} for $I$
with respect to the ordering $\preceq$ if the initial monomials of the elements of 
$\mathcal{B}$ generate the initial ideal $\init_{\preceq}(I)$. A finite set 
$\mathcal{U}\subset I$ is called a {\bf universal standard basis} for $I$ if 
$\mathcal{U}$ is a standard basis for $I$ for any local monomial ordering 
$\preceq$ on $\Gamma$.
\end{definition}

\begin{remark} \label{remstand} 
     The terminology \emph{standard basis} was introduced in 
     \cite[Chapter III.1]{H 64} for a slightly  different concept, not involving 
     any ordering.
\end{remark}

The existence of a standard basis for any ideal $I\subset K[[\Gamma]]$ and 
any monomial ordering $\preceq$ on $\Gamma$ follows from Noetherianness of
$K[[\Gamma]]$ by a standard argument of the theory of Gr{\"o}bner bases. The
following three propositions are also standard.

\begin{proposition}
If $\mathcal{B}$ is a standard basis for an ideal $I\subseteq K[[\Gamma]]$ with
respect to some monomial ordering, then $\mathcal{B}$ generates $I$.
\end{proposition}

The next proposition shows that a standard basis induces a well-defined 
normal form for any element of $K[[\Gamma]] / I$. It corresponds to the remainder
of the division by this basis. 

\begin{proposition} \label{P:basisquot}
  Let $\preceq$ be a monomial ordering on $\Gamma$ and $I$ be an ideal of 
  $K[[\Gamma]]$. Suppose that $(f_1,\dotsc, f_p)$ is an associated standard basis.
  Consider the set $\Gamma_{\preceq}(I) \subset \Gamma$ of exponents of the 
  monomials belonging to the monomial ideal $\init_{\preceq}(I)$. Then,
  $\init_{\preceq}(I)$ is generated as a semigroup ideal by the exponents of the 
  initial monomials $\init_{\preceq}(f_1) ,\dotsc, \init_{\preceq}(f_p)$. Every
  element of $K[[\Gamma]] / I$ has a unique representative as a series whose 
  monomials have exponents in the complement $\Gamma \setminus 
  \Gamma_{\preceq}(I)$. This normal form is the remainder of the division 
  algorithm by $(f_1,\dotsc, f_p)$. 
\end{proposition}

In particular, the remainder of the division of any element of $I$ by a standard 
basis of $I$ is necessarily $0$. In fact, this characterizes standard basis: 

\begin{proposition} \label{charstand}
  Let $\preceq$ be a monomial ordering on $\Gamma$ and $I$ be an ideal of 
  $K[[\Gamma]]$. Take $\mathcal{B} = (f_1, \dotsc, f_p) \in K[[\Gamma]]^p$. Then 
  $\mathcal{B}$ is a standard basis of $I$ with respect to $\preceq$ if and only 
  if the remainder of the division of any element of $I$ by $\mathcal{B}$ is $0$.
\end{proposition}

The previous result allows to prove the following \emph{stability} property of 
standard basis when we change the defining monomial ordering.

\begin{proposition} \label{stabmon}
  Suppose that $\mathcal{B} = (f_1, \dotsc, f_p) \in K[[\Gamma]]^p$ is a standard
  basis of $I$ with respect to the monomial ordering $\preceq$. If $\preceq'$ is a
  second monomial ordering such that $\init_{\preceq}(f_i) = \init_{\preceq'}(f_i)$
  for all $i \in \{1, \dotsc , p \}$, 
  then $\mathcal{B}$ is also a standard basis with respect to $\preceq'$. 
\end{proposition}

\begin{proof}
  Let $f\in I$. Divide 
  $f$ by $\mathcal{B}$ with respect to $\preceq'$. Denote by $r\in I$ the 
  remainder. By the previous proposition, it suffices to show that $r =0$. 
   
  Suppose by contradiction that this is not the case. We know by 
  Proposition~\ref{L:division} that no monomial of $r$ is divisible by any 
  monomial $\init_{\preceq'}(f_i) = \init_{\preceq}(f_i)$. But $r \in I$ and 
  $\mathcal{B}$ is a standard basis with respect to $\preceq$, 
  which implies that: 
  $$ \init_{\preceq}(r) = \sum_{i =1}^k  g_i \init_{\preceq}(f_i)$$
  for some $g_1 ,\dotsc,  g_k \in K [[ \Gamma ]]$. This shows that the monomial  
  $\init_{\preceq} (r)$ is divisible by one of the monomials 
  $\init_{\preceq}(f_i)$, which is a contradiction. 
  Thus $r=0$. It follows that $\init_{\preceq'}(f)$ is divisible by some 
  $\init_{\preceq}(f_i)$ and hence $\mathcal{B}$ is a standard basis with respect
  to $\preceq'$.
\end{proof}

The Newton polyhedron of $f$ constrains deeply the possible initial terms of 
$f \in K[[\Gamma]]$ with respect to arbitrary monomial orderings of $\Gamma$.

\begin{lemma}\label{ininewton}
  For any $f \in K[[\Gamma]]$, the exponent of the initial monomial 
  $\init_{\preceq}(f)$ is an element of the finite subset of $\Gamma$ consisting
  of the vertices of $\Newton(f)$. 
\end{lemma}

\begin{proof}
  Denote by $V(f)$ the set of vertices of $\Newton^+(f)$ and by $m_0$ the exponent
  of $\init_{\preceq}(f)$. There exists $n \in \Newton(f) \cap M(\Gamma)_{\Q}$ with 
  $m_0 - n \in \check{\sigma}(\Gamma)$. Indeed, take a half-line starting from $m_0$
  and going to infinity inside $\Newton^+(f)$ in a rational direction (that is, 
  in direction of an element of $\Gamma$). Define then $n$ as the intersection 
  of the boundary of $\Newton^+(f)$ with the opposite half-line. We have
  $n = m_0$ if and only if $m_0$ belongs to $\Newton(f)$.
    
  Consider now the canonical extension of $\preceq$ to the whole rational 
  vector space $M(\Gamma)_{\Q}$. We denote this extension by the same symbol
  $\preceq$. It can be constructed in the same way as we construct the extension to 
  $\Q$ of the usual order on $\N$: extend it first to $M(\Gamma)$ by setting  
  $m_1 - m_2 \succ 0 \Leftrightarrow m_1 \succ m_2$ for any $m_1, m_2 \in \Gamma$,
  then to $M(\Gamma)_{\Q}$ by setting $\lambda \cdot m \succ 0$ for any 
  $\lambda \in \Q_+^*$ and any $m \in M(\Gamma)$ such that $m \succ 0$. It is a 
  routine exercise to verify that we get like this a well-defined total order on 
  $M(\Gamma)_{\Q}$. 
    
  Let us come back to the exponents $m_0 \in \Gamma$ and to 
  $n\in  \check{\sigma}(\Gamma) \cap \Gamma_{\Q}$. As, by construction, $m_0 - n$ 
  is positively proportional to an element of $\Gamma$, we get the inequality 
  $m_0 \succeq n$. 
    
  Choose now an arbitrary face $P$ of the Newton diagram $\Newton(f)$ containing 
  $n$. It is a compact convex polyhedron in $M(\Gamma)_{\R}$, with vertices in 
  $\Gamma$ and with dimension at most $\mbox{rk} (M(\Gamma)) -1$. If 
  $(v_j)_{\in J}$ is the set of its vertices, we have therefore a convex 
  expression of $n$ in terms of those vertices:
  $$ n = \sum_{j \in J} p_j \cdot v_j,  \mbox{ with } \sum_{j \in J} p_j =1 
  \mbox{ and } p_j \in [0, 1]  \mbox{ for all }  j \in J.$$
  Let $v_0$ be the minimal vertex of $P$ with respect to $\preceq$. Then, as all 
  the coefficients $p_j$ are non-negative, we deduce from the compatibility of 
  $\preceq$ with the $\Q$-vector space structure of $M(\Gamma)$ that  
  $n = \sum_{j \in J} p_j \cdot v_j \  \succeq \   
  \sum_{j \in J} p_j \cdot v_0 = v_0.$ Combining this inequality with the 
  inequality $m_0 \succeq n$ obtained before, we get $m_0 \succeq v_0$. As $m_0$ 
  is by definition the exponent of $\init_{\preceq}(f)$, we deduce that 
  $m_0 = v_0$, which proves the lemma. 
\end{proof}

\begin{corollary}  \label{L:twoorders}
  Let $I$ be an ideal of $K[[\Gamma]]$, $\preceq$ a monomial ordering on $\Gamma$,
  and $\mathcal{B}=\{f_1,\dotsc,f_k\}$ a standard basis of $I$ with respect to 
  $\preceq$. Let $\preceq'$ be a second monomial ordering which coincides with
  $\preceq$ when restricted to the finite set $\{m\in\Gamma\,|\, \ 
  \exists \ i=1,\dotsc,m \colon m\in\Newton(f_i)\}$. Then $\mathcal{B}$ is also a 
  standard basis with respect to $\preceq'$.
\end{corollary}
\begin{proof}    
  By Lemma~\ref{ininewton}, we have $\init_{\preceq}(f_i) = 
  \init_{\preceq'}(f_i)$ for all $i \in \{ 1 ,\dotsc,  p\}$, which implies the desired
  assertion by Proposition \ref{stabmon}. 
\end{proof}

As a consequence, a standard basis for a monomial order remains standard for 
conveniently defined \emph{neighboring} orders. Following Sikora \cite{S 04}, 
Boldini \cite{Bol 09} and \cite{Bol 10}, we see now that there is indeed a notion
of topology on the space of monomial orders such that standard bases are
locally constant.

Let $S$ be any set. Denote by: $$TO(S)$$ the set of all total orderings of $S$. 
One has a natural topology on it. Intuitively, given two elements $a, b \in S$ 
such that $a \prec b$ for some ordering $\preceq \in TO(S)$, then this 
strict inequality should also hold in a neighborhood of $\preceq$. Therefore, 
one is forced to declare the subsets:
 $$U_{(a,b)}: = \{ \preceq \in TO(S) \ \mid \  a\preceq b  \}$$ 
{\em open}, for all $a,b\in S$. Therefore, we endow $TO(S)$ with the topology 
generated by them. 
  
In the case when $S$ is a semigroup and we only take the orderings that are 
compatible with the semigroup law, this topology was defined by Sikora 
\cite{S 04}. The extension to arbitrary sets was done by Boldini \cite{Bol 09}.
Sikora proved that under the additional hypothesis that $S$ is countable the 
associated topology is compact. Boldini proved the analogous fact 
for an arbitrary countable set:

\begin{proposition}[{\cite[Teorem~1.4]{Bol 09}}]
  If the set $S$ is countable, then the space $TO(S)$ is compact.
\end{proposition}

Given an element $a\in S$, let $SO_a(S)$ be the subspace of $TO(S)$ consisting
of all total orderings for which the element $a$ is minimal, i.e., $a\leq b$
$\forall b\in S$.
  
\begin{proposition}[{\cite[Theorem~1.5]{Bol 09}}]
  The subspace $SO_a(S)$ is closed in $TO(S)$ for each $a\in S$. Hence, if $S$
  is countable, $SO_a(S)$ is compact.
\end{proposition}

Now we let $S=\Gamma$ be an affine pointed semigroup. We denote by: $$MO(\Gamma)$$ the 
set of all monomial orderings on $\Gamma$.

\begin{proposition}[{\cite[Theorem~2.4]{Bol 09}}]
  $MO(\Gamma)$ is a closed compact subset of $SO_0(\Gamma)$.
\end{proposition}

\begin{lemma}[cf. {\cite[Lemma~2.10]{Bol 09}}]
  Let $I$ be an ideal of $K[[\Gamma]]$, and $\mathcal{B}$ a finite subset of $I$.
  Then, the set of all monomial orderings $\preceq$ such that 
  $\mathcal{B}$ is a standard basis with respect to $\preceq$ is open in 
  $MO(\Gamma)$.
\end{lemma}
  
\begin{proof}
  The proof is essentially the same as in \cite{Bol 09}. In view of 
  Lemma~\ref{L:twoorders}, we must only replace the support of $\mathcal{B}$
  with the set of monomials of the series from $\mathcal{B}$ lying on the union of
  Newton diagrams of elements of $\mathcal{B}$.
\end{proof}

It is not obvious from the definition that universal standard basis 
indeed exist. Nevertheless, it is an immediate consequence of the compactness of 
the space of monomial orderings:

\begin{theorem}
  Any ideal of the ring $K[[\Gamma]]$ has a universal standard basis.
\end{theorem}

\begin{proof}
  Our argument is similar to that of \cite[Theorem~2.14]{Bol 09}. 
  The family $U_\mathcal{B}$, where $\mathcal{B}$
  runs over all finite subsets of $I$, forms an open covering of the space 
  $MO(\Gamma)$. Since $MO(\Gamma)$ is compact, we can choose a finite subcovering
  $U_{\mathcal{B}_1} ,\dotsc, U_{\mathcal{B}_k}$. We conclude that the union 
  $\cup_{i=1}^{k} \mathcal{B}_i$ is a universal standard basis of $I$.
\end{proof}

The following proposition will be used in the proof of Theorem~\ref{Existrop}:

\begin{proposition}\label{P:ugb}
  If $\ \mathcal{U}=\{f_1,\dotsc,f_p\}$ is a universal standard basis for an ideal
  $I\subset K[[\Gamma]]$ and $w\in\sigma(\Gamma)$, then $\init_w(\mathcal{U})=
  \{\init_w(f_1) ,\dotsc, \init_w(f_p)\}$ is a universal standard basis for the 
  initial ideal $\init_w(I)$.
\end{proposition}

\begin{proof}
  We have to show that $\init_w(\mathcal{U})$ is a standard basis of $\init_w(I)$
  for any monomial ordering. Let $g\in\init_w(I)$, and $\preceq_\alpha$ be a monomial
  ordering on $\Gamma$. Consider the ordering $\preceq_{w,\alpha}$ which is 
  defined by comparing the monomials first by $\preceq_w$, and then by 
  $\preceq_\alpha$. Clearly the initial terms of $\init_w(f_1)  ,\dotsc,
  \init_w(f_p)$, $g$ with respect to $\preceq_\alpha$ and with respect to 
  $\preceq_{w,\alpha}$ coincide. On the other hand, since $\mathcal{U}$ is 
  universal, $\init_{w,\alpha}(g)$ is divisible by at least one of 
  $\init_{w,\alpha}(f_i)$ (we use here also the fact that for each 
  $g\in\init_w(I)$ there exists $f\in I$ such that $\init_w(g)=\init_w(f)$, see 
  Lemma~\ref{L:initforms} below). This concludes our proof.
\end{proof}

\begin{lemma}\label{L:initforms}
  Suppose that $w \in \sigma(\Gamma)$. Then, for all $h \in \init_w I$, 
  there exists $f \in I$ with $\init_w h = \init_w f$. 
\end{lemma}   

\begin{proof}
  Since $h\in \init_w I$, there exist $h_1 ,\dotsc, h_n\in K[[\Gamma]]$ and 
  $f_1 ,\dotsc, f_n\in I$ such that:
  \begin{equation}\label{E:h}
  h=h_1\init_w(f_1)+\cdots+h_n\init_w(f_n).
  \end{equation}
  Notice that $w$, considered as a morphism of semigroups from $\Gamma$ to 
  $\R_{\geq 0}$, has a countable image with infinity as the single accumulation 
  point. Thus we may write: 
  $$\im(w)=\{\mu_0=0,\mu_1,\mu_2,\dots\},$$
  where $\mu_0<\mu_1<\mu_2<\cdots$. Now every series $g\in K[[\Gamma]]$ can be 
  decomposed into its weighted homogeneous components $g_{\mu_i}$, $w(g_{\mu_i})=
  \mu_i$:
  $$g=\sum_{i=0}^{\infty} g_{\mu_i}.$$
  Each $g_{\mu_i}$ is a $w$-weighted homogeneous series. 
  Applying such a decomposition to \eqref{E:h} and 
  comparing forms of $w$-order $w(h)$, we get: 
  $$\init_w(h)=\sum_{j=1}^{n} h_{j,w(h)-w(f_j)} \init_w(f_j),$$
  where $h_{j,w(h)-w(f_j)}$ is the $w$-homogeneous component of
  $h_j$ of order $w(h)-w(f_j)$ (it is $0$ by definition if $w(h)-w(f_j)<0$).

  Now consider the following element of $I$:  
          $$f=\sum_{j=1}^{n} h_{j,w(h)-w(f_j)} f_j.$$
  Each $f_j$, $1\leq j\leq n$, has the form: 
       $$f_j=\init_w(f_j)+(\text{terms of order }>w(f_j)).$$
  It follows that $\init_w(f)=\init_w(h)$.
\end{proof}  

\section{Tropical bases}\label{S:tbases}

As in the case of subvarieties of tori \cite[Theorem 11]{BJSST 07}, in this 
section we prove the existence of tropical bases of ideals of $K[[\Gamma]]$. These 
bases are particular systems of generators of $I$ that allow us to compute
$\ptrop(I)$ as the intersection of local tropicalizations of hypersurfaces
defined by these generators. 
\medskip

\begin{definition}\label{D:tbasis}
  Let $\Gamma$ be a pointed affine semigroup. 
  A {\bf tropical basis} of an ideal $I$ of the ring $K[[\Gamma]]$  
  is a universal standard basis 
  $\{f_1,\dotsc,f_p\}$ of $I$ such that for any $w \in \sigma(\Gamma)$, the ideal 
  $\init_w(I)$ contains a monomial if and only if one of the initial terms 
  $\{\init_w(f_1) ,\dotsc, \init_w(f_p)\}$ is a monomial. 
\end{definition}  

It is not always true that a universal standard basis is tropical:

\begin{example}
  We consider the same polynomials as the ones chosen in 
  \cite[Example 10]{BJSST 07}, but this time seen as generators of an 
  ideal of the ring $\C[[x,y,z]]$ of formal power series in three variables. 
  Namely, we take:
    $$I = (x + y + z, xy (x + y), xz (x + z), yz (y + z)).$$
  Any two of the last three polynomials are redundant as generators, but they 
  are needed in order to get a universal standard basis. We show that the four
  generators of $I$ form a universal standard basis of $I$. 
   
  Consider an arbitrary monomial ordering $\preceq$. We have to show 
  that the initial ideal $\init_{\preceq} (I)$ is generated by the initial terms
  of these four polynomials. Since the set is symmetric in $x, y, z$, it is 
  enough to study the case when  $x \prec y \prec z$. Therefore, we have to show
  that $\init_{\preceq} (I)$ is generated by the monomials $x$ and $y^2 z$   
  ($y \prec z$ implies $y^2 z \prec y z^2$).

  Which monomials are not in the ideal generated by $x$ and $y^2 z$? Only those 
  of the form $y^k, z^k$, and $y z^k$. Let us show that none of these belongs to 
  $\init_{\preceq} (I)$. 
   
  Consider first the case of $y z^k$.  If $y z^k \in \init_{\preceq} (I)$, then 
  there exist two series $f, g \in \C[[x,y,z]]$ such that $y z^k$ is the 
  $\preceq$-initial term of:
    \begin{equation}\label{speceq}
          f \cdot (x + y + z) + g \cdot (y^2 z + y z^2),   
    \end{equation}
  since $x+y+z$ and $y^2 z + y z^2$ generate $I$. 
  First, let us substitute $x=0$ (take the quotient $k[[x,y,z]]/(x)$ and 
  consider the induced monomial order on it). We get: 
  $$y z^k = \init_{\preceq}(f_0 \cdot (y + z) + g_0 \cdot (y^2 z + y z^2) )= $$
  $$= \init_{\preceq}((y + z)(f_0 + y z g_0)) = y \init_{\preceq}(f_0 + y z g_0).$$
  It follows that the initial term of $f_0$ is $z^k$,
  and hence $z^k$ has a non-zero coefficient in $f$ too.

  Therefore, when we distribute the product $f \cdot (x + y + z)$, we get the 
  monomial $x z^k$ as a term in this expansion. But $x z^k \prec y z^k$. 
  Therefore, it must cancel in \eqref{speceq}. As this monomial does not appear 
  in $g \cdot (y^2 z + y z^2)$, we see that $x z^k$ cancels only if $f$ contains 
  also the monomial $x^2 z^{k-1}$. Again, then the product $f \cdot (x + y + z)$ 
  contains $x^2 z^{k-1}$ which is less than $x z^k$ and $y z^k$. We conclude that 
  $f$ contains also $x^3 z^{k-2}$, and so on. But then we come to a contradiction,  
  because the series $f$ does not have any negative powers of $z$.

  The argument for $y^k$ and $z^k$ is similar and even easier, because we do 
  not need to pass to the quotient $k[[x,y,z]]/(x)$.

  The fact that the four polynomials are not a tropical basis is proved now exactly
  as in \cite[Example 10]{BJSST 07}. Namely, consider the weight $w = (1,1,1)$. 
  The four polynomials are equal to their initial terms with 
  respect to $w$ (they are homogeneous), therefore these initial terms are not 
  monomials. But $xyz \in I$, therefore $\init_w(I)$ contains the monomial $xyz$. 
  This shows that the four polynomials do not form a tropical basis of $I$.  
\end{example}

Therefore, we are led to ask whether tropical bases for ideals of rings of the 
form $K[[\Gamma]]$ exist necessarily. This is indeed the case:

\begin{theorem} \label{Existrop}
  Any ideal of the ring $K[[\Gamma]]$ has a tropical basis.
\end{theorem}  

\begin{proof}
  Starting from any universal standard basis $\mathcal{U}=\{f_1,\dotsc,f_k\}$ 
  of a given ideal $I$, 
  we shall construct a tropical basis of $I$ by adding new series to 
  $\mathcal{U}$.

  Note that the cone $\sigma= \sigma(\Gamma)$ is 
  naturally stratified by the relative interiors of its faces:
  $$\sigma=\bigsqcup_{\tau\leq\sigma} \mathring{\tau}.$$
  Furthermore, if $f\in K[[\Gamma]]$, each $w\in\sigma$ can be considered as a 
  function on the extended Newton diagram $\Newton^+(f)$ of $f$. This function 
  takes its minimal value on some face of $\Newton^+(f)$. We say 
  that this face is \emph{cut by the function} $w$. Now, we define an equivalence 
  relation on the set of vectors of the cone $\sigma$: $w\sim w'$ if and only if
  $w$ and $w'$ are contained in the same stratum $\mathring{\tau}$ and for all 
  $i$, $1\leq i\leq k$, $w$ and $w'$ cut the same face of $\Newton^+(f_i)$. The 
  reader can easily check that there are only finite number of equivalence classes 
  of $\sim$, that they give a new stratification of $\sigma$ refining the one 
  described above, and the closure of each equivalence class is a rational 
  polyhedral cone. The set of these cones is a fan that we denote by 
  $\Sigma_{\mathcal{U}}$. Moreover, it follows from Proposition~\ref{P:ugb} that 
  if $w\sim w'$, then $\init_w(I)=\init_{w'}(I)$. 
  
  Thus, $\Sigma_{\mathcal{U}}$ is a refinement of the {\em local Gr{\"o}bner fan of 
  the ideal} $I$. This notion was introduced by Bahloul and Takayama in 
  \cite{BT 04}, \cite{BT 07} for ideals of formal power series rings 
  $K[[X_1 ,\dotsc, X_n]]$ as a local 
  analog of the notion of {\em Gr\"obner fan of an ideal} of a polynomial ring 
  introduced by Mora and Robbiano \cite{MR 88}. It may be immediately extended in
  our context.

  Let $\rho$ be a cone of $\Sigma_{\mathcal{U}}$ such that for some (and thus for
  any) $w\in\mathring{\rho}$ the initial ideal $\init_w(I)$ contains a
  monomial. If $\init_w(f_i)$ is a monomial for some $f_i\in\mathcal{U}$, we do not
  add any series to $\mathcal{U}$. Assume then that none of $\init_w(f_i)$ is a 
  monomial and let $\chi^m\in\init_w(I)$ be a monomial. Choose an irrational point
  $w'\in\sigma$, so that the preoder determined by the vector $w'$ is actually a 
  monomial ordering. Let $\preceq_{w,w'}$ be a monomial ordering defined by 
  comparing the monomials first by $\preceq_w$ and then by $\preceq_{w'}$. Now, 
  divide the monomial $\chi^m$ by $\mathcal{B}$ with respect to $\preceq_{w,w'}$ 
  (see Proposition \ref{L:division}). We get an expression: 
  $$\chi^m=\sum_i g_i f_i +r.$$
  Notice that the initial monomials of $f_1 ,\dotsc, f_k$ with respect to 
  $\preceq_{w,w'}$ are independent of $w$ whenever $w\in\mathring{\rho}$. It 
  follows that the remainder $r$ is also independent of $w$ (the reason is that 
  the monomials not contained in the initial ideal $\init_{\preceq_{w,w'}}(I)$ 
  form a ``basis'' of 
  the quotient ring $K[[\Gamma]]/I$, see the proof of Lemma~\ref{L:flat}). 
  Also, since $\chi^m$ can be represented as a combination of $w$-initial forms of 
  $f_1 ,\dotsc, f_k$, the value $w(r)$ of $r$ is strictly
  greater than $w(\chi^m)=\langle w,m\rangle$, and this also holds for all
  $w\in\mathring{\rho}$. The element $f_\rho=\chi^m-r$ lives in $I$, and by 
  construction the $w$-initial form of $f_\rho$ is the monomial $\chi^m$, for any 
  $w\in\mathring{\rho}$. Adding to $\mathcal{U}$ all the series of the form $f_\rho$, 
  $\rho\in\Sigma_{\mathcal{U}}$, as described above, we get a tropical basis 
  for $I$.
\end{proof}

Finally, we generalize the notion of tropical basis so that it allows also to study 
the initial ideals corresponding to arbitrary, not necessarily finite, vectors
$w\in\overline{\sigma}(\Gamma)$. If $\tau$ is a face of $\sigma(\Gamma)$, 
we consider the \emph{$\tau$-truncation} $I_\tau$ of the ideal $I$ 
(see Definition~\ref{Trunc}). It is easy to see that any element of $I_\tau$ 
is a truncation of some element of $I$. For each $\tau$, let us choose a finite 
set $\mathcal{B}_\tau$ of elements of $I$ such that the set of truncations of 
$\mathcal{B}_\tau$ is a tropical basis for $I_\tau$. Setting:
$$\mathcal{B}=\bigcup_{\tau\leq\sigma} \mathcal{B}_\tau,$$
we get a finite subset of $I$ such that its truncation in every 
$K[[\Gamma_{\tau}]]$ is a tropical basis of $I_\tau$. We turn this property 
into a definition:

\begin{definition} \label{extropbas}
A finite subset of an ideal $I$ of $K[[\Gamma]]$, such that its truncation in every 
ring $K[[\Gamma_{\tau}]]$ for varying faces $\tau$ of $\sigma(\Gamma)$ 
is a tropical basis of $I_\tau$
is called an {\bf extended tropical basis} of $I$. 
\end{definition}

An extended tropical basis $\mathcal{B}$ may be characterized also by the 
property that for any face $\tau$ of $\sigma$ and any 
$w\in\overline{\sigma} \cap (N/ N_{\tau})_{\R}$, the initial ideal $\init_w(I)$ 
(considered as an ideal of the ring $K[[\Gamma_{\tau}]]$) contains a 
monomial if and only if one of the initial forms $\init_w(f)$, $f\in\mathcal{B}$, 
is a monomial.

\medskip
\section{The local finiteness theorem}\label{S:lstruct}

Our main goal here is to describe the piecewise-linear structure of the local 
tropicalization. We were not able to prove this fact in full generality. 
We could do this only for quotient rings of the ring of formal
power series $K[[\Gamma]]$ over a pointed affine semigroup $\Gamma$ (see Theorem 
\ref{T:main}) and for another related class of morphisms (see Theorem \ref{T:moregen}).

\medskip
In this section we keep the assumption that $\Gamma$ is an affine pointed semigroup.
As usual, $\sigma = \sigma(\Gamma)$. 
Recall from Remark \ref{specnotloc} that we denote:
$$\ptrop(I) := \ptrop(\gamma) \mbox{ and }\nntrop(I) : = \nntrop(\gamma)$$
if  $I$ is an ideal of the ring $K[[\Gamma]]$ and 
$\gamma\colon\Gamma\to K[[\Gamma]]/I$ is the natural semigroup morphism.
We start proving that the definitions through extensions of valuations and 
initial ideals lead to the same concept of local tropicalization for the canonical 
morphism of semigroups $(\Gamma, +) \to (K[[\Gamma]], \cdot)$. The following
result plays an essential role in the proof of Theorem~\ref{Tropinit}.

\begin{theorem}[{\cite[Corollary 1]{Bm 71}}]\label{T:Bergman}
Let $R$ be a commutative ring with unit and $v$ a valuation on $R$. 
Let $I$ be an ideal of $R$ and $S$ a multiplicative subsemigroup of $(R, \cdot)$ 
such that there is no $g\in S$, $f\in I$ satisfying $v(g)=v(f)< v(f- g)$. 
Then, there exists a valuation $v'\geq v$ on $R$ such that 
$v'|_I=+\infty$, $v'|_S=v|_S$.
\end{theorem}

We now apply Theorem~\ref{T:Bergman} to local tropicalizations.

\begin{theorem} \label{Tropinit}
Suppose that $\Gamma$ is an affine pointed semigroup. 
Let $I$ be an ideal of the ring $K[[\Gamma]]$ and 
denote by $\gamma\colon\Gamma\to K[[\Gamma]]/I$ the natural semigroup morphism. 
Then: 
\begin{itemize}
  \item[(i)] $\ptrop(I)=\nntrop(I)\cap \overline{\sigma}^\circ$.
  \item[(ii)] 
  The following two subsets of the linear variety $L(\sigma,N)$ coincide:
  \begin{enumerate}
    \item the local nonnegative tropicalization $\nntrop(I)$;
    \item the set $T$ of those $w\in\overline{\sigma}\subset L(\sigma,N)$ 
    such that the initial ideal $\init_w(I)$ is monomial free.
  \end{enumerate}
\end{itemize}
\end{theorem}

\begin{proof}
  Part (i) follows from Proposition~\ref{P:closedtrop} (ii), so let us prove
  the second part. First we show that $\nntrop(\gamma)\subseteq T$. Any 
  valuation $v$ of $K[[\Gamma]]/I$ lifts to a valuation $\overline{v}$ of 
  $K[[\Gamma]]$ such that $\overline{v}|_I=+\infty$. Consider the vector 
  $w\in L(\sigma,N)$ determined by the valuation $\overline{v}$. If $v$ is 
  nonnegative on $K[[\Gamma]]/I$, then the vector $w$ is contained in 
  $\overline{\sigma}$. Consider an arbitrary $f\in I$. Since 
  $\overline{v}(f)=+\infty$, $w$ takes its minimal value on at least two
  monomials in $f$. Therefore, $\init_w(f)$ is not a monomial. As $f$ was chosen 
  arbitrarily inside $I$, we see that indeed $\init_w(I)$ is monomial free.

  Now let us show that $T\subseteq\nntrop(\gamma)$. Choose any $w\in
  \overline{\sigma}$ such that $\init_w(I)$ is monomial free. The extended 
  weight vector $w$ defines a monomial valuation on $K[[\Gamma]]$ or on 
  $K[[\Gamma_{\tau}]]$ if $w$ belongs to a stratum at 
  infinity of $L(\sigma,N)$: $f\mapsto w(f)$ (see \eqref{E:ord}). In the latter 
  case we also want to consider $w$ as a valuation on the whole of $K[[\Gamma]]$. 
  For this, if $w\in(N/N_{\tau})_\R$ and $f= \sum_{m\in\Gamma} a_m \chi^m$, first 
  take the $\tau$-truncation $f_\tau$ (see Definition \ref{Trunc}), 
  and then apply the valuation $w$. Note that for any $f\in I$ and 
  $m\in\Gamma$, it is impossible to have simultaneously 
  $w(\chi^m)=w(f)<w(f-\chi^m)$ because $f$ has at least two monomials $\chi^m$ and 
  $\chi^n$ of minimal value (as $\init_w(I)$ is supposed monomial free). Thus, by 
  Theorem~\ref{T:Bergman}, there exists a valuation on $K[[\Gamma]]/I$ giving 
  exactly the point $w$ under the tropicalization map.
\end{proof}

\begin{remark}\label{R:locstructoftrop}
Consider any point $w\in \ptrop(I)\cap N$. Then, by Theorem~\ref{Existrop} and
Theorem~\ref{Tropinit} (ii), in a neighborhood of $w$ the positive tropicalization 
of an ideal $I$ coincides with that of the initial ideal $\init_w(I)$.
\end{remark}

In the sequel we want to make a clear distinction between a {\em fan} and a set 
which is the {\em support of a fan}, but without fixed fan structure. That is why 
we introduce the following definition:

\begin{definition}\label{D:conicalset}
  A {\bf PL cone} in a vector space $L$ is a subset 
  $\Sigma\subseteq L$ that can be represented as a finite union of convex 
  polyhedral cones. It is called {\bf rational}, if it can be represented as a 
  finite union of rational convex polyhedral cones, that is, if it is the support 
  of a fan. A subset $\Sigma$ of a linear variety $L(\sigma,N)$ is called a 
  {\bf (rational) PL conical subspace} if for each stratum $(N/N_\tau)_\R$ the 
  intersection $\Sigma\cap (N/N_\tau)_\R$ is empty or a (rational) PL cone. 
\end{definition}

The following are examples of PL cones and PL conical subspaces:

\begin{definition}
  The {\bf Newton cone} of a series $f\in K[[\Gamma]]$, denoted  
  $\Newton^\perp(f)$, is the set of vectors $w\in\sigma$ such that, seen as a 
  function on the Newton diagram $\Newton(f)$, $w$ attains its minimum on a face 
  of positive dimension. If $f=0$, we set $\Newton^\perp(f)=\sigma$ by 
  definition. The {\bf extended Newton cone} of $f$ is the disjoint
  union $\widetilde{\Newton^\perp(f)} := 
  \bigsqcup_{\tau\leq\sigma}\Newton^\perp(f_\tau)$, where $f_\tau$ is the
  $\tau$-truncation of $f$. 
\end{definition}

If $f\ne 0$ and $f\ne u\cdot\chi^m$, where $m\in \Gamma$
and $u$ is a unit in $K[[\Gamma]]$, then the Newton cone is indeed a 
rational PL conical subspace of pure dimension $n-1$, where $n$ is the rank of 
the lattice $N=N(\Gamma)$. Notice that our Newton cone is different from 
what is usually called the \emph{normal fan} of $f$. The Newton cone is the 
support of the $(n-1)$-skeleton of the standard normal fan. 

\begin{example}
Consider $\Gamma = \N^2 \subset \Z^2 = M$ and the reducible polynomial 
$f = x(x + y)(x + y^2) =  x^3 + x^2 y + x^2 y^2 + x y^3 \in K[[\Gamma]] = K[[x,y]]$.
Here we set $x = \chi^{(1,0)}$ and $y = \chi^{(0,1)}$, where 
$e_1 = (1,0), e_2=(0,1)$ form the canonical basis of $\Z^2$. In Figure 
\ref{fig:Extnorm} are represented its associated extended Newton diagram, Newton 
cone and extended Newton cone. The black discs in the drawing of the extended 
Newton diagram represent the exponents of the monomials of $f$. 
The Newton diagram $\Newton(f)$ has two edges, denoted $AB$ and $BC$  
in the figure, where $A = (3,0), B=(2,1), C = (1,3)$. The Newton cone 
$\Newton^\perp(f)$ lives in the dual plane $\R^2 = N_{\R}$, endowed with the dual 
basis $(v_1, v_2)$ of $(e_1, e_2)$. It is contained in the cone 
$\sigma = \R_+ v_1 + \R_+ v_2$, whose edges are $\tau_i = \R_+ v_i$ for $i = 1,2$.  
It is the union of two closed half-lines, $H_{AB}$ normal 
to $AB$ and $H_{BC}$ normal to $BC$. The extended Newton cone 
$\widetilde{\Newton^\perp(f)}$ lives in the affine linear 
variety  $L(\sigma, N)$. In addition to $H_{AB}$ and $H_{BC}$ it contains at 
infinity the half-line $H_1$, projection of $\sigma$ to $L_1 = 
(N/N_{\tau_1})_{\R}$, and the point $L_{12} = (N/N_{\sigma})_{\R}$. 
Note that the closures of the three half-lines $H_{AB}, H_{BC}, H_1$ 
contain the point $L_{12}$ at infinity. 
\end{example}

\bigskip
\begin{figure}[h!]
\labellist
\small\hair 2pt
\pinlabel  {$A$} at 276 250
\pinlabel  {$B$} at 275 297
\pinlabel  {$C$} at 260 335
\pinlabel  {$H_{AB}$} at 132 180
\pinlabel  {$H_{BC}$} at 177 140
\pinlabel  {$H_{AB}$} at 387 150
\pinlabel  {$H_{BC}$} at 442 90
\pinlabel  {$0$} at 0 36
\pinlabel  {$0$} at 330 36
\pinlabel  {$L_{12}$} at 510 208
\pinlabel  {$H_1$} at 530 123
\pinlabel  {$L_1$} at 493 20
\pinlabel  {$L_2$} at 293 199
\pinlabel  {$\tau_1$} at 108 44
\pinlabel  {$\tau_2$} at -10 125
\pinlabel  {$\tau_1$} at 420 44
\pinlabel  {$\tau_2$} at 310 125
\endlabellist
\centering
\includegraphics[scale=0.50]{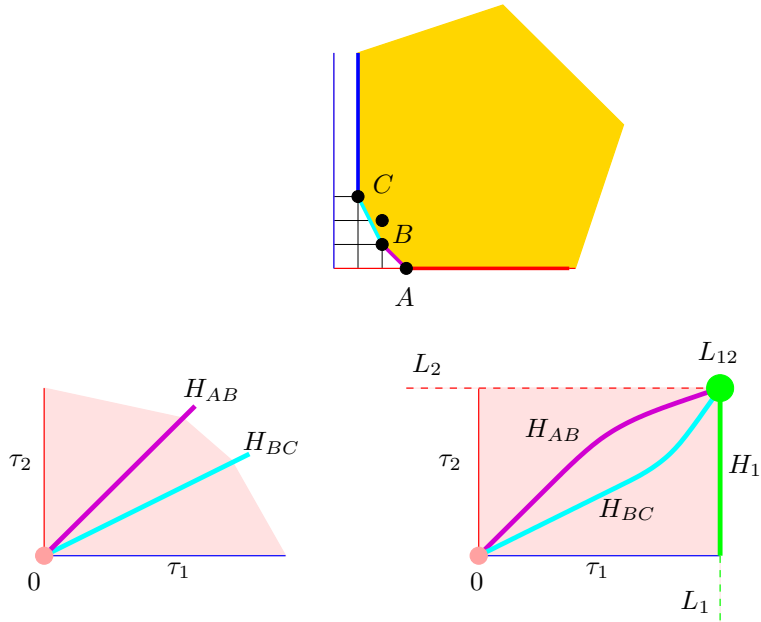}
\caption{An extended Newton diagram and its associated Newton cone and 
    extended Newton cone in dimension $2$}
\label{fig:Extnorm}
\end{figure}
\medskip

The extended Newton cone of $f$ can be connected to the closure of the 
Newton cone $\Newton^\perp(f)$: 

\begin{lemma}\label{L:closofnf}
Let $f\in K[[\Gamma]]$, and let $\mathfrak{m}_\tau$ denote the ideal of 
$K[[\Gamma]]$ generated by all the monomials with exponents outside $\tau^\perp$, 
where $\tau$ is a face of $\sigma$. Then:
$$\overline{\Newton^\perp(f)}\cap (N/N_{\tau})_\R \subseteq 
\Newton^\perp(f_\tau)$$
with equality if $f$ is not contained in $\mathfrak{m}_\tau$. Here, 
$\overline{(\cdot)}$ denotes the closure in $\overline{\sigma}$ and each 
$\Newton^\perp(f_\tau)$ is naturally embedded in the stratum $(N/N_{\tau})_\R$. 
The only cases when $f$ is contained in $\mathfrak{m}_\tau$ but still there is an
equality in the formula above is when $f=0$, $\tau=\{0\}$ or $\tau=\sigma$. 
In particular, we have: 
$$\overline{\Newton^\perp(f)}\subseteq \bigsqcup_{\tau\leq\sigma} 
\Newton^\perp(f_\tau)$$
with equality if and only if $f$ is not contained in any of $\mathfrak{m}_\tau$, 
$\tau\leq\sigma$. 
\end{lemma}

\begin{proof}
  If $f_\tau=0$, the inclusion is trivial since $\Newton^\perp(f_\tau)$ is the 
  projection of the cone $\sigma$ to $(N/N_{\tau})_\R$. The equality in the 
  cases $f=0$, $\tau=\{0\}$ or $\tau=\sigma$ can be checked directly. The 
  reason why these are the only exceptions is that if $f\ne 0$, then 
  $\Newton^\perp(f)$ is a \emph{proper} PL cone in $\sigma$, and its limit
  points in $(N/N_{\tau})_\R$ cannot be the whole projection of $\sigma$
  unless $\tau=\{0\}$ or $\tau=\sigma$.

  Now assume that $f\notin\mathfrak{m}_\tau$, so that $f_\tau\ne 0$. Let $v^\tau$ 
  be an element of $\overline{\sigma}\cap(N/N_{\tau})_\R$ and $(v_n)_{n\in \N}$ 
  a sequence of elements of $\Newton^\perp(f)$ converging to $v^\tau$. This 
  implies that the corresponding sequence $(v^{\tau}_{n})_{n\in N}$ of projections 
  of $(v_n)_{n\in N}$ to $(N/N_{\tau})_\R$ converges to $v^\tau$ in the usual 
  sense. Let $\mu$ be the minimal value of $v^\tau$ considered as a function on 
  the Newton diagram $\Newton(f_\tau)$. Note that if $\rho$ is the face of 
  $\Newton(f_\tau)$ where $v^\tau$ attains its minimum, then starting from some 
  number $n_0 \in \N$, for $i\geq n_0$ each function $v^{\tau}_{i}$ reaches its 
  minimal value on some face of $\rho$. 
  
  Let $U$ be the open subset of $N_\R$ formed by all functions $v$ that take 
  values strictly greater than $\mu$ on all monomials of $f$ lying on 
  $\Newton(f)$, except possibly those lying on $\tau^\perp$. Then 
  $\overline{U}_\tau$ (see~\eqref{E:opensetoflv}) is an open neighborhood of 
  $v^\tau$ and all elements $v_i$ are contained in $\overline{U}_\tau$ for 
  $i\geq n_1$, for sufficiently large number $n_1$. It follows that all $v_i$ 
  attain their minimal values on the faces of $\rho\subseteq\Newton(f_\tau)=
  \Newton(f)\cap\tau^\perp$ for $i\geq\max\{n_0,n_1\}$. Thus, the face $\rho$ has 
  positive dimension and $v_\tau\in\Newton^\perp(f_\tau)$. This shows that:
  $$\overline{\Newton^\perp(f)}\cap (N/N_{\tau})_\R \subseteq 
  \Newton^\perp(f_\tau).$$

  If $\Newton^\perp(f_\tau)=\varnothing$, then the equality holds. Assume 
  that $\Newton^\perp(f_\tau)$ is nonempty. It follows that there are 
  at least $2$ vertices of $\Newton(f)$ lying on the linear subspace $\tau^\perp$. 
  Let $\overline{v}\in\Newton^\perp(f_\tau)$ and $\rho$ a face of 
  $\Newton(f_\tau)\subseteq\Newton(f)$ where $\overline{v}$ attains its minimum. 
  Let $v\in\sigma$ be any element that projects to $\overline{v}$ under the 
  canonical projection
  $N_\R \stackrel{\pi_{\tau}}\to (N/N_{\tau})_\R.$ 
  Now let $u\in\tau\setminus\{0\}$ be any vector and consider the sequence 
  $v_n=v+nu$, $n\geq 0$. For $n$ big enough $v_n$ attains its minimum on some face 
  of $\Newton(f)\cap\tau^\perp$. This face must be $\rho$, 
  $v_n\in \Newton^\perp(f)$ for $n\gg0$, and $v_n\to \overline{v}$ 
  in $L(N,\sigma)$.
\end{proof}

The following proposition describes the local tropicalization in the
hypersurface case. It should be compared with {\cite[Theorem~2.1.1]{EKL 06}} and 
{\cite[Section~3]{BG 84}}, where a description of the global tropicalization of
a hypersurface is given. The languages used in these papers are different from
ours, still we can interpret them in the language of general tropicalization
developed by us in Section~\ref{S:defoftrop}. To get the setting of \cite{BG 84}, 
and \cite{EKL 06}, take $\Gamma$ to be a free finitely generated abelian group
$G$, $R$ the coordinate ring of a subvariety of the torus $\spec K[G]$, and
$\gamma\colon G\to R$ the natural morphism. The field $K$ is supposed to be
endowed with an arbitrary (not necessarily trivial) real valuation. This leads
to a global tropicalization that is a PL (but not necessarily  
conical) subspace. Notice that in the case of local tropicalization we must restrict
to trivially valued fields. Indeed, if $v$ is a nonnegative valuation on a
ring $R$, then $v$ must be trivial on any subfield $K\subseteq R$.

\begin{proposition}[Local tropicalization in the hypersurface case]
\label{P:hypersurface}
Let $f\in K[[\Gamma]]$ be a non-invertible series. Then, the nonnegative 
tropicalization $\nntrop(f)$ of the natural semigroup morphism 
$\Gamma\to K[[\Gamma]]/(f)$ coincides with the extended Newton cone of $f$.
The closure of the positive tropicalization $\ptrop(f)\subseteq 
\overline{\sigma}^\circ$ in $\overline{\sigma}$ is $\nntrop(f)$. 
\end{proposition}

\begin{proof}
Apply Theorem~\ref{Tropinit} and Lemma~\ref{L:closofnf} to the principal ideal
$I=(f)$. We leave the details to the reader.
\end{proof}

We are ready to state the finiteness theorem of local tropicalization.

\begin{theorem}[The local finiteness theorem]\label{T:main}
  Let $I$ be a prime ideal of the ring $K[[\Gamma]]$ of formal power series over
  an affine pointed semigroup $\Gamma$, and assume that the Krull dimension of the 
  quotient ring $K[[\Gamma]]/I$ equals $d$. Then, the local positive tropicalization 
  $\ptrop(I)\subseteq\overline{\sigma}^\circ$ and the local nonnegative 
  tropicalization $\nntrop(I)\subseteq\overline{\sigma}$ are rational PL 
  conical subspaces of pure dimension $d$ and the closure of $\ptrop(I)$ in the 
  space $\overline{\sigma}$ is $\nntrop(I)$.
\end{theorem}

\begin{proof}
  If the set $I\cap\Gamma$ is nonempty, it is a prime ideal of the semigroup
  $\Gamma$. Therefore, it must equal $\Gamma\setminus\tau^\perp$ for some 
  face $\tau$ of the cone $\sigma$. In this case:
  $$\nntrop(I)=\nntrop(I_\tau),$$
  where $I_\tau$ is the ideal of $K[[\Gamma_{\tau}]]$ generated by all
  $\tau$-truncations $f_\tau$ of $f\in I$, and the same for the positive 
  tropicalization. The right hand side of the equality above is a subset of
  $L(N/N_\tau,\pi_\tau(\sigma))$, which in turn is naturally a subset of
  $L(N,\sigma)$. The equality itself follows from the commutative diagram:
  $$\xymatrix{
  & K[[\Gamma]]/(\Gamma\setminus\tau^\perp) \ar[r] \ar@{=}[d] 
  & K[[\Gamma]]/I \ar@{=}[d] \\
  & K[[ \Gamma_{\tau} ]] \ar[r] & K[[ \Gamma_{\tau} ]]/I_\tau.
  }
  $$
  Also notice that our theorem is obvious if $I=\{0\}$. 
  
  Thus, \emph{in the 
  rest of the proof we assume that $I$ is monomial free and nonzero}. 
  Let $\mathcal{B}=\{f_1,\dotsc,f_m\}$ be an extended tropical basis for $I$ 
  (see Definition \ref{extropbas}). 
  From Theorem~\ref{Tropinit}, Definition~\ref{D:tbasis}, and 
  Proposition~\ref{P:hypersurface} we deduce that:
  \begin{equation}\label{E:nntropasintersection}
  \nntrop(I)=\bigcap_{i=1}^{m}\left(\bigsqcup_{\tau\leq\sigma}
  \Newton^\perp((f_i)_\tau)\right).
  \end{equation}
  We conclude that $\nntrop(I)$ is a rational PL conical subspace. By 
  Theorem~\ref{Tropinit} (ii), the same holds for the positive tropicalization.
  We have:
  \begin{equation}\label{E:ltropasintersection}
  \ptrop(I)=\left(\bigcap_{i=1}^{m} \left(\bigsqcup_{\tau\leq\sigma}
  \Newton^\perp((f_i)_\tau)\right)\right) \cap \overline{\sigma}^\circ.
  \end{equation}
  The last assertion of the theorem also follows directly from 
  \eqref{E:nntropasintersection} and \eqref{E:ltropasintersection} and the 
  equality $\overline{\Newton^\perp(f_i) \cap\mathring{\sigma}}=
  \Newton^\perp(f_i)$, where the closure is taken in $\sigma$. Thus, it remains 
  only to prove the assertion about the dimension of the local positive 
  tropicalization.

  For any $f=\sum a_m \chi^m\in K[[\Gamma]]$ and $w\in N\cap\mathring{\sigma}$,  
  consider a deformation: 
  $$f_t=t^{-w(f)}\sum a_m t^{\langle w,m\rangle} \chi^m\in K[t][[\Gamma]]$$
  of $f$, where $K[t][[\Gamma]]$ is the ring of formal power series over $\Gamma$ 
  with coefficients in the polynomial ring $K[t]$, and let 
  $I_t\subseteq K[t][[\Gamma]]$ be the ideal generated by all $f_t$, $f\in I$. 
  Write $S=K[t][[\Gamma]]/I_t$. Notice that for any fixed $t_0\in K$, $t_0\ne 0$, 
  the assignment $\chi^m\mapsto t_{0}^{\langle w,m\rangle} \chi^m$ defines an 
  automorphism of the ring $K[[\Gamma]]$.

  Let us show that for every $w\in\ptrop(I)$, there exist a maximal face of 
  $\ptrop(I)$ of dimension $d$ containing $w$. We already know that $\ptrop(I)$ is 
  a rational PL conical subspace, thus its rational points (with respect to the 
  lattice $N=N(\Gamma)$) are dense in it. Thus, it is sufficient to prove our claim
  for an integer point $w\in N$. Furthermore, in a neighborhood of the point $w$ 
  the local positive tropicalization of the ideal $I$ coincides with that of the 
  initial ideal $\init_w(I)$ (Remark~\ref{R:locstructoftrop}).

  Consider the deformation $I_t\in K[t][[\Gamma]]$ of the ideal $I$. By 
  Lemma~\ref{L:flat} below, we may view $K[t][[\Gamma]]/I_t$ as a flat 
  family of schemes over $\spec K[t]$. The special fiber over $0$ of this 
  family is $\spec K[[\Gamma]]/\init_w(I)$. Also, the ring $K[t][[\Gamma]]/I_t$ 
  is equidimensional (Lemma~\ref{L:equidim}). Now from 
  \cite[21.B Theorem~50]{M 70} it follows that all irreducible components of the 
  special fiber have the same dimension as a general fiber $\spec K[[\Gamma]]/I$, 
  namely $d$. In other words, every minimal associated prime of 
  $\init_w(I)\subset K[[\Gamma]]$ has depth $d$. If $\mathfrak{p}_1 ,\dotsc, 
  \mathfrak{p}_k$ are all these minimal primes, then we have the decomposition 
  $\ptrop(\init_w(I))=
  \cup_{1 \leq i \leq k} \ptrop(\mathfrak{p}_i)$ by Lemma~\ref{redval}. Choose 
  $\mathfrak{p}^*$ such that $w\in\ptrop(\mathfrak{p}^*)$.

  Notice that the ideal $J^*=\init_w(I)\subset K[[\Gamma]]$ is generated by
  power series which are in fact polynomials. Thus we can also consider the ideal $J$
  generated by the same polynomials inside the ring $K[\Gamma]$. It is contained 
  in the maximal ideal $\Gamma^+ = \Gamma\setminus\{0\}$, thus by standard
  theory of completions (\cite[Proposition~10.13, Corollary~11.19]{AM 69}
  \cite[Corollary~17.9, 17.12]{N 62}) we conclude:
  $$K[[\Gamma]]/J^*\simeq \widehat{K[\Gamma]/J},\quad \dim K[\Gamma]/J=
  \dim K[[\Gamma]]/J^*=d.$$
  The positive tropicalization of $J^*$ is just the part of the usual 
  tropicalization of $J$ contained in $\overline{\sigma}^\circ$ (see 
  Proposition~\ref{P:localglobal}). Therefore, since $\dim K[\Gamma]/J=d$, we 
  already see that the faces of $\ptrop(J^*)$ passing through $w$ have dimension 
  not greater than $d$. On the other hand, $\mathfrak{p}=\mathfrak{p}^*\cap 
  K[\Gamma]$ is a prime containing the ideal $J$, thus 
  $\dim K[\Gamma]/\mathfrak{p}\leq d$. Now, let 
  $\widehat{\mathfrak{p}}=\mathfrak{p} K[[\Gamma]]$. We have again:
  $$\dim K[[\Gamma]]/\widehat{\mathfrak{p}}=\dim K[\Gamma]/\mathfrak{p}.$$
  But $\widehat{\mathfrak{p}}\subseteq \mathfrak{p}^*$, hence 
  $\dim K[[\Gamma]]/\widehat{\mathfrak{p}}\geq d$. We conclude that
  $\dim K[\Gamma]/\mathfrak{p}=d.$ 
  By properties of the usual tropicalization, $\trop(\mathfrak{p})$ is purely 
  $d$-dimensional. Clearly, $w\in\trop(\mathfrak{p})$. This implies our claim and 
  the theorem.
\end{proof}

In the following three lemmas we keep the notations of the proof of 
Theorem~\ref{T:main}.

\begin{lemma}[cf. {\cite[Theorem~15.17]{Eis 04}}]\label{L:flat}
  The $K[t]$-algebra $S$ is flat.
\end{lemma}
  
\begin{proof}
  Fix a monomial ordering $\preceq$ refining the preorder $\preceq_w$. 
  By Proposition~\ref{P:basisquot}, any class of $K[[\Gamma]]\mod I$ 
  has a unique representative of the form:
  $$\sum_{m\in\Gamma} a_m \chi^m,$$
  where $a_m=0$ for each $\chi^m \in\init_\preceq(I)$. 
  We claim that $\Gamma\setminus\init_\preceq(I)$ plays a similar role for $S$ over 
  $K[t]$. Indeed, if $g\in S$ and $\sum a_m(t) \chi^m$ is a representative of $g$ 
  in $K[t][[\Gamma]]$, let $a_m(t) \chi^m$ be its leading term (with respect to 
  $\preceq$) with $m\in\init_\preceq(I)$. Then, we can find an element $f\in I_t$ 
  with leading term $c\chi^m$, $c\in K^*$. Take the reduction $g-(1/c)a(t)f$. In 
  this way, we delete from $g$ all monomials contained in the initial ideal 
  $\init_\preceq(I)$. Moreover, under such a reduction the terms of $g$ that are 
  less than $\chi^m$ remain unchanged. Thus, despite the reduction process is 
  infinite, the terms that are less than a given monomial $\chi^m$ can change only
  a finite number of times. This shows that the result of the reduction is
  an element of $K[t][[\Gamma]]$.

  We see, in particular, that $S$ has no torsion as a $K[t]$-module. Over a 
  principal ideal domain this is equivalent to being flat 
  (\cite[Corollary~6.3]{Eis 04}).
\end{proof}

\begin{lemma}\label{L:prime}
  If $I\subset K[[\Gamma]]$ is a prime ideal, then $I_t\subset K[t][[\Gamma]]$
  is also prime.
\end{lemma}
\begin{proof}
  Assume that $ab\in I_t$, with $a,b\in K[t][[\Gamma]]$:
  $$a=\sum a_m(t) \chi^m,\quad b=\sum b_m(t) \chi^m.$$
  Fix a monomial ordering refining the $w$-partial ordering.    
  Substituting $t=1$ to $a$ and $b$ we get $a(1)b(1)\in I_{1}=I$. First assume 
  that $a(1)\ne 0$, $b(1)\ne 0$. Since the ideal $I$ is prime, one of $a(1)$, 
  $b(1)$, say $a(1)$, is contained in $I$. Consider the deformation $g=(a(1))_t$ 
  of $a(1)$. After a choice of an appropriate coefficient $c(t)$ the first monomial 
  of the the reduction $a-c(t)g$ which is not $0$ at $t=1$ is less than that of 
  $a$. Notice that $(a-c(t)g)b\in I_t$, thus we may take $a-c(t)g$ instead of $a$. 
  Repeating this argument, we come to a situation when either one of the series in 
  the product is $0$, and thus $a$ or $b\in I_t$, or every coefficient of every 
  term of $a$ or $b$ takes value $0$ under the substitution $t=1$. In this case 
  every coefficient of $a$ (or $b$) is divisible by $t-1$. Then we have a relation of the form:
  $$(t-1)^k a'b'\in I_t,$$
  where $a'(1)$ and $b'(1)\ne 0$. But since the algebra $S=K[t][[\Gamma]]/I_t$ has 
  a basis consisting of monomials, it follows that $a'b'\in I_t$. Notice that 
  after a finite number of the previous two steps the initial monomial $\chi^m$ of 
  $a$ or $b$ will drop with respect to the chosen monomial ordering. After this, 
  new reductions involve only the monomials strictly greater than $\chi^m$. This 
  implies the convergence of the process of reduction of $a$ and $b$. Thus, 
  $a\in I_t$ or $b\in I_t$, as we wanted to show.
\end{proof}

\begin{lemma}\label{L:equidim}
  The ring $K[t][[\Gamma]]/I_t$ is equidimensional, that is, if $\mathfrak{m}_1$ 
  and $\mathfrak{m}_2$ are any two maximal ideals in this ring, then the height of 
  $\mathfrak{m}_1$ equals the height of $\mathfrak{m}_2$.
\end{lemma}
\begin{proof}
  Since the ideal $I_t$ is prime (Lemma~\ref{L:prime}), it suffices to show that
  the ring $R=K[t][[\Gamma]]$ is equidimensional. Let $\mathfrak{m}\subset R$ 
  be a maximal ideal. The crucial observation is that any series
  $\sum_{m\in\Gamma} a_m(t)\chi^m$ that begins with a non-zero constant $a_0(t)=
  a_0\in K$, $a_0\ne 0$, is invertible. It follows that:
  $$\mathfrak{m}_0=\{a_0(t)\,|\,a_0(t)+\sum_{m\in\Gamma^+} a_m(t)\chi^m \in
  \mathfrak{m}\}$$
  is a \emph{proper} ideal of $K[t]$. Moreover, $\mathfrak{m}_0\ne \{0\}$ because
  otherwise $\mathfrak{m}$ would not be maximal (a bigger ideal would be, e.g.,
  $\mathfrak{m}+(t)$). Let $\mathfrak{m}'=\mathfrak{m}\cap K[t]$.
  If $\mathfrak{m}'=\{0\}$, consider the ideal $\mathfrak{m}+(f(t))$, where
  $f$ generates $\mathfrak{m}_0$. It contains $\mathfrak{m}$ as a proper subset. On 
  the other hand, $\mathfrak{m}+(f(t))=(1)$ is impossible because this would imply 
  that $f$ is invertible. This contradicts the maximality of $\mathfrak{m}$.
  Thus, $\mathfrak{m}'=\mathfrak{m}_0=(f)$, and $f$ is irreducible in $K[t]$.
  Furthermore, we have $R/Rf\simeq L[[\Gamma]]$, where $L$ is a finite algebraic
  extension of $K$. The ideal $\mathfrak{m}$ maps to the maximal ideal of
  $L[[\Gamma]]$ under the canonical projection $R\to R/Rf$. Since $L[[\Gamma]]$
  is a finite $K[[\Gamma]]$ module, both algebras have the same Krull
  dimension, equal to the height of their maximal ideals. Let $\dim K[[\Gamma]]=d$.
  The height of $\mathfrak{m}$ equals the dimension of the localization 
  $R_\mathfrak{m}$, and, since $f\in\mathfrak{m}$, $R/Rf\simeq 
  R_\mathfrak{m}/R_\mathfrak{m}f$. By \cite[Corollary~11.18]{AM 69}, we get
  $\dim R_\mathfrak{m}/R_\mathfrak{m}f = d+1$. This number is independent of 
  $\mathfrak{m}$.
\end{proof}

\begin{corollary}
  Let $I$ be an ideal of the formal power series ring $K[[\Gamma]]$. Then 
  $\nntrop(I)$ and $\ptrop(I)$ are rational PL conical subspaces in 
  $\overline{\sigma}$ and $\overline{\sigma}^\circ$ respectively. 
\end{corollary}

\begin{proof}
  It suffices to consider the nonnegative tropicalization. Let $\mathfrak{p}_1$, 
  $\dots$, $\mathfrak{p}_k$ be the minimal associated primes of the ideal $I$. It 
  follows from Lemma~\ref{redval} that:
  $$\ptrop(I)=\bigcup_{i=1}^{k} \ptrop(\mathfrak{p}_i).$$
  But each of $\ptrop(\mathfrak{p}_i)$ is a rational PL conical subspace by
  Theorem~\ref{T:main}.
\end{proof}

Now, let us pass to a more general setting, which applies when we aim to
tropicalize a family of schemes or varieties over a field $K$, as explained 
after the proof of the next theorem: 

\begin{theorem} \label{T:moregen}
  Let $\Gamma$ be an arbitrary affine pointed semigroup. Let $\gamma\colon
  \Gamma\to (R, \cdot)$ be a local morphism, where $(R, \mathfrak{m})$ is  
  a complete local ring. Assume that $R$ contains a field $K$ and consider the 
  induced local morphism of rings $\overline{\gamma}\colon K[[\Gamma]]\to R$. 
  If $R$ is either:
  \begin{itemize}
    \item[a)] integral over $\overline{\gamma}(K[[\Gamma]])$, or
    \item[b)] Noetherian, flat over $K[[\Gamma]]$, and the ideal 
    $(\Gamma^+)R$ is prime,
  \end{itemize}
  then: 
  $$\ptrop(\gamma)=\ptrop(\ker\overline{\gamma}) \text{ and }
  \nntrop(\gamma)=\nntrop(\ker\overline{\gamma}).$$
  In particular, the positive tropicalization $\ptrop(\gamma)$ is a rational 
  PL conical subspace in $\overline{\sigma}^\circ$, and similarly
  the nonnegative tropicalization $\nntrop(\gamma)$ is a rational PL 
  conical subspace in $\overline{\sigma}$.
\end{theorem}

\begin{proof}
  The theorem is a consequence of Theorems~\ref{T:extprinciple} and \ref{T:main}. 
  Indeed, by Theorem~\ref{T:extprinciple} any local valuation
  on $K[[\Gamma]]/\ker\overline{\gamma}$ extends to a local valuation on $R$. On
  the other hand, any local valuation on $R$ obviously restricts to a local 
  valuation on $K[[\Gamma]]/\ker\overline{\gamma}$. Thus, we have the equality
  $\ptrop(\gamma)=\ptrop(\ker\overline{\gamma})$. The proof for the nonnegative
  tropicalization is similar.
\end{proof}

We explain now how tropicalization of families can be studied in the 
framework of \emph{relative tropicalization}. Let $\Gamma$ be an affine pointed
semigroup and $I$ be an ideal of the ring $K[[\Gamma]]$.
Consider also the semigroup $\langle t\rangle=\Z_{\geq 0}$, which will be treated 
as a multiplicative semigroup generated by $t$. The corresponding semigroup 
power series ring with coefficients in the field $K$ is isomorphic to the formal 
power series ring $K[[t]]$ in one variable $t$. If $\lambda\colon 
\langle t\rangle\to \Gamma$ is a local semigroup morphism, we get an induced
morphism of complete local rings $K[[t]]\to K[[\Gamma]]/I$ and a linear map 
$\trop(\lambda)$ of the positive tropicalizations: 
$$\trop(\lambda)\colon \ptrop(I)\to \ptrop(\langle t\rangle)=\ri_{>0}.$$

Let $\varphi\colon L(\sigma,N(\Gamma))\to L(\R_{>0},\Z)$ be the linear map 
inducing $\trop(\lambda)$. Since $\ptrop(I)$ is a rational PL conical subspace, 
for any $a \in \Q_{> 0}$ the fiber $(\trop(\lambda))^{-1}(a)$ is 
a finite rational polyhedral complex in the linear variety $\varphi^{-1}(a)$. 
Notice that a valuation on $K[[t]]$ is completely determined by its value on the 
generator $t$. Thus, the fiber $(\trop(\lambda))^{-1}(a)$ admits the following 
interpretation: it is the tropicalization of the valuations on $K[[\Gamma]]/I$ 
extending the valuation $v$ on $K[[t]]$ and such that $v(t)=a$. 

The fiber $(\trop(\lambda))^{-1}(+\infty)$ is the local tropicalization of the 
special fiber of the map $\spec(K[[\Gamma]]/I)\to \spec K[[t]]$ over the unique 
closed point of $\spec K[[t]]$. With the notation of 
Definition~\ref{defloctropter}, we can write: 
$$(\trop(\lambda))^{-1}(a)=
\ptrop(\mathcal{V}_{(S,v_a)}(K[[\Gamma]]/I,\mathfrak{m}),\gamma),$$
where $S$ is the image of the ring $K[[t]]$ in $K[[\Gamma]]/I$ under the 
homomorphism $\lambda$, $v_a$ is the valuation of $S$ determined by the condition
$v_a(t)=a$, $\mathfrak{m}$ is the maximal ideal of $K[[\Gamma]]/I$, and $\gamma$
is the natural morphism of semigroups $\gamma\colon\Gamma\to K[[\Gamma]]/I$. 
We conclude that: 
$$\ptrop(\mathcal{V}_{(S,v_a)}(K[[\Gamma]]/I,\mathfrak{m}),\gamma)$$ 
is a finite rational polyhedral complex, and it has pure dimension $d-1$ if $I$ is 
a prime ideal of depth $d$.

\medskip
\section{Comparison between local and global tropicalization}\label{S:localglobal}

The aim of this section is to explain that the local tropicalization of the germ 
at a closed orbit of a subvariety of a toric variety can be obtained as the 
intersection of the global tropicalization with the linear variety associated to 
the cone describing the closed orbit. 
\medskip

We start with a subscheme $X$ of an affine toric variety $\spec(K[\Gamma])$. If the
toric variety is not normal, we can always pass to its normalization and
lift $X$ to it. By Corollary~\ref{C:normal} and 
Lemma~\ref{L:normalgamma} below, this does not change the tropicalization of $X$.

\begin{lemma}\label{L:normalgamma}
  Let $\Gamma$ be an affine semigroup and $K$ be a field. Then the integral closure
  of $K[\Gamma]_{(\Gamma^+)}$ in its field of fractions is 
  $K[\sat(\Gamma)]_{(\sat(\Gamma)^+)}$.
\end{lemma}
\begin{proof} 
  It is standard that the integral closure of $K[\Gamma]$ in its field of 
  fractions is $K[\sat(\Gamma)]$ (see, e.g., \cite{F 93}).
  By \cite[Proposition 5.12]{AM 69}, the integral closure of 
  $K[\Gamma]_{(\Gamma^+)}$ in its field of fractions is the ring 
  of fractions $S^{-1}K[\sat(\Gamma)]$ of the ring $K[\sat(\Gamma)]$ with respect 
  to the multiplicative subsemigroup $S := K[\Gamma] \setminus (\Gamma^+)$. 
  Let us show that {\em this ring of fractions is equal to the localization}
  $K[\sat(\Gamma)]_{(\sat(\Gamma)^+)}$.

  Consider an arbitrary fraction $f/g\in K[\sat(\Gamma)]_{(\sat(\Gamma)^+)}$, with 
  the property that 
  $f \in K[\sat(\Gamma)]$ and $g \in K[\sat(\Gamma)] \setminus (\sat(\Gamma)^+)$. 
  We want to prove that there exists $h \in K[\sat(\Gamma)]$ such that 
  $g \cdot h \in S$. We use the following classical fact: if $X_1 ,\dotsc, X_l$ 
  are independent variables and $n \in \N^*$, then:
  \begin{equation} \label{prodlin}
       \prod_{\underline{c}_i}(\sum_{j=1}^l \ c_{ij}X_j) = Q(X_1^n ,\dotsc, X_l^n)
  \end{equation}
  where the $l$-uples $\underline{c}_i=(c_{i1} ,\dotsc,c_{il})$ vary among all 
  possible choices of $n$-th roots of unity in $\C^*$, and where 
  $Q \in \Z[X_1 ,\dotsc, X_l]$. This can be proven by elementary Galois-type 
  arguments. More precisely, we get a polynomial in the $n$-th powers of the 
  variables because the left-hand side is invariant under any substitution 
  $X_i\mapsto\eta X_i$, where $\eta$ is an arbitrary $n$-th root of unity. The 
  coefficients are integers because we work in an integral extension of $\Z$, 
  obtained by adjoining the $n$-th roots of unity, and because 
  the left-hand-side is invariant by all the automorphisms of this extension. 
  Moreover, equation \eqref{prodlin} shows that $Q$ is a homogeneous 
  polynomial (of degree $D = n^l$) and that it contains one power $X_i^D$ of each 
  variable among its monomials. 
  
  Denote by $U(X_1 ,\dotsc,  X_l) \in \C[X_1 ,\dotsc, X_l]$ the product of all 
  linear forms of the left-hand side of \eqref{prodlin} which are distinct from 
  $X_1 + \cdots + X_l$. Since the ring $\Z[X_1 ,\dotsc,  X_l]$ is factorial, 
  we see that $U(X_1 ,\dotsc, X_l) \in \Z[X_1 ,\dotsc, X_l]$. 
  Let us rewrite \eqref{prodlin} in the form:
        \begin{equation} \label{prodlinbis}
       (X_1 + \cdots + X_l)\cdot U(X_1 ,\dotsc, X_l) = Q(X_1^n ,\dotsc, X_l^n).
  \end{equation}
  
  Return now to our polynomial $g \in K[\sat(\Gamma)] \setminus (\sat(\Gamma)^+)$. 
  Suppose that there are $l \in \N^*$ non-zero terms in $g$. Choose an order 
  $t_1 ,\dotsc, t_l$  of them, and denote $m_i\in \Gamma$ the exponent of $t_i$. 
  Replace the variables $X_i$ of (\ref{prodlinbis}) by the 
  terms $t_i$. If we choose $n\in \N$ so that $n \cdot m_i \in \Gamma$ for all the 
  exponents $m_i$ of the monomials of $g$ (which is possible by the definition of 
  the saturation), then $Q(t_1^n ,\dotsc,  t_l^n) \in K[\Gamma]$. Moreover, 
  we claim that 
  $Q(t_1^n ,\dotsc,  t_l^n) \in S  = K[\Gamma] \setminus (\Gamma^+)$. 
  If this holds, the proof is finished, as $h = U(t_1^n ,\dotsc,  t_l^n)$ 
  satisfies the desired property $g \cdot h \in S$. 
  
  Let us explain why $Q(t_1^n ,\dotsc,  t_l^n) \in S$. Consider the Newton 
  polyhedron $\mathcal{N}(g) \subset M(\Gamma)_{\R}$ of $g$, i.e., the convex hull 
  of the exponents  of its monomials. The hypothesis that $g \in S$ shows that 
  $\mathcal{N}(g)$ has at least one vertex in $\Gamma^*$. Since $\Gamma^*$ is a 
  face of $\Gamma$, there exists $v \in N(\Gamma)$ which, when seen as a function 
  on the vertices of $\mathcal{N}(g)$, attains its minimum on exactly one vertex, 
  which is moreover contained in $\Gamma^*$. Assume that it is the vertex $m_1$. 
  {\em Then, the exponent $D m_1$ appears in $Q(t_1^n ,\dotsc,  t_l^n)$ only once, 
  coming from the monomial $X_1^D$ of $Q(X_1^n ,\dotsc,  X_l^n)$}. Indeed, suppose 
  that $X_1^{a_1} \cdots X_l^{a_l}$ is any other monomial of 
  $Q(X_1^n ,\dotsc,  X_l^n)$. The exponent of the term 
  $t_1^{a_1} \cdots t_l^{a_l}$ of $K[\Gamma]$ is $a_1 m_1+ \cdots + a_l m_l$. 
  As $Q$ is homogeneous of degree $D$, we have $\sum_{i=1}^l a_i =D$. Therefore:  
   $$\langle v ,   a_1 m_1+ \cdots + a_l m_l - D m_1 \rangle 
  = \sum_{i =2}^l a_i \langle v ,  m_i - m_1 \rangle.$$ 
  Our hypothesis that the new monomial is distinct from $X_1^D$ shows that at 
  least one of the nonnegative integers $a_2 ,\dotsc,  a_l$ is positive. Choose 
  such an $a_k >0$. Since also $\langle v , m_k - m_1 \rangle>0$ and all the other 
  members $a_i$ and $\langle v  ,  m_i - m_1 \rangle$ in this formula are 
  nonnegative, we conclude that the exponent of $t_1^{a_1} \cdots t_l^{a_l}$ is 
  indeed different from the exponent of $t_1^D$. Therefore, 
  $Q(t_1^n ,\dotsc,t_l^n)\in S$, as it contains the monomial $t_1^D$.
  \end{proof}

Thus, there is no loss in generality if we assume in this section that $\Gamma$ 
is a {\em saturated} affine semigroup. Denote $\Gamma=\sat(\Gamma)=
\check{\sigma}\cap M(\Gamma)$. If $R$ is a ring and $\mathfrak{p}$ a 
prime ideal, let $\psi_{\mathfrak{p}}$ denote the associated morphism of 
localization $\psi_{\mathfrak{p}}\colon R \rightarrow R_{\mathfrak{p}}$. 
The proofs of the following results are easy and left to the reader.

\begin{lemma}
  Let $\Gamma \overset{\gamma}{\longrightarrow} (R, \cdot)$ be a
  morphism of semigroups and let $I(\gamma)\subset R$ be the ideal generated by
  the image $\gamma(\Gamma^+)$. Let $\mathfrak{p}$ be a prime ideal of $R$
  containing $I(\gamma)$ (that is, a point of the subscheme of $\spec
  R$ defined by the ideal $I(\gamma)$). Then the morphism of
  semigroups  $\Gamma \overset{\psi_{\mathfrak{p}} \circ
      \gamma}{\longrightarrow} (R_{\mathfrak{p}}, \cdot)$ satisfies
      $\psi_{\mathfrak{p}} \circ \gamma(\Gamma^+) \subset \mathfrak{p}
      R_{\mathfrak{p}}$. 
\end{lemma}

\begin{proposition} \label{globloc}
  Let $R$ be a ring, $\mathfrak{p}$ be one of its prime ideals and
  $(\Gamma,+) \overset{\gamma}{\longrightarrow} (R, \cdot)$ a
  morphism of semigroups such that 
  $\gamma(\Gamma^+)\subset  \mathfrak{p}$. Then, for any subspace $\W
  \subset \V(R)$, we have:  
  $$\trop(\W, \gamma) \cap \overline{\sigma}^\circ(\Gamma) =
  \ptrop(\V(\psi_{\mathfrak{p}})^{-1}(\W), \gamma).$$ 
\end{proposition}

In particular, we get the following property of subschemes of toric varieties, 
comparing local and global tropicalization:

\begin{corollary}
  Let $X$ be a subscheme of a toric variety $\zcal(\Delta, N)$. Let 
  $A\in X$ be a closed point which is an orbit $O_{\sigma}$ of 
  $\zcal(\Delta, N)$, where $\sigma$ is a cone of $\Delta$ with non-empty 
  interior. Then: $\ptrop(X,A)=\trop(X) \cap \overline{\sigma}^\circ$. 
\end{corollary}

We would like to emphasize the special case used in the proof of 
Theorem~\ref{T:main} (which holds for arbitrary, not necessarily saturated, 
pointed affine semigroups):

\begin{proposition}\label{P:localglobal}
  Let $\Gamma$ be an affine pointed semigroup, let $I$ be an ideal of the ring 
  $K[\Gamma]$ contained in the maximal ideal $(\Gamma^+)$, and let $\widehat{I}$ 
  be the extension of $I$ in the power series ring $K[[\Gamma]]$. Then:
  $\ptrop(\widehat{I})=\trop(I)\cap\overline{\sigma}^{\circ}.$
\end{proposition}

\begin{proof}
  If a ring valuation $v$ of $K[\Gamma]$ is nonnegative on $\Gamma$ and positive
  on $\Gamma^+$, then it is nonnegative on the whole ring $K[\Gamma]$ and positive 
  on the maximal ideal $(\Gamma^+)$. Thus, the valuation $v$ canonically extends 
  to a local valuation of the ring $K[[\Gamma]]$. Conversely, any local
  valuation $w$ of $K[[\Gamma]]$ restricts to a nonnegative valuation of 
  $K[\Gamma]$, which is positive on the maximal ideal $(\Gamma^+)$.
\end{proof}

In fact, we can reconstruct the global tropicalization of a subvariety or a 
subscheme $X$ over a field $K$ of a toric variety $\zcal(\Delta, N)$ from 
the local tropicalizations of the germs of this subscheme at the orbits of some
birational modification of $\zcal(\Delta, N)$. If $X$ does not pass through any
such orbit (e.g., $X=1\in\torus$), then the global tropicalization of $X$
consists of one point and there is essentially nothing to reconstruct. So, let us
suppose that this is not the case. 

\begin{notation}\label{N:orbits}
Let $\sigma$ be a cone of $\Delta$, 
and $O_\sigma$ the corresponding orbit of the big torus in $\zcal(\Delta, N)$. 
$O_\sigma$ is the unique closed orbit of the affine toric variety $\zcal(\sigma,N)=
\spec K[\check{\sigma}\cap M]$. We denote the semigroup
$\check{\sigma}\cap M$ by $\Gamma$. Assume that the orbit $O_\sigma$ is contained 
in the subscheme $X$. Let $I_{X,\sigma}$ denote the ideal of 
$X$ in the local ring $K[\Gamma]_{(\Gamma^+)}$, and $\widehat{I_{X,\sigma}}$ the 
corresponding ideal in the completion $K(\Gamma^*)[[\Gamma']]$ of 
$K[\Gamma]_{(\Gamma^+)}$ at its maximal ideal (see 
Section~\ref{S:powerseriesring}). We have the positive tropicalization
$\ptrop(\widehat{I_{X,\sigma}})=\ptrop(X,\sigma)$, which is a PL conical subspace in 
$\overline{\sigma}^\circ$, and the nonnegative tropicalization 
$\nntrop(\widehat{I_{X,\sigma}})=\nntrop(X,\sigma)$, which is a PL conical subspace 
in $\overline{\sigma}$. These tropicalizations are well defined due to the 
following result:
\end{notation}

\begin{proposition}\label{P:toroidalinvariance}
  Let $I$ be an ideal of the power series ring $K[[\Gamma]]$, where $K$ is an
  arbitrary field. Let $\Phi$ be an 
  automorphism of $K[[ \Gamma]]$ sending each element of $\Gamma$ to a product of 
  itself by a unit of $K[[ \Gamma]]$. Then, the positive and the nonnegative
  tropicalizations of $I$ and of $\Phi(I)$ coincide. 
\end{proposition}

\begin{proof}
  If $v$ is any nonnegative ring valuation of $K[[\Gamma]]$, then $v(u)=0$ for any
  unit $u$ of $K[[\Gamma]]$. It follows that $v(\Phi(f))=v(f)$ for all $f\in
  K[[\Gamma]]$.
\end{proof}

Proposition~\ref{P:toroidalinvariance} shows that the local tropicalization 
of a germ of subvariety of an affine toric variety at the unique closed orbit 
depends only on the \emph{toroidal structure} in the neighborhood of that orbit. 
In Section~\ref{Toroidal} we will use this fact to define tropicalization of 
subvarieties of algebraic toroidal embeddings. As a first application of the 
previous proposition, we generalize Proposition~\ref{P:localglobal}.

\begin{proposition}
  Let $\Gamma$ be a saturated affine semigroup, and $I$ an ideal of $K[\Gamma]$ 
  contained in the ideal $(\Gamma^+)$. Fix an isomorphism 
  $\widehat{K[\Gamma]_{(\Gamma^+)}}\simeq K(\Gamma^*)[[\Gamma']]$ and let 
  $\widehat{I}$ be the extension of $I$ in the ring $K(\Gamma^*)[[\Gamma']]$.
  Then:
  $$\ptrop(\widehat{I})=\trop(I)\cap\overline{\sigma}^{\circ}.$$
\end{proposition}

\begin{proof}
  Recall that an isomorphism between the completion of $K[\Gamma]_{(\Gamma^+)}$ 
  and $K(\Gamma^*)[[\Gamma']]$ is defined up to a unit. By 
  Proposition~\ref{P:toroidalinvariance}, the positive tropicalization
  $\ptrop(\widehat{I})$ does not depend on the isomorphism between 
  $\widehat{K[\Gamma]_{(\Gamma^+)}}$ and $K(\Gamma^*)[[\Gamma']]$. Then the proof 
  goes along the same lines as the proof of Proposition~\ref{P:localglobal}.
\end{proof}

Now let $X$ be a subscheme of a toric variety $\zcal(\Delta, N)$. We 
use Notation~\ref{N:orbits}.

\begin{lemma}\label{L:compatibilityofltrop}
  Let $\tau$ be a face of $\sigma$. Assume that $O_\tau\subseteq X$. Then: 
  $$\nntrop(X,\sigma)\cap(\overline{\tau}^\circ)=\ptrop(X,\tau),$$
  or, equivalently:
  $$\overline{\ptrop(X,\sigma)}\cap(\overline{\tau}^\circ)=\ptrop(X,\tau).$$
\end{lemma}

\begin{proof}
  Let $\Gamma=\check{\sigma}\cap M$,
  $\Gamma(\tau)=\check{\tau}\cap M(\Gamma)$. We have the following diagram of 
  rings and ideals:
  $$
  \xymatrix{
  & K(\Gamma^*)[[\Gamma']]  & K[\Gamma]_{(\Gamma^+)}\ar[l]_a \ar[r]^c
  & K[\Gamma]_{(\Gamma\setminus\Gamma_\tau)}\ar[r]^(.4){b}
  & K(\Gamma(\tau)^*)[[\Gamma(\tau)']] \\
  & \widehat{I_{X,\sigma}}\ar@{^{(}->}[u]  
  & I_{X,\sigma}\ar@{^{(}->}[u]\ar[l]\ar[r] & I_{X,\tau}\ar@{^{(}->}[u]\ar[r]  
  & \widehat{I_{X,\tau}}\ar@{^{(}->}[u]
  }
  $$
  where $c$ is the morphism of localization, $a$ is the composition of
  the natural morphism of a local ring to its completion with the fixed 
  isomorphism $\widehat{K[\Gamma]_{(\Gamma^+)}}\simeq K(\Gamma^*)[[\Gamma']]$.
  $b$ is defined similarly to $a$, and the arrows in the second row are 
  induced by the arrows in the first. 

  Now, let $\overline{v}$ be a valuation of the ring $K(\Gamma^*)[[\Gamma']]$
  (infinite on the ideal $\widehat{I_{X,\sigma}}$) inducing an element 
  $w\in\nntrop(X,\sigma)\cap(\overline{\tau}^\circ)$.
  Let $v$ be the restriction of $\overline{v}$ to $K[\Gamma]_{(\Gamma^+)}$. 
  Since $v$ takes only value $0$ on the subsemigroup $\chi^{\Gamma_\tau}$, we can push
  it forward to the localization $K[\Gamma]_{(\Gamma\setminus\Gamma_\tau)}$ and, 
  since $v$ is positive on the ideal $\chi^{(\Gamma\setminus\Gamma_\tau)}$, we can 
  further push it forward to a local valuation of 
  $K(\Gamma(\tau)^*)[[\Gamma(\tau)']]$, thus producing an element of 
  $\ptrop(X,\tau)$. Going to the opposite direction, we can easily show
  that any local valuation of $K(\Gamma(\tau)^*)[[\Gamma(\tau)']]$ (infinite on
  the ideal $\widehat{I_{X,\tau}}$) defines a nonnegative valuation on
  $K(\Gamma^*)[[\Gamma']]$, positive on the ideal $(\Gamma\setminus\Gamma_\tau)$
  and trivial on the subsemigroup $\chi^{\Gamma_\tau}$.
\end{proof}

As a consequence of the results of this section, we get the following theorem
describing the connection between the global tropicalization of a subvariety or
a subscheme $X$ of a normal toric variety $\mathcal{Z}(\Delta, N)$ and the local 
tropicalizations of germs of $X$ at the orbits of $\mathcal{Z}(\Delta, N)$.

\begin{theorem} \label{normLocglob}
  Let $\Delta$ be a fan. Let $X$ be a subscheme of the toric variety 
  $\mathcal{Z}(\Delta, N)$ and $\trop(X)\subseteq L(\Delta,N)$ be the 
  tropicalization of $X\subseteq\zcal(\Delta, N)$ in the sense of 
  Remark \ref{specnot}. If $\sigma$ is a cone of $\Delta$ such that 
  $O_\sigma\subseteq X$, then:
  $$\trop(X)\cap \mathring{\sigma}=\ptrop(X,\sigma).$$
\end{theorem}

By Corollary \ref{C:normal} and Lemma \ref{L:normalgamma},   
we get the following generalization of the previous theorem 
to subschemes of arbitrary, not necessarily normal, toric varieties:

\begin{theorem} \label{Locglob}
  Let  $\mathcal{S}$  be a fan of semigroups, with associated fan $\Delta$. 
  Let $X$ be a subscheme of the toric variety 
  $\zcal(\mathcal{S})$ and $\trop(X)\subseteq L(\Delta,N)$ be the 
  tropicalization of $X\subseteq\zcal(\mathcal{S})$ in the sense of 
  Remark \ref{specnot}. If $\sigma$ is a cone of $\Delta$ such that 
  $O_\sigma\subseteq X$, then:
  $$\trop(X)\cap \mathring{\sigma}=\ptrop(X,\sigma).$$
\end{theorem}

If the orbit $O_\sigma$ is not contained in $X$, then it is natural to set by
definition $\nntrop(X,\sigma)=\ptrop(X,\sigma)=\varnothing$. Let us consider a 
particular case when $\zcal(\Delta, N)=\torus$ is simply a torus.
For any subvariety $X\subseteq \torus$ we have the familiar tropicalization
$\trop(X)$. In addition $\Gamma=M(\Gamma)$, $\Gamma^+ =(0)$, $\torus=
\Hom(\Gamma,K^*)$, and $\Delta=(0)$. Then, $K(\Gamma^*)[[\Gamma']]=K(\Gamma)$ is 
the field of rational functions on $\torus$. If $X=\torus$, then $I_X=\{0\}$ and 
the positive and the nonnegative tropicalization consist of the point $\{0\}$ 
corresponding to the trivial valuation on $K(\Gamma)$. If $X$ is a proper 
subvariety, then $\nntrop(X,0)=\ptrop(X,0)=\varnothing$. Still, the 
tropicalization $\trop(X)$ can be reconstructed from local tropicalizations with a 
help of an auxiliary fan.

Some new terminology and notation is in order. Let $\Sigma$ be a PL 
cone (Definition~\ref{D:conicalset}) in an $\R$-vector space 
$V$ and $v$ a point of $\Sigma$. If $\Sigma=\cup\sigma$ is a fan structure on 
$\Sigma$, let $\sigma(v)$ be the unique cone that contains $v$ in its relative 
interior. For the point $v\in \Sigma$ there is a unique subspace 
$T_v\Sigma\subseteq V$ with the following property: $T_v\Sigma$ is the minimal 
(with respect to inclusion) subspace of $V$ such that for any fan structure 
$\Sigma=\cup\sigma$, $T_v\Sigma$ contains the cone $\sigma(v)$. We say that 
$T_v\Sigma$ is the \emph{tangent space} to $\Sigma$ at the point $v$. Now, let 
$\Delta$ be a fan in $V$. Again, for a point $v\in\supp\Delta$, we let 
$\delta(v)$ be the unique cone of $\Delta$ such that $v$ is contained in the 
relative interior of $\delta(v)$. We say that a PL cone  $\Sigma$ and a fan 
$\Delta$ are \emph{transversal} at a point $v\in \supp\Delta\cap\Sigma$ if 
$T_v\Sigma+\langle\delta(v)\rangle=V$, where $\langle\delta(v)\rangle$ is the 
subspace of $V$ spanned by $\delta(v)$. We say that $\Sigma$ and $\Delta$ are 
transversal if they are transversal at each point $v\in\supp\Delta\cap\Sigma$.

\begin{corollary}
Let $X$ be a subvariety of a torus $\torus=\Hom(\Gamma,K^*)$. Let $\Delta$ be
a rational polyhedral fan in $N(\Gamma)_{\R}$ that is transversal to the 
tropicalization $\trop(X)$ of $X$ and such that $\trop(X)$ is contained in
$\supp\Delta$. Then, $\trop(X)$ is a disjoint union of the real parts of all
local positive tropicalizations $\ptrop(X,\sigma)$, $\sigma\in\Delta$.
\end{corollary}

\begin{proof}
We shall only outline the main ideas in the proof, leaving the details to the
reader. It suffices to show that for each point $v$ of $\trop(X)$, the closure of
$X$ in the toric variety $\zcal(\Delta, N)$ contains the orbit $O_{\sigma(v)}$.
The ideal $I_{X,\sigma(v)}$ of the closure of $X$ in the affine toric variety
$\zcal(\sigma(v), N)$ is generated by all polynomials $f\in I_X$ whose support
is contained in $\check{\sigma}(v)$. A sufficient condition for such a polynomial 
$f=\sum a_m\chi^m$ to vanish on $O_{\sigma(v)}$ is that the extended Newton 
diagram $\Newton^+(f)$ is not generated by one point, i.e., there is no 
$m\in\Gamma$ such that $\Newton^+(f)= m+\check{\sigma}(v)$. But this condition 
indeed holds for each $f\in I_{X,\sigma(v)}$ because $v\in \trop(X)$ and 
$\trop(X)$ and $\Delta$ are transversal.
\end{proof}

\medskip
\section{Toroidal meaning of local tropicalization} \label{Toroidal}

In this section, we show that tropicalization is an invariant of the ambient 
toroidal structure. More precisely, the tropicalization of an algebraic, analytic 
or formal germ of subvariety of an affine toric variety at its closed orbit, 
depends only on the associated toroidal structure. We use this fact to define the 
tropicalization of a subvariety of a toroidal embedding. 
\medskip

The basic reference for the notions used in this section is 
\cite[Chapter II]{KKMS}. First, we recall the basic definitions and fix the 
notations. The ground field $K$ will be assumed to be algebraically closed.

\begin{definition}
  (\cite[Chapter II, Definition 1]{KKMS}). A {\bf toroidal embedding} over a field 
  $K$ is a pair $(U,X)$, where $U\subseteq X$ is a Zariski open subset of a normal 
  algebraic variety $X$ over $K$, such that for every closed point $x\in X$ there 
  exists an affine toric variety $(\torus, \mathcal{Z})$ over $K$, where $\torus$ 
  is the open torus $\torus\subseteq \mathcal{Z}$, a closed point 
  $t\in \mathcal{Z}$, and an isomorphism of $K$-local algebras:
  $$\widehat{\mathcal{O}}_{X,x}\simeq \widehat{\mathcal{O}}_{\mathcal{Z},t}$$
  such that the ideal in $\widehat{\mathcal{O}}_{X,x}$ generated by the ideal of
  $X\setminus U$ maps isomorphically to the ideal in 
  $\widehat{\mathcal{O}}_{\mathcal{Z},t}$ generated by the ideal of $\mathcal{Z} 
  \setminus\torus$.
\end{definition}  

Notice that the previous definition implies that $U$ is smooth. The notation 
$(U,X)$, with $U$ coming first, is intended to suggest that $X$ is thought of 
as a total space into which $U$ embedds and that, as for toric varieties, this 
total space may change without changing $U$. 

The orbit of $t\in \mathcal{Z}$ can always be assumed closed, by diminishing 
perhaps $\mathcal{Z}$. Such a pair $(\mathcal{Z},t)$, together with a formal 
isomorphism as above is called a \emph{local model} of $(U, X)$ at $x$. The 
definition implies that the irreducible components of $X\setminus U$ 
(if nonempty) have codimension $1$ in $X$. We denote them by $(E_i)_{i\in I}$, so 
that $X\setminus U=\cup_{i\in I} E_i$. If all the varieties $E_i$ are normal, 
a toroidal embedding $(U,X)$ is called a \emph{toroidal embedding without self 
intersections}. 

\medskip
{\em In the sequel we consider only toroidal embeddings without self intersections.} 
The set $U$ and the connected  components of the sets 
$\cap_{i\in J} E_i\setminus \cup_{i\notin J} E_i$, $J\subseteq I$, define
a natural \emph{stratification} of the space $X$. If $Y$ is a stratum,
the \emph{star} $\Star(Y)$ of $Y$ is the union of all strata $Z$ such that $Y$
is contained in the closure of $Z$.

Let $Y$ be a stratum. Following \cite[Chapter II, Definition 3]{KKMS}, we denote:
\begin{itemize}
  \item 
            $M^Y=$  the group of Cartier divisors on  $\Star(Y)$,\\
                   supported on the hypersurface  $\Star(Y)\setminus U$;
  \item $N^Y=\Hom(M^Y,\Z)$;
  \item $M_{+}^{Y}=$ subsemigroup of $M^Y$ of effective divisors;
  \item $\sigma^Y=\{w\in N_{\R}^{Y}\,|\,\langle w,u\rangle\geq 0 
  \text{ for all } u\in M_{+}^{Y}\}\subseteq N_{\R}^{Y}$.
\end{itemize}
Note that the cone $\sigma^Y$ is strongly convex and that $\mathrm{rk} \ M^Y =  
\mathrm{codim}_X Y$.

\begin{proposition}
  Let $Y$ be a stratum of the toroidal embedding without self-intersection 
  $(U, X)$. Then, the completion $\widehat{\mathcal{O}}_{X,Y}$ of the local ring 
  of $X$ at $Y$ is isomorphic to the ring  $K(Y)[[M_{+}^{Y}]]$ of formal power 
  series over the semigroup $M_{+}^{Y}$ with coefficients in the field $K(Y)$ of 
  rational functions on $Y$. This isomorphism is defined up to multiplication by 
  units.
\end{proposition}

\begin{proof}
  Each Cartier divisor on $\Star(Y)$ defines a principal ideal in the local ring
  $\mathcal{O}_{X,Y}$. Thus, to each element of $M_{+}^{Y}$ we may assign a 
  defining function, i.e., an element of $\mathcal{O}_{X,Y}$ (well-defined up to a 
  unit). We Define this correspondence on a set of elements of $M_{+}^{Y}$ which 
  form a basis of $M_{\R}^Y$ and then extend it to all of $M_{+}^{Y}$. We obtain 
  a morphism of semigroups $M_{+}^{Y}\to\mathcal{O}_{X,Y}$. Notice also that the 
  ring $\mathcal{O}_{X,Y}$ and hence its completion $\widehat{\mathcal{O}}_{X,Y}$, 
  contain a field, say, the field $K$. Then, it follows from the theory of 
  complete rings (see \cite[Chapter V]{N 62}) that $\widehat{\mathcal{O}}_{X,Y}$ 
  contains also a field isomorphic to its residue field, that is to $K(Y)$. Let us 
  fix such a subfield. In this way, we get a morphism of rings: 
  $$\alpha_Y\colon K(Y)[[M_{+}^{Y}]]\to \widehat{\mathcal{O}}_{X,Y},$$
  unique up to a unit. We now prove that it is an isomorphism of complete
  local rings.

  The injectivity of $\alpha_Y$ is clear, so let us prove the surjectivity. Since 
  $\mathcal{O}_{X,Y}$ naturally embeds into its completion 
  $\widehat{\mathcal{O}}_{X,Y}$ and the ring $K(Y)[[M_{+}^{Y}]]$ is complete, it 
  suffices to prove that $\mathcal{O}_{X,Y}$ lies in the image of 
  $K(Y)[[M_{+}^{Y}]]$. First, note that the image of $M_{+}^{Y}$ generates the 
  maximal ideal of $\mathcal{O}_{X,Y}$. Indeed, consider the diagram: 
  $$\mathcal{O}_{X,Y}\hookrightarrow \mathcal{O}_{X,x}\hookrightarrow
  \widehat{\mathcal{O}}_{X,x} \overset{\varphi}{\simeq} 
  \widehat{\mathcal{O}}_{\mathcal{Z},t}$$
  of rings, where $x$ is a closed point of the stratum $Y$ and $(\mathcal{Z},t)$ 
  is a local model at $x$. By the properties of toroidal embeddings (see 
  \cite[Chapter II, Corollary~1]{KKMS}) the ideal of the stratum $Y$ maps
  to the ideal of the closed orbit of $\mathcal{Z}$ under $\varphi$. This last 
  ideal is generated by the image of $M_{+}^{Y}$ in 
  $\widehat{\mathcal{O}}_{\mathcal{Z},t}$. Let $\mathfrak{m}$ be the maximal ideal 
  of $\mathcal{O}_{X,Y}$. We see that $M_{+}^{Y}$ is a subset of $\mathfrak{m}$ 
  and it generates the ideal $\widehat{\mathfrak{m}}$ of the stratum $Y$ in 
  $\widehat{\mathcal{O}}_{X,x}$. But since the ring $\mathcal{O}_{X,x}$ is 
  Noetherian, we conclude that $M_{+}^{Y}$ also generates $\mathfrak{m}$.

  Consider now some $f\in\mathcal{O}_{X,Y}$. Fix a finite subset $\{f_1,\dotsc,f_k\}$ of 
  $M_{+}^{Y}$ which generates the maximal ideal $\mathfrak{m}$. Let $a_0\in K(Y)$ 
  be a representative of the class of $f$ in $\mathcal{O}_{X,Y}/\mathfrak{m}$. 
  Then, $f-a_0\in\mathfrak{m}$ and we can write:
  $$f-a_0=\sum_i g_i f_i,\quad g_i\in\mathcal{O}_{X,Y} \text{ for all } i.$$
  Applying the same argument to $g_i$ we find $a_1 ,\dotsc, a_k\in K(Y)$ 
  such that:
  $$f=a_0+\sum_i a_i f_i \mod \mathfrak{m}^2.$$
  Repeating this argument we represent $f$ as an image of a series in
  $K(Y)[[M_{+}^{Y}]]$. This proves that $\alpha_Y$ is surjective, as we wanted
  to show.
\end{proof}

To each toroidal embedding, we canonically associate a \emph{conical polyhedral 
complex with integral structure}. Let us recall the construction.

\begin{definition}(\cite[Chapter II, Definition 5]{KKMS}).
  A {\bf conical polyhedral complex} $\Delta$ is formed by: 
  \begin{itemize}
    \item a topological space $|\Delta|$; 
    \item a finite family of closed subsets $\sigma_i$ called {\bf cones};
    \item a finite
  dimensional real vector space $V_i$ of real valued continuous functions on 
  $\sigma_i$ such that:
    \end{itemize} 
    
  \begin{enumerate}
  \item a basis of $V_i$ defines a homeomorphism from $\sigma_i$ to a polyhedral
  cone ${\sigma'}_i\subset \R^{n_i}$, not contained in a hyperplane;
  \item faces of ${\sigma'}_i$ correspond also to cones of $\Delta$;
  \item $|\Delta|$ is a disjoint union of relative interiors of $\sigma_i$ for all
  $i$;
  \item if $\sigma_j$ is a face of $\sigma_i$, then the restriction of $V_i$ to
  $\sigma_j$ is $V_j$.
  \end{enumerate}
\end{definition}

\begin{remark}
  Even if we use the same notation as for fans of cones, it is important to note 
  that in a conical polyhedral complex we do not have an embedding of the various 
  cones in a fixed vector space. In particular, if we consider the conical 
  polyhedral complex associated to a fan, we loose the information about this 
  embedding. 
\end{remark}

\begin{definition}(\cite[Chapter II, Definition 6]{KKMS}).
  An {\bf integral structure} on a conical polyhedral complex $\Delta$ is a set
  of finitely generated abelian groups $L_i\subset V_i$ such that:
  \begin{enumerate}
  \item $(L_i)_{\R}\simeq V_i$;
  \item if $\sigma_j$ is a face of $\sigma_i$, then the restriction of $L_i$ to
  $\sigma_j$ is $L_j$.
  \end{enumerate}
\end{definition}

Let $(U,X)$ be a toroidal embedding. Let $Y$ be a stratum, and $Z$ a stratum in
$\Star(Y)$. Then, the canonical surjective map $M^Y\to M^Z$ induces a canonical
inclusion $N_{\R}^{Z}\to N_{\R}^{Y}$ such that $N^Z=N_{\R}^{Z}\cap N^Y$, and if
$Z$ corresponds to the face $\tau$ of $\sigma^Y$, then the inclusion 
$N_{\R}^{Z}\to N_{\R}^{Y}$ maps $\sigma^Z$ isomorphically to $\tau$ 
(see \cite[Chapter II, Corollaries 1 and 2]{KKMS} for the details). Now consider 
the topological space:
      $$|\Delta|=\bigsqcup_Y \sigma^Y / \sim,$$
where the disjoint union is taken over all strata of $(U,X)$ and the equivalence
relation $\sim$ is the gluing of cones along common faces. The triple 
$(|\Delta|,M_{\R}^{Y},M^Y)$ is called the conical polyhedral complex (simply 
\emph{conical complex} in the sequel) of the toroidal embedding $(U,X)$.

For each cone $\sigma^Y$ of the conical complex $\Delta$ we have a linear variety
$L(\sigma^Y,N^Y)$ and the closure $\overline{\sigma^Y}$ (see 
Section~\ref{troptor}). {\em The gluing of cones of $\Delta$ naturally extends to 
a gluing of their closures}. More precisely, let $Y_1$, $Y_2$, and $Z$ be strata of $(U,X)$, 
and suppose that $Y_1$
  and $Y_2$ are contained in the closure of $Z$. Recall that 
  $\overline{\sigma^{Y_1}}$ is defined as the set of all nonnegative simigroup
  homomorphisms from $\sigma^{Y_1}\cap M^{Y_1}$ to $\ri$, and similarly
  $\overline{\sigma^{Y_2}}$ and $\overline{\sigma^Z}$. Since $M^Z$ is naturally
  a sublattice of both $M^{Y_1}$ and $M^{Y_2}$, and $\sigma^Z$ is a common
  face of $\sigma^{Y_1}$ and $\sigma^{Y_2}$, $\Hom_{sg}(\sigma^Z\cap M^Z,
  \ri)$ is a common subset of $\Hom_{sg}(\sigma^{Y_1}\cap M^{Y_1},\ri)$ and
  $\Hom_{sg}(\sigma^{Y_2}\cap M^{Y_2},\ri)$. This allows to glue the extended
  cones $\overline{\sigma^{Y_1}}$ and $\overline{\sigma^{Y_2}}$ along 
  $\overline{\sigma^Z}$. The stratum at infinity of $\overline{\sigma^{Y_1}}$ that 
  corresponds to the face $\sigma^Z$ is equipped with the lattice $N^{Y_1}/N^Z$ and 
  the vector space $(N^{Y_1}/N^Z)_{\R}$. For an illustration in dimension two, 
  see Example~\ref{Ex:Toroitrop} and the accompanying Figure~\ref{fig:Toroitrop}.

\begin{definition}
  Let $\Delta$ be the conical complex of a toroidal embedding $(U,X)$. Denote by 
  $|\overline{\Delta}|=(\bigsqcup_Y \overline{\sigma^Y}) / \sim \ $ 
  the topological space obtained by gluing the extended cones of $\Delta$ as 
  explained before. Equip it with the additional structure $M^Y$ that is 
  inherited from $\Delta$, and with all the analogous additional structure (quotient 
  lattices, vector spaces of real functions) on the strata of 
  $\overline{\sigma^Y}$ at infinity. We call it the 
  {\bf extended conical complex} of the toroidal embedding $(U,X)$, denoted 
  $\overline{\Delta}$.
\end{definition}

Now, let $\mathcal{I}$ be an ideal sheaf on a toroidal embedding $(U,X)$ defining
a subscheme $W$. This sheaf generates an ideal $I^Y$ (perhaps non-proper) in the 
local ring $\mathcal{O}_{X,Y}$ of every stratum $Y$. Fix an isomorphism 
$\widehat{\mathcal{O}}_{X,Y}\simeq K(Y)[[M_{+}^{Y}]]$ and let $\widehat{I}^Y$ be the 
ideal generated by $I^Y$ in $K(Y)[[M_{+}^{Y}]]$. Let $\Gamma$ be the semigroup 
$M_{+}^{Y}$ and $\gamma$ the natural morphism of semigroups: 
$$\gamma\colon  M_{+}^{Y} \to K(Y)[[M_{+}^{Y}]]/\widehat{I}^Y.$$
Then, we have the positive tropicalization $\ptrop(W,Y)=\ptrop(\gamma)$ and the
nonnegative tropicalization $\nntrop(W,Y)=\nntrop(\gamma)$, which are conical
sets in $(\overline{\sigma^Y})^{\circ}$ respectively in $\overline{\sigma^Y}$. 
By Proposition~\ref{P:toroidalinvariance}, these tropicalizations do not depend 
on the choice of an isomorphism between $\widehat{\mathcal{O}}_{X,Y}$ and 
$K(Y)[[M_{+}^{Y}]]$.
  
\begin{lemma}
  Let $Y$ and $Z$ be strata, and $Z\subseteq \Star(Y)$. If $Z\subseteq W$, then:
  $$\nntrop(W,Y)\cap (\overline{\sigma^Z})^{\circ}=\ptrop(W,Z).$$
\end{lemma}
\begin{proof}
  The proof is essentially the same as the proof of 
  Lemma~\ref{L:compatibilityofltrop}.
\end{proof}

This lemma justifies the following definition: 
\begin{definition}
  Let $W$ be a subscheme of a toroidal embedding $(U,X)$. The disjoint union:
  $$\trop(W)=\bigsqcup_Y \ptrop(W,Y)$$
  of positive tropicalizations of all germs of $W$ at strata of $(U,X)$, 
  considered as a subset of the extended conical complex $\overline{\Delta}$ of 
  the toroidal embedding $(U,X)$ is called the {\bf tropicalization} of the 
  subscheme $W$.
\end{definition}

\begin{theorem}
  Let $W$ be a subscheme of a toroidal embedding $(U,X)$. Then for every stratum
  $Y$ of $(U,X)$ the intersection $\trop(W)\cap \overline{\sigma^Y}$ is a rational 
  polyhedral conical set. If the germ of $W$ at $Y$ has pure dimension $d$, then
  $\trop(W)\cap \overline{\sigma^Y}$ has pure real dimension $d$.
\end{theorem}
\begin{proof}
The proof follows from Theorem~\ref{T:main}.
\end{proof}

\begin{example}\label{Ex:Toroitrop}
  In the top part of Figure \ref{fig:Toroitrop} is represented a (singular) curve
  $W$ in a smooth surface $X$, and $E_1 ,\dotsc, E_4$ are smooth curves of $X$ 
  crossing normally in succession at the points $A, B, C$. 
  Therefore, if $U := X \setminus \bigcup_{1 \leq i \leq 4} E_i$, the pair 
  $(X, U)$ is a toroidal embedding. In the bottom part of the figure we 
  represent the associated tropicalisation, which is obtained by gluing the 
  positive local tropicalisations in the neighborhood of the points $A, B, C$. 
  We denote by $\sigma_P$ the 2-dimensional cone corresponding to each point 
  $P \in \{A, B, C\}$, and by $\tau_i$ the $1$-dimensional cone corresponding to 
  the curve $E_i$, for each $i \in \{1 ,\dotsc,  4\}$. Notice that at the point $C$ we 
  have two irreducible components of $W$, but that their tropicalizations coincide,
  as both are smooth and transversal to $E_3$ and $E_4$.  
\end{example}

\bigskip
\begin{figure}[h!]
\labellist
\small\hair 2pt
\pinlabel  {$A$} at 140 457
\pinlabel  {$B$} at 240 363
\pinlabel  {$C$} at 364 469
\pinlabel  {$\sigma_A$} at 144 146
\pinlabel  {$\sigma_B$} at 260 196
\pinlabel  {$\sigma_C$} at 394 150
\pinlabel  {$L_A$} at 95 223
\pinlabel  {$L_B$} at 200 260
\pinlabel  {$L_C$} at 414 236
\pinlabel  {$E_1$} at 170 483
\pinlabel  {$E_2$} at 88 480
\pinlabel  {$E_3$} at 200 320
\pinlabel  {$E_4$} at 495 370
\pinlabel  {$W$} at 236 436
\pinlabel  {$\tau_1$} at 136 56
\pinlabel  {$\tau_2$} at 180 84
\pinlabel  {$\tau_3$} at 266 100
\pinlabel  {$\tau_4$} at 320 56
\endlabellist
\centering
\includegraphics[scale=0.50]{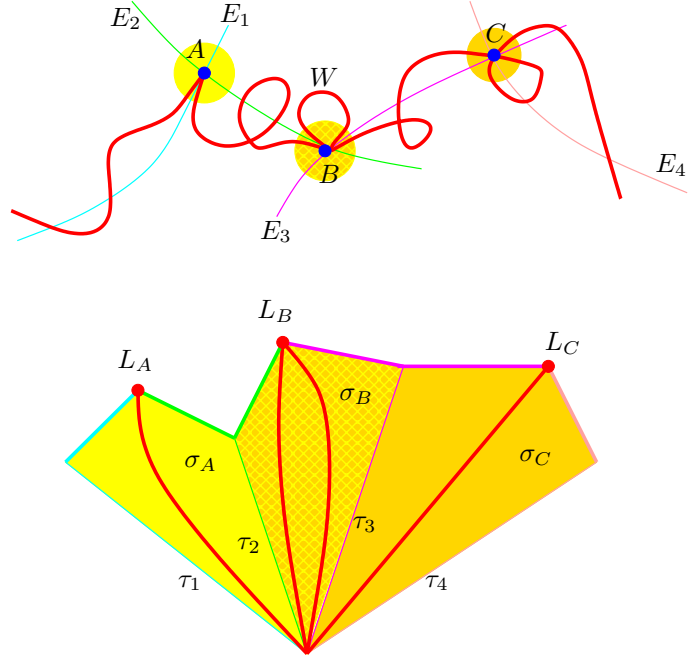}
\caption{Tropicalization of a curve in a toroidal embedding of dimension 2}
\label{fig:Toroitrop}
\end{figure}
\medskip

\medskip
\section{An extension of the definition of tropicalization} \label{S:extdef}

There is yet a more general version of local tropicalization. We are not going to
develop the theory here, but, as we promised in Introduction, we shall describe 
the main idea of the construction. 
\medskip

If $R$ is any commutative ring with unit and $(R^*, .)$ its group of units, then, as 
explained after Definition \ref{congr}, we can define the quotient  $R/R^*$ as a 
multiplicative semigroup. Any nonnegative valuation $v$ on $R$ defines a semigroup 
morphism $R/R^*\to \R_{\geq 0}$ (the argument is the same as the one given in the 
proof of Proposition~\ref{P:toroidalinvariance}). Then, we can speak about 
tropicalization of subsets $\mathcal{W} \subset \mathcal{V}(R)$ not only in the 
presence of semigroup morphisms $\gamma\colon\Gamma\to R$, but also of morphisms 
defined modulo units, that is of semigroup morphisms:
$$\gamma\colon\Gamma\to R/R^*.$$

This yields a functorial construction that generalizes the one described in 
Section \ref{S:defoftrop}. 

For instance, let $\eta$ be a point (not necessarily closed) of a normal algebraic 
variety (over an arbitrary algebraically closed field $K$) or an analytic space 
$X$. From now on, we consider $X$ as a germ at the point $\eta$. Let $D=\cup D_i$ 
be a reduced hypersurface on $X$. \emph{We do not assume that the pair $(X,D)$ is 
toroidal in any sense}. Let $\Gamma$ be the semigroup of effective Cartier 
divisors supported on $D$. The semigroup $\Gamma$ generalizes of the semigroup 
$M^Y_+$ defined in Section~\ref{Toroidal}. Since the semigroup of 
\emph{all} effective Cartier divisors on $X$ is isomorphic to 
$\mathcal{O}_{X,\eta}/\mathcal{O}_{X,\eta}^{*}$, the semigroup $\Gamma$ is 
naturally embedded in $\mathcal{O}_{X,\eta}/\mathcal{O}_{X,\eta}^{*}$. 
This embedding is given by assigning to each Cartier divisor a defining 
function (well-defined modulo a unit). 

Let us show that $\Gamma$ is an affine semigroup. Denote by $G$ the group of all 
Weil divisors supported on $D$, and by $H$ the group of all Cartier divisors 
supported on $D$. The group $G$ is free, thus $H$ is free as a subgroup of $G$. 
All effective Weil $\Q$-divisors form a rational polyhedral cone $\sigma$ of 
maximal dimension in $G_{\Q}$. Thus $\Gamma=H\cap\sigma$ is finitely generated
by Gordan's lemma (\cite[Section 1.2, Proposition 1]{F 93}). We conclude that 
$\Gamma$ is indeed an affine semigroup. Therefore, whenever a hypersurface $D$ on 
a normal germ $X$ is fixed, we can tropicalize any ideal $I$ of the local ring 
$\mathcal{O}_{X,\eta}$, by considering either the positive or the nonnegative 
local tropicalization of the canonical map $\Gamma \to R/R^*$, where 
$R := \mathcal{O}_{X,\eta} / I$. 

In this way, we extend the notions of positive and nonnegative local 
tropicalizations to the case of local semigroup morphisms  $(\Gamma, +)  \to 
(R/R^*, \cdot)$, where $R$ is an arbitrary  local ring.

\medskip
\section{Comparison with the literature}\label{S:literature}

In this section we compare our work with other results in the literature, we 
sketch some possible directions of development and we conclude by stating two 
open problems. 
\medskip

There are already several books and  plenty of papers on tropical geometry. The 
field is developing very fast, and sometimes ideas come to minds of several authors
almost simultaneously. It may well happen that our work is very close to something
already done or something currently being developed by other researchers. In this 
section we would like to explain what we think is new in our approach and 
what is taken from other sources.

The idea of tropicalization, though the term itself is relatively new, appeared 
already in Bergman's paper \cite{Bm2 71} from 1971. Even all three definitions of 
the tropicalization (using valuations, the definition based on initial ideals, 
and the one using $K$-valued points) are present there. Bieri and Groves 
\cite{BG 84} proposed the elegant point of view that the piecewise-linear 
complexes that are now called tropicalizations are \emph{invariants of the 
morphisms $M\to K^*$ from a finitely generated free abelian group $M$ to the
multiplicative group $K^*$ of a field $K$} or, more generally, of the morphisms 
$M\to (R,\cdot)$ to the multiplicative semigroup of a ring $R$. 

As the reader should remember, we defined local tropicalization as a subset of an 
extended affine space, and this subset corresponds to a morphism 
$\Gamma\to (R,\cdot)$ from a semigroup $\Gamma$. This generalizes
Bieri and Groves' point of view, though Payne's work \cite{P 08}, where 
tropicalizations of embeddings into arbitrary toric varieties are studied, was 
also very motivating for us. Extensions of affine spaces (called linear varieties 
in our paper) were already defined in \cite{AMRT 75}.  They are explained also in 
\cite{K 08}, \cite{P 08} and \cite{R 10}; our presentation has no substantial 
differences, but we describe in more detail the topology of those spaces.

As far as we know, tropicalizations of semigroup morphisms $\Gamma\to R$ for
arbitrary  local rings $R$ have not been studied in the literature before.
However, tropicalizations of not only algebraic but also analytic objects were
defined and studied by Touda \cite{T 05}, Rabinoff \cite{R 10}, and 
Gubler \cite{Gub 07}. 

In fact, the main part of our paper (Sections~\ref{S:powerseriesring}, 
\ref{S:lstruct}, and \ref{S:localglobal}) were an extension of Touda's work 
\cite{T 05}, though we started this project without knowing about it. Touda 
studies tropicalizations of ideals in the ring of formal power series over the 
field $\C$ of complex numbers. He works with the definition of local 
tropicalization using weights (analog of the second one used for global 
tropicalization, as recalled in the introduction). He proves then a theorem about 
piecewise-linear structure of the local tropicalization. As an important tool in 
his proofs, he uses the notion of local Gr\"obner fan of an ideal in a formal 
power series ring, as well as its properties proven by Bahloul and Takayama 
in \cite{BT 04}. 

The differences with our approach are the following. We work in the more general 
setting of morphisms $\Gamma\to (R,\cdot)$, in particular, $R$ can be an algebra over
an arbitrary field $K$, and we consider general ring valuations which lead to local
tropicalizations living in an extended affine space, whereas Touda restricts
only to the real part of the local tropicalization. Another new result in our
local finiteness theorem is the statement about dimension of local 
tropicalization. We should also note that some important steps of the 
construction of a tropical basis (e.g., \cite[Proposition~6.3]{T 05}) are left 
without proof in \cite{T 05}.

The main objects of the papers \cite{R 10} and \cite{Gub 07} are rings of series 
with some convergence conditions over fields endowed with a \emph{nontrivial} 
valuation and ideals in these rings. Notice that our local conditions (see 
Definition~\ref{D:loctrop}) imply that if the local ring $R$ has a subfield $K$, 
then any local valuation on $R$ is \emph{trivial} on $K$. Thus we think that our 
work is in a way complementary to \cite{Gub 07} and \cite{R 10}. Another important 
difference is that we could work completely without the theory of affinoid 
algebras that plays a major role in \cite{Gub 07} and \cite{R 10}, and in the 
proof of piecewise-linear structure of the tropicalization in \cite{EKL 06}. The 
local conditions lead also naturally to the question about extensions of 
nonnegative valuations treated in Section~\ref{S:extval}. Despite the fact that 
the literature on the valuation theory is very rich, we are not aware of any 
reference for questions of this kind.

In the proof of the local finiteness theorem we follow well-known ideas. The
use of Gr{\"o}bner basis techniques in describing the structure of tropicalization
is common, perhaps, since the paper \cite{SS 04} of Speyer and Sturmfels. To show 
the existence of universal standard, or Gr{\"o}bner, bases in power series rings 
we apply the method of Sikora \cite{S 04} (as explained by Boldini in \cite{Bol 09}).
Different and more constructive proofs should exist, but we do not know about 
them. It would be interesting to check if Sikora's method is applicable also to 
affinoid algebras. As it is said in \cite[Remark~8.8]{R 10}, a theorem on the 
existence of a universal standard basis for an ideal in an affinoid algebra would 
be an important part of the analytic tropical geometry. The method of a flat 
degeneration of an ideal to its initial ideal is rather standard, see, e.g., 
\cite[Theorem~15.17]{Eis 04}. The fact that an ideal $I$ and its initial ideal 
$\init_w(I)$ locally around $w$ have the same tropicalization has also been 
observed earlier, see \cite[Remark~7.9.2]{R 10}.

As we showed in Section~\ref{S:localglobal}, the usual tropicalization of 
subvarieties of a torus or of a toric variety can be glued from the local
tropicalizations. However, to claim that our local tropicalization generalizes the 
usual one would not be completely honest, since we essentially use properties
of the tropicalization of subvarieties of toric varieties in the proof of
Theorem~\ref{T:main}.

We are not aware of any other treatment of tropicalization of subvarieties of 
toroidal embeddings. A new feature in this case is the absence of the 
``big torus'' in a toroidal embedding. However, our local tropicalization is well 
suited for this situation since it uses only the ``formal torus embedding'' 
$\spec K[[x_1,\dots,x_n]]$. Once the theory of tropicalization of ideals of 
the rings $K[[\Gamma]]$ has been developed, the construction of tropicalization 
of subvarieties of toroidal embeddings is very natural and straightforward.

\medskip

Let us describe now some possible interactions of our work with developing parts
of mathematics. 

One should be able to prove in the toroidal 
setting an analog of Payne's main theorem from \cite{P 08} 
relating tropicalizations and analytifications in the Berkovich sense.
This would allow to make a bridge with Thuillier's work \cite{Thu 07} 
on the analytification of toroidal embeddings.

Our final general definition of tropicalization associated to a 
morphism of semigroups $\Gamma \rightarrow R/ R^*$ should be useful as a starting
point for 
tropicalizing log-structures. This seems to be one of the current directions of
development of tropical geometry, as indicated by Gross in his book \cite{G 11}
and in his talk \cite{Gtalk 11}. Indeed, a log-scheme is a scheme $X$ equipped
with a morphism of sheaves of (multiplicative) semigroups $\alpha_X\colon 
\mathcal{M}_X \rightarrow \mathcal{O}_X$, such that $\alpha_X$ realizes an 
isomorphism between $\alpha_X^{-1}(\mathcal{O}_X^*)$ and $\mathcal{O}_X^*$.
Let $\overline{\mathcal{M}}_X : = \mathcal{M}_X / \alpha_X^{-1}(\mathcal{O}_X^*)$.
Quoting from \cite[Page 101]{G 11} : ``\emph{The sheaf of monoids 
$\overline{\mathcal{M}}_X$, written additively [...] should be viewed as 
containing combinatorial information about the log structure}''.
Note that $\alpha_X$ induces a canonical morphism of sheaves of semigroups:
  $$ \overline{\mathcal{M}}_X \rightarrow \mathcal{O}_X / \mathcal{O}_X^*.$$
That is, we are ready for gluing our affine definitions of tropicalizations! 
 
The fact that we have isolated the category of semigroups as part of the structure 
allowing tropicalization should allow us to also make connections with algebraic 
geometry over the field with one element, as described for instance by Connes 
and Consani in \cite{CC 09}. As explained in Chapter 3 of that paper, the 
category of semigroups and morphisms of semigroups is an essential component 
of it. 
 
Another field which has already very important connections with tropical geometry
is the theory of Berkovich analytic spaces. As explained by Berkovich \cite{B 11},
the category of semigroups also plays an important role there. As the 
title of Berkovich's talk indicates, this should be seen as part of a project of
relating analytic geometry to geometry over the field with one element.

\medskip
We finish with two problems about local tropicalization.

\begin{problem}
Let $\gamma\colon (\Gamma, +) \to (R/ R^*, \cdot)$ be an arbitrary local morphism, 
where $\Gamma$ is a pointed affine semigroup and $R$ is a complete local ring. 
We do not suppose that $\gamma$ is the natural morphism of $\Gamma$ to a quotient 
of a power series ring $K[[\Gamma]]$ over a field $K$, as in 
Section \ref{S:lstruct}. Does the local tropicalization $\ptrop(\gamma)$ have 
piecewise-linear structure in such a general case? This question is interesting 
both in the case when $R$ contains a field or when it does not.
\end{problem}

\begin{problem}
Find a proof of Theorem~\ref{T:main} that is independent of the standard theory 
of tropicalization of subvarieties of toric varieties.
\end{problem}

\bibliographystyle{amsalpha}

\end{document}